\documentclass[fleqn,11pt,oneside]{article}

\usepackage{float}
\usepackage{hyperref}
\usepackage[final]{graphicx}
\usepackage{bbm}  % blackboard math
\usepackage{bm}  % bold math
\usepackage{amsmath,amsthm,amssymb,amscd}
\usepackage{caption,subcaption}
\usepackage{booktabs,multirow}
\usepackage{setspace} 
\usepackage[top=1.1in, bottom=1.1in, left=1.6in, right=1.1in]{geometry}
\usepackage[inline,final]{showlabels}
\usepackage{dashbox}
\usepackage{microtype}
\usepackage[medium]{titlesec}
\usepackage{srcltx}
\usepackage{verbatim}

\newtheorem{theorem}{Theorem}[section]
\newtheorem{lemma}[theorem]{Lemma}
\newtheorem{corollary}[theorem]{Corollary}
\newtheorem{definition1}{Definition}[section]
\newtheorem{observe}{Observation}[section]
\newtheorem{remark1}[observe]{Remark}
\newtheorem{example1}{Example}[section]
\newtheorem{aside1}[observe]{Aside}

\newenvironment{remark}{\begin{remark1} \rm}{\end{remark1}}

\def\qed{\hfill$\blacksquare$\\} \renewenvironment{proof}{\noindent {\bf 
Proof.}}{\qed}

\renewcommand{\tilde}{\widetilde}

\newif\ifshowboxes \showboxestrue
\newcommand\fboxit[1]{ \ifshowboxes \par \noindent%
  \fbox{\begin{minipage}{\textwidth}#1\vfill\end{minipage}} \else #1 \fi}

\renewcommand\Re{\operatorname{Re}}
\renewcommand\Im{\operatorname{Im}}

\newcommand{\floor}[1]{\ensuremath{ {\lfloor #1 \rfloor} }}

\newcommand{\abs}[1]{\ensuremath{ {\lvert #1 \rvert} }}

\DeclareFontFamily{U}{mathx}{\hyphenchar\font45}
\DeclareFontShape{U}{mathx}{m}{n}{
      <5> <6> <7> <8> <9> <10>
      <10.95> <12> <14.4> <17.28> <20.74> <24.88>
      mathx10
      }{}
\DeclareSymbolFont{mathx}{U}{mathx}{m}{n}
\DeclareFontSubstitution{U}{mathx}{m}{n}
\DeclareMathAccent{\widecheck}{0}{mathx}{"71}
\DeclareMathAccent{\wideparen}{0}{mathx}{"75}

\def\R{\mathbbm{R}}
\def\1{\mathbbm{1}}

\providecommand{\e}[1]{\ensuremath{\times 10^{#1}}}

\showboxesfalse

\begin{document}

\begin{center}
   \begin{minipage}[t]{6.0in}

In this manuscript, we develop an efficient algorithm to evaluate the
azimuthal Fourier components of the Green's function for the Helmholtz
equation in cylindrical coordinates. A computationally efficient algorithm
for this modal Green's function is essential for solvers for electromagnetic
scattering from bodies of revolution (e.g., radar cross sections, antennas).
Current algorithms to evaluate this modal Green's function become
computationally intractable when the source and target are close or when the
wavenumber is large.  Furthermore, most state of the art methods
cannot be easily parallelized.  In this manuscript, we present an algorithm
for evaluating the modal Green's function that has performance independent
of both source-to-target proximity and wavenumber, and whose cost grows as
$O(m)$, where $m$ is the Fourier mode. Furthermore, our algorithm is
embarrassingly parallelizable.

\thispagestyle{empty}

  \vspace{ -100.0in}

  \end{minipage}
\end{center}

\vspace{ 2.60in}
\vspace{ 0.50in}

\begin{center}
  \begin{minipage}[t]{4.4in}
    \begin{center}

\textbf{On the Efficient Evaluation of the Azimuthal Fourier Components of the
Green's Function for Helmholtz's Equation in Cylindrical Coordinates} \\

  \vspace{ 0.50in}

James Garritano$\mbox{}^{\ast \, \sharp \, \otimes}$, 
Yuval Kluger$\mbox{}^{\ast \, \sharp}$, \\
Vladimir Rokhlin$\mbox{}^{\ast \, \ddagger \, \oplus}$,  
Kirill Serkh$\mbox{}^{\dagger\, \diamond}$ \\
              \today

    \end{center}
  \vspace{ -100.0in}
  \end{minipage}
\end{center}

\vspace{ 2.00in}

\vfill

\noindent 
$\mbox{}^{\otimes}$  This author's work  was supported in part by NIH
F30HG011193 and by US NIH MSTP Training Grant T32GM007205.
\\
\noindent 
$\mbox{}^{\oplus}$  This author's work was supported in part by ONR
N00014-18-1-2353 and NSF DMS-1952751. \\
\noindent 
$\mbox{}^{\diamond}$  This author's work  was supported in part by the NSERC
Discovery Grants RGPIN-2020-06022 and DGECR-2020-00356.

\vspace{4mm}

\noindent
$\mbox{}^{\ast}$ Program in Applied Mathematics, Yale University, New
Haven, CT 06511 \\
\noindent
$\mbox{}^{\sharp}$ Yale School of Medicine, New Haven, CT 06511 \\
\noindent
$\mbox{}^{\ddagger}$ Dept.~of Mathematics, Yale University, New Haven, CT 06511 \\
\noindent
$\mbox{}^{\dagger}$ Dept.~of Math. and Computer Science, University of Toronto,
Toronto, ON M5S 2E4

\vspace{2mm}

\vfill
\eject

\tableofcontents

\section{Introduction}
This manuscript will detail how to efficiently compute the azimuthal Fourier
components of the Green's function (i.e., the modal Green's functions) for the
Helmholtz equation in three dimensions,
known to be
%
%%%\fboxit{
\begin{align}
G_m(\bm{x},\bm{x}')=G_m(|\bm{x}-\bm{x}'|)
 = \frac{1}{2\pi}
 \int_{-\pi}^{\pi} \frac{1}{4\pi} \frac{e^{ik| \bm{x}-\bm{x}'|}}{|\bm{x}-\bm{x}'|}
    e^{-im\theta} d\theta ,
\end{align}
%%%}
%
where $\bm{x}, \bm{x'} \in \mathbb{R}^3$, $k$ is the wavenumber, and $m$ is the $m$th
azimuthal Fourier mode.
Rewriting this equation in cylindrical coordinates, with $\bm{x} = (r,\theta,z)$
and $\bm{x'} = (r',\theta ', z')$, and letting $\phi = \theta - \theta'$,
the formula for the $m$th Fourier coefficient becomes 
\begin{align}
G_m(\bm{x},\bm{x}')= \frac{1}{4\pi^2 R_0} \int_{0}^{\pi}
\frac{ e^{-i\kappa\sqrt{1-\alpha \cos \phi}} }
{\sqrt{1 - \alpha\cos\phi} }
\cos{m\phi} d\phi ,
\end{align} 
where $\kappa=kR_0$, $\alpha = 2rr'/R_0^2$, and $R_0^2=r^2+r'^2+(z-z')^2$.

This integral has two features which make numeric integration difficult: the
integrand is oscillatory, and it is near-singular when 
the distance between $\bm{x}$ and
$\bm{x'}$ is small (i.e., as $\alpha$ approaches 1).
However, the integrand vanishes for sufficiently large 
imaginary values of $\kappa$, suggesting that Cauchy's theorem can be used to
construct a contour on which all the oscillations occur where the 
integrand is negligible.

When devising an appropriate contour, it is helpful to consider three cases:
1) when $\kappa$ is zero and $m\geq 0$, 2) when $\kappa$ is arbitrary
and $m$ is small,
and 3) when both $\kappa$ and $m$ are large.

Determining the appropriate contour when $\kappa=0$ and $m\geq 0$ 
(when the Helmholtz equation becomes the Laplace equation) is
trivial, because on any vertical contour (into quadrant~IV of the complex plane) 
the integrand 
monotonically decays. 
When $\kappa>0$ and $m=1$,
the appropriate contours were solved by Gustafsson \cite{gustafsson} via the method
of steepest descent. However, Gustafsson did not analyze cases where 
both $\kappa>0$ and $m>1$. In fact, when
both $m$ and $\kappa$ are large, it turns out that no contour exists 
on which the entire integrand monotonically decays.

We develop on Gustafsson's work by integrating along the contour on which the
spherical wave component,
\begin{align}
\frac{e^{-i\kappa \sqrt{1 - \alpha \cos \phi} }}{\sqrt{1 - \alpha \cos \phi}} ,
\end{align}
monotonically decays. However, the part of the integrand dependent on azimuthal
frequency,
$\cos(m\phi)$, behaves poorly and grows on this contour.
To circumvent this behavior, we replace the term $\cos(m\phi)$ with 
a rational function approximation which does
not grow in the complex plane. The growth of $\cos(m\phi)$ along the contour
is subsumed in a collection of residues which must be added to the 
resulting integral.
\subsection{The Modal Green's Functions for the Helmholtz Equation} \label{sec:mgf}
The Green's function for the Helmholtz equation in three dimensions satisfies the
equation
\begin{align}
(\nabla^2+k^2)G_k(\bm{x},\bm{x}')=\delta(\bm{x}-\bm{x}') ,
\end{align}
where k is the wave number and $\bm{x}$, $\bm{x'} \in \mathbb{R}^3$.
The solution is an outgoing spherical wave, given by the formula 
\begin{align}
G_k(\bm{x},\bm{x}')=G_k(|\bm{x}-\bm{x}'|)
 = \frac{1}{4\pi} \frac{e^{ik| \bm{x}-\bm{x}'|}}{|\bm{x}-\bm{x}'|}.
\end{align}
We consider a problem with rotational symmetry (i.e., a body of revolution). 
Switching 
to cylindrical coordinates and expanding $G_k$ in Fourier series, we have 
\begin{align}
G_k(\bm{x},\bm{x}') = \sum_{m=-\infty}^{\infty} 
  G_{m,k}(r,z,r',z')e^{im(\theta - \theta')}, 
\end{align}
where $\bm{x} = (r,\theta,z)$, $\bm{x'} = (r',\theta', z')$.
Let $\phi=\theta-\theta'$ denote the difference in azimuthal angles.
The formula for the $m$th coefficient is 
\begin{align}
G_{m,k}(r,z,r',z')= \frac{1}{2\pi}\int_{-\pi}^{\pi} 
G_k(r,z,r',z',\phi)e^{-im\phi}d\phi.
\end{align}
We adopt notation consistent with the literature 
(see, for example, \cite{cohl,epstein,young}) and omit the subscript $k$
denoting the wavenumber. Expanding the representation for the $m$th Fourier 
coefficient, we have 
\begin{align} \label{eq:gmcyl}
\hspace{-3em}
G_m(r,z,r',z')=\frac{1}{2\pi}\int_{-\pi}^{\pi}
\frac{e^{ik\sqrt{r^2+r'^2-2rr'\cos\phi+(z-z')^2}}}
{4\pi\sqrt{r^2+r'^2-2rr'\cos\phi+(z-z')^2}}
e^{-im\phi}d\phi.
\end{align}
We then introduce the parameter $R_0$, given by
\begin{align}
R_0=\sqrt{r^2+r'^2+(z-z')^2},
\end{align}
which we use to rewrite (\ref{eq:gmcyl}), 
by defining $\kappa = kR_0$ and
$\alpha = 2rr'/R_0^2$, obtaining 
\begin{align} \label{eq:mgfexp}
G_m(\bm{x},\bm{x}')=\frac{1}{8\pi^2R_0}\int_{-\pi}^{\pi} 
\frac{e^{i\kappa\sqrt{1 - \alpha\cos\phi}}}
{\sqrt{1 - \alpha\cos\phi}} e^{-im\phi}d\phi.
\end{align}
Any numerical scheme for evaluating $G_m$ must depend on four parameters: $\kappa$,
$\alpha$, $R_0$, and $m$. Notably, $\alpha$ is bounded by $0 \leq \alpha < 1$, and
determines the growth of the integrand near $\phi =0$.
In Section \ref{se:removesingularity}, we will 
introduce the 
parameters $\beta_-$ and $\beta_+$, defined to be  
\begin{align} \label{eq:betaBoth}
\beta_- = \sqrt{1/\alpha -1} , \qquad \beta_+ = \sqrt{1/\alpha +1}.
\end{align}
We also introduce the parameters $\Delta$ and $\rho_0$, defined as
\begin{align}
\Delta &= \sqrt{(r-r')^2 + (z-z')^2},\\
\rho_0 &= 2rr'.
\end{align}
Note that $\Delta$ is the minimum distance between the source and the target, $R_0$
is the maximum distance between the source and
the target, and that \fboxit{$\Delta^2=R_0^2-\rho_0$}.
Lastly, we observe that $\beta_-$ and $\beta_+$ also given by the formulae,
\begin{align} \label{eq:beta_-eff}
\beta_- = \frac{\Delta}{\rho_0},\qquad \beta_+ = \sqrt{\frac{R_0^2+\rho_0}{\rho_0}}.
\end{align}
We note that numerically computing $\beta_-$ from $\alpha$ using 
(\ref{eq:betaBoth}) will result in cancellation error when $\alpha \approx 1$,
so it is usually better to compute $\beta_-$ directly from
formula (\ref{eq:beta_-eff}).

A representative sample of the literature related to the evaluation
of the modal Green's functions can be found in 
\cite{abdelmageed,andreasen,cheng,conway,epstein,gedney,helsing2,lai,
liu,matviyenko,vaessen}. 
\subsubsection{Number of Fourier Coefficients Needed}
\label{se:numNeed}
Matviyenko in \cite{matviyenko} derived an upper bound, $r_+$, such that
all Fourier modes $m >r_+$ geometrically decay as $m$ increases, 
with $r_+$ given by
%
% Note: Matviyenko's notation is quite different from this manuscript,
% and this formula isn't obvious. 
% Jim has a notebook titled "matviyenko_r_p_simplfied.nb" which demonstrates
% that the simplification is correct - it does not claim Matviyenko is
% correct, merely that our simplification of his formula is correct. 
%
\fboxit{
\begin{align} \label{eq:numCoeffNeed}
r_+ = \frac{\kappa}{\sqrt{2}} \sqrt{1 + \sqrt{1 - \alpha^2}},
\end{align}
}
where $\alpha = \rho_0/R_0^2$ and $\kappa=kR_0$ (see \cite{matviyenko}, 
formulae (37) and (38)).
When $\alpha\approx 1$, formula (\ref{eq:numCoeffNeed}) simplifies to
\begin{align} \label{eq:r_+approx}
r_+ \approx \frac{\kappa}{\sqrt{2}}.
\end{align}
Using Matviyenko's formula for the decay of the modal
Green's functions  (see \cite{matviyenko}, formula (40)),  
it can be shown that the magnitude of any Fourier coefficient $m > r_+$ is
bounded by
\begin{align} \label{eq:decay}
|G_m| < |G_{\lfloor r_+ \rfloor }| \Bigg(\frac{\sqrt{1 - \sqrt{1-\alpha^2}}}
{\sqrt{1 + \sqrt{1 - \alpha^2 }}} \Bigg)^{m - \lfloor r_+ \rfloor }.
\end{align}
%
%(see \cite{matviyenko}, formula (40)).
%
Substituting $\alpha = 1/(\beta_-^2+1)$ into (\ref{eq:decay}), this bound can
be simplified to 
\fboxit{
\begin{align} \label{eq:decay2}
\hspace{-3em}
|G_m| < |G_{\lfloor r_+ \rfloor }| \Bigg(1- \frac{ \beta_-\sqrt{\beta_-^2 +2}}
{1 + \beta_-^2} \Bigg)^{\frac{m-\lfloor r_+ \rfloor }{2}} 
\Bigg(1+ \frac{\beta_- \sqrt{\beta_-^2 +2}}
{1 + \beta_-^2} \Bigg)^{-\frac{m-\lfloor r_+ \rfloor }{2}},
\end{align}
}
where $\beta_-$ is the scaled source-to-target distance.
When $\beta_-$ is small, $1+\beta_-^2 \approx 1$, and (\ref{eq:decay2}) can
be approximated as  
%
% Note: this formula is not obvious. See matviyenko_decay_coeff.nb.
%
\fboxit{
\begin{align} \label{eq:decay3}
\hspace{-3em}
|G_m| \lesssim |G_{\lfloor r_+ \rfloor }| \Bigg(
\frac{1- \beta_- \sqrt{2}}{1+ \beta_- \sqrt{2}} \Bigg)
^{\frac{m - \lfloor r_+ \rfloor }{2}} \approx |G_{\lfloor r_+ \rfloor }|
(1 - \beta_- 2\sqrt{2})
^{\frac{m - \lfloor r_+ \rfloor }{2}},
\end{align}
}
where we have replaced the exponentiated term with its truncated Taylor expansion
in $\beta_-$.
Formula (\ref{eq:decay3}) can be used to determine the Fourier mode $M$
such that for $m > M$, $|G_m| < \epsilon$, where $M$ is given by 
\begin{align} \label{eq:Mbig}
M \approx \frac{2\log(\epsilon) -2\log{(|G_{\lfloor r_+ \rfloor }|)}}
{\log(1 - 2\sqrt{2} \beta_-)}  + \lfloor r_+ \rfloor.
\end{align}
By substituting (\ref{eq:numCoeffNeed}) into (\ref{eq:Mbig}, we can characterize
the order of $M$ as a function of $\beta_-$ and $\kappa$ when the source
and target are close (i.e., $\alpha \gtrsim 0.99$ or
equivalently $\beta_- \lesssim 10^{-2}$) as
\fboxit{
\begin{align} \label{eq:numNeed}
\begin{split}
\hspace{-4em}
M \approx & \frac{2\log(\epsilon) -2\log{(|G_{\lfloor r_+ \rfloor }|)}}
{\log(1 - 2\sqrt{2} \beta_-)}  +  \frac{\kappa}{\sqrt{2}} \\
=& \bigg(1 -\frac{1}{ \sqrt{2} \beta_- } \bigg) 
\big( \log(\epsilon) -\log{(|G_{\lfloor r_+ \rfloor }|} \big)+ \,
 \frac{\kappa}{\sqrt{2}}  + O(\beta_-^2)\\
=& O\Big(\frac{1}{\beta_-} + \kappa\Big)
,
\end{split}
\end{align}
}
where we have replaced the denominator of (\ref{eq:Mbig}) with its Taylor
expansion in $\beta_-$.
Lastly, it is often useful to write (\ref{eq:numNeed}) in terms
of the radius of the body of revolution and the minimum source-to-target
distance, $\Delta$.  
We rewrite (\ref{eq:numNeed}) as
\begin{align}\label{eq:numNeedPre2}
M = O\Big(\frac{\rho_0}{\Delta} + k\sqrt{\rho_0 + \Delta^2}\Big),
\end{align}
where we have substituted $\kappa = kR_0$ and $\beta_- = \rho_0/\Delta$. 
Because $\Delta^2 \ll \rho_0$,
\begin{align}\label{eq:numNeedDeltaRho}
M = O\Big(\frac{\rho_0}{\Delta} + k\sqrt{\rho_0}\Big).
\end{align}
Recall that $\rho_0 = 2rr'$.
When $\Delta$ is very small, $2rr'\approx 2r^2$, meaning that 
formula (\ref{eq:numNeedDeltaRho}) becomes
\begin{align}\label{eq:numNeedR}
M = O\Big(\frac{2r^2}{\Delta} + kr\sqrt{2}\Big) = 
O\Big(\frac{r^2}{\Delta} + kr\Big).
\end{align}
\subsubsection{Informal Description of the Spectra of the Green's Functions}

The Green's function for the Helmholtz equation in cylindrical coordinates,
\begin{align}
G(\bm{x},\bm{x}')= 
\frac{e^{i\kappa\sqrt{1 - \alpha\cos\phi}}}
{\sqrt{1 - \alpha\cos\phi}},
\end{align}
can be viewed as 
the product of the Green's function for the Laplace equation,
\begin{align}
G^L(\bm{x},\bm{x}')= \frac{1}{\sqrt{1-\alpha \cos \phi}},
\end{align}
with the band-limited term $\exp(i \kappa R(\theta))$, where
$R = \sqrt{1 -\alpha \cos \phi}$, 
$\bm{x} = (r,\theta,z)$, $\bm{x'} = (r',\theta', z')$, $\phi = \theta - \theta'$,
$\alpha = 2rr'/R_0^2$, and $\kappa=kR_0$, where $G$ and $G_L$ are understood
to be a functions of $\kappa$, $\alpha$, and $R_0$.

Recall that the Fourier transform of the product of two functions
is the convolution of their Fourier transforms.
Hence, the modal Green's function can be thought of as the convolution
of the Fourier coefficients of $G^L(\bm{x},\bm{x}')$ and the Fourier
coefficients of $\exp(i \kappa R(\theta))$.
Using this observation, combined with Matviyenko's formulae for the 
cut-off frequency (formula (\ref{eq:numCoeffNeed}))
and the rate of the decay of Fourier coefficients (formula (\ref{eq:decay})),
we now provide a rough description of the spectra of the Green's functions.

First, we describe the spectra of the Green's functions of the Laplace
equation. 
The Fourier coefficients of the Green's functions of the Laplace
equation decay as a function of $\alpha$, with the rate given by
formula (\ref{eq:decay}) (see
Figure \ref{fig:laplace}).
\begin{figure}[H]
\centering
\includegraphics[width=0.95\textwidth]{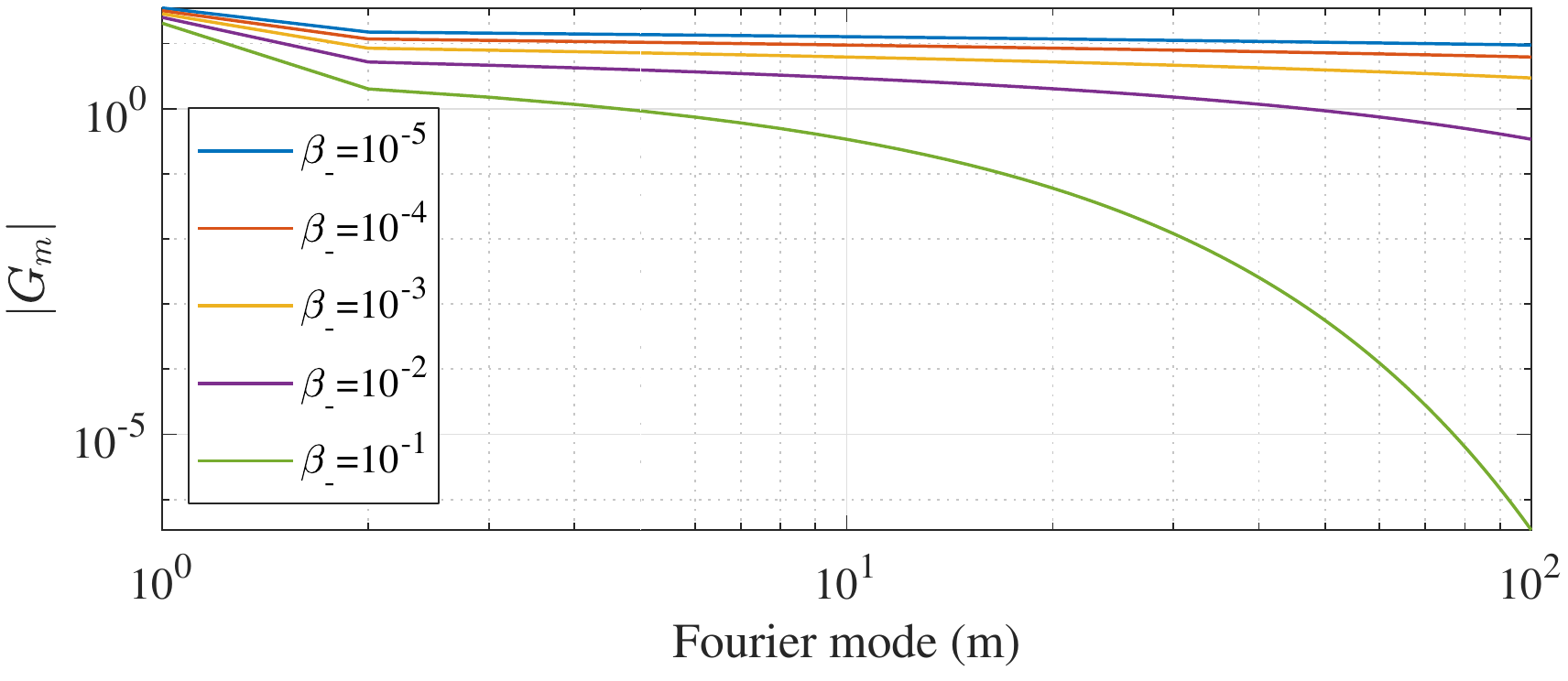}
\caption[Spectra of the Green's functions for the Laplace equation]{
\bf{ Spectra of the Green's functions for the Laplace equation with
varying source-to-target distance.}
}
\label{fig:laplace}
\end{figure}
Now that we have characterized the spectra of the Green's functions
for the Laplace equation, we are ready to describe the spectra of
the Green's functions for the Helmholtz equation when $\kappa >0$. 
Recall that formula (\ref{eq:numCoeffNeed}) provides a cut-off frequency 
$r_+$ such that, for $m > r_+$, $G_m$ geometrically decays;
recall also that $r_+$ scales with $\kappa$
(see Figure \ref{fig:kappaclose2}). For all $m > r_+$, the rate
of the decay of the Fourier coefficients is determined by the source-to-target
distance by formula (\ref{eq:decay}) (see Figure \ref{fig:spectra}).
\begin{figure}[H]
\centering
\includegraphics[width=0.95\textwidth]{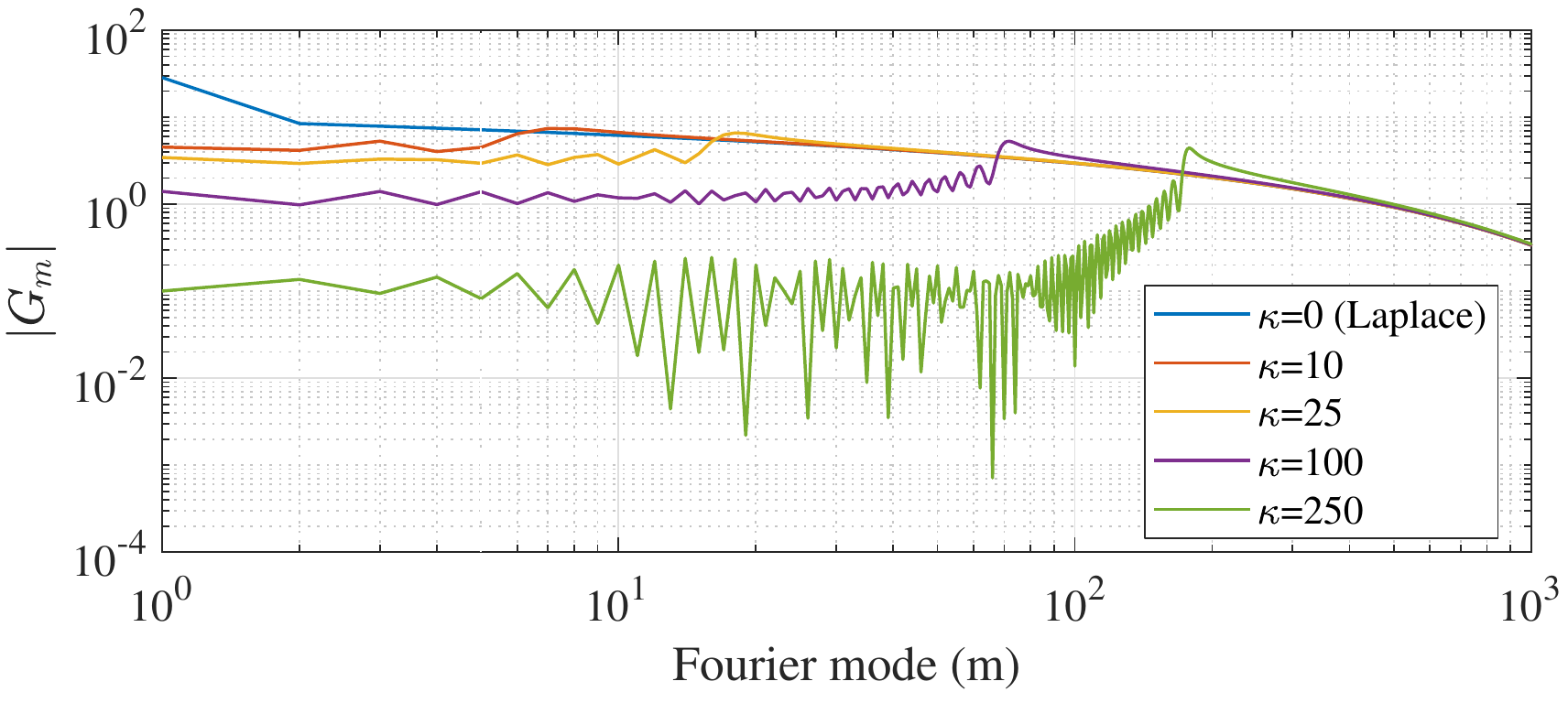}
\caption[Spectra of the Green's functions for the Helmholtz equation with varying 
scaled wavenumber $\kappa$]
{ \bf{Spectra of the Green's functions for the Helmholtz equation
when the source and target are close
with $\beta_-=10^{-3}$ and varying $\kappa$}. 
%The approximation for the index $r_+$ after
%which the Fourier modes geometrically decay
%(given by formula (\ref{eq:r_+approx})) is represented by a vertical dashed line.
}
\label{fig:kappaclose2}
\end{figure}
\begin{figure}[H]
\centering
\includegraphics[width=0.95\textwidth]{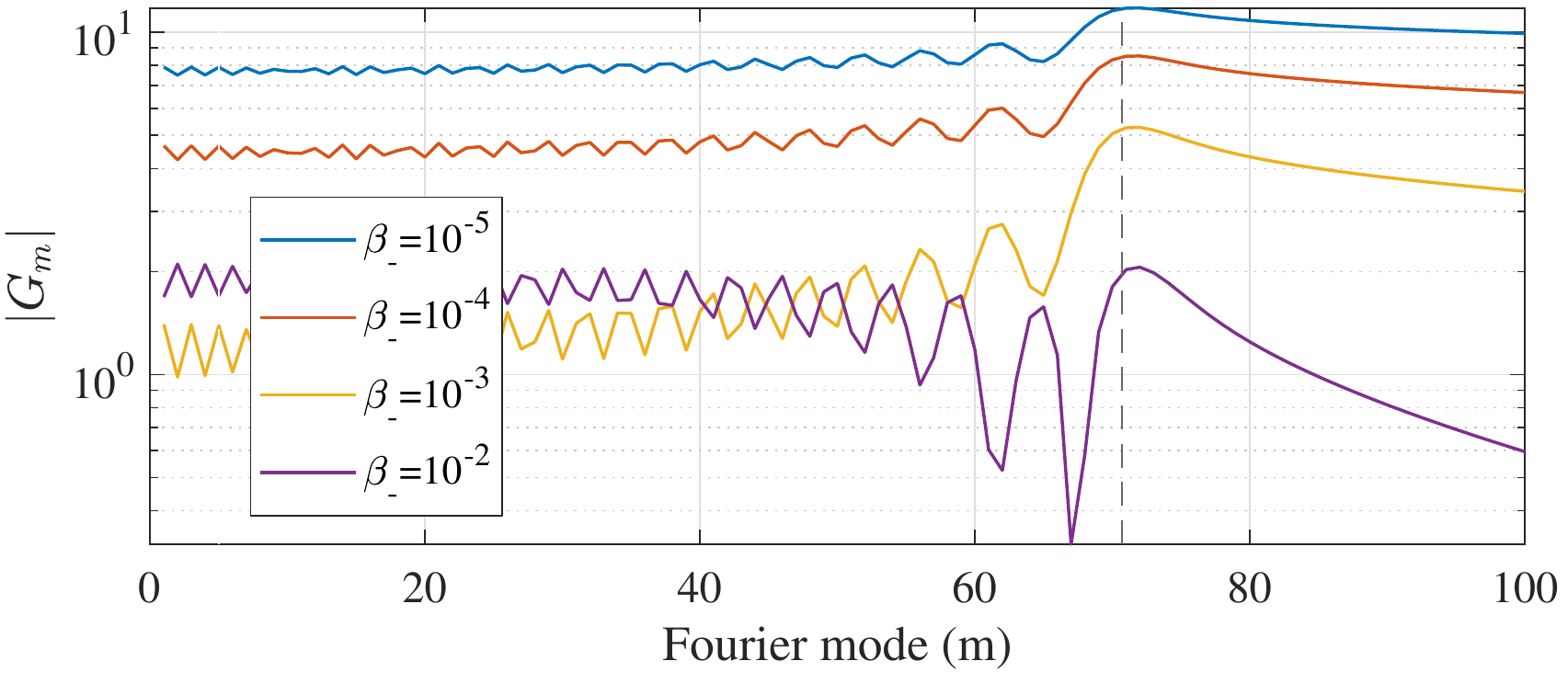}
\caption[Spectra of the Green's functions for the Helmholtz equation
with varying source-to-target distance $\beta_-$]
{ {\bf Spectra of the Green's functions for the Helmholtz equation 
when the source and target are close
for $\kappa=100$ with varying $\beta_-$.}
The approximation for the index $r_+$ after
which the Fourier modes geometrically decay
(given by formula (\ref{eq:r_+approx})) is represented by a vertical dashed line.
}
\label{fig:spectra}
\end{figure}
\subsection{Review of the Literature}
Recall from Section \ref{sec:mgf} that the modal Green's function is
a function of three parameters: 
$\kappa$, $m$, and $\alpha$. 
We divide the literature on fast algorithms for evaluating the modal
Green's function into two categories:
those that evaluate the general case of any combination of input parameters
and those that evaluate special cases of input parameters (e.g., when the
source and target
are well-separated, when $m=1$, etc.).

Almost all modern fast general-case algorithms are based on the application of the
Fast Fourier Transform (FFT) (see, for example, \cite{epstein,gedney,helsing,
helsing2,lai,vaessen,wang,young}).
In contrast, the special-case algorithms have a diverse set of methodologies
which cannot easily be summarized. Because this manuscript's topic is a
general-case algorithm which works for all input parameters, we do not review
the literature of special-case algorithms, with the exception of Gustfasson's
contour integration technique \cite{gustafsson}, which we develop on 
extensively in this manuscript.
\subsubsection{Introduction to FFT-Based Kernel Splitting}
Almost all published fast general-case algorithms 
for evaluating the modal Green's functions
(i.e., those that 
take as input an arbitrary source-to-target distance and arbitrary Fourier mode) 
use the Fast Fourier Transform (FFT). 
Because computing the Fourier coefficients of a near-singular function is 
not efficient, and because the Green's function becomes near-singular
for $\alpha \approx 1$ (i.e., when the source and target are close), 
all modern implementations of FFT-based methods
employ kernel splitting, the technique of splitting the integrand
into a near-singular portion and a non-singular portion, then computing
each portion's coefficient's separately.
For the non-singular portion, the FFT is often fast
and efficient.
For the near-singular portion, some other technique, usually a purpose-made
recurrence 
relation, is used to evaluate the singular integral.

Two kernel splittings are used in the literature of
FFT-based evaluations of the modal Green's function: the splitting of 
Gedney and Mittra 
\cite{gedney} (see, for example, \cite{gedney,vaessen,wang})
and the splitting of Helsing \cite{helsing} (see, for example, \cite{epstein,
helsing,lai,young}).

Gedney and Mittra in \cite{gedney} isolate
the near-singular portion of the integral by adding and
subtracting a $1/R$ term, resulting in the splitting
\begin{align}\label{eq:gedneysplit}
\begin{split}
G_m &= \int \frac{ e^{-i\kappa \sqrt{1-\alpha\cos\phi}} } {\sqrt{1-\alpha\cos\phi}}
 +\frac{1}{\sqrt{1 - \alpha \cos \phi}} - \frac{1}{\sqrt{1-\alpha \cos \phi}} d\phi\\
 &= \underbrace{\int \frac{1}{\sqrt{1 - \alpha \cos\phi}}  d\phi}
          _{\small \text{ near-singular}}
  +\underbrace{\int \frac{ e^{i\kappa \sqrt{1-\alpha\cos\phi}}-1}
         {\sqrt{1-\alpha\cos\phi}}d\phi}_
         {\small \text{non-singular, not very smooth}}.
\end{split}
\end{align}
With this splitting, as $\alpha \to 1$, the near-singular term's growth is
unbounded, while the non-singular term's growth is bounded.
Although this splitting isolates the near-singular portion of the integral, 
the resulting
non-singular integral is not very smooth as $\alpha \to 1$. Hence, directly applying
the FFT to the non-singular
portion remains inefficient when $\alpha \approx 1$, meaning 
that efficient evaluation of the non-singular
term in (\ref{eq:gedneysplit}) requires additional manipulation of the 
integrand. 

Helsing in \cite{helsing} split the integral via application of
Euler's formula, resulting in
\begin{align}\label{eq:helsingsplit}
\begin{split}
%\hspace{-4.5em}
G_m &= \int \frac{ e^{-i\kappa \sqrt{1-\alpha\cos\phi}} }
{\sqrt{1-\alpha\cos\phi}}d\phi\\
     &= \underbrace{ \int \frac{ \cos(  \sqrt{1-\alpha\cos\phi}  )}
     { \sqrt{1-\alpha\cos\phi}  } d\phi}_
     { \small \text{ near-singular}}
     + \underbrace{\int  \frac{i \sin(  \sqrt{1-\alpha\cos\phi}  )}
     { \sqrt{1-\alpha\cos\phi}  } d\phi}_
     { \small \text{  non-singular, smooth}}.
\end{split}
\end{align}
In contrast to (\ref{eq:gedneysplit}), the non-singular portion of 
formula (\ref{eq:helsingsplit}) is smooth as $\alpha \to 1$. Consequently, the Fourier
coefficients of the non-singular portion of (\ref{eq:helsingsplit}) 
can be efficiently evaluated by the FFT.

Careful analysis of recent implementations of both kernel splitting 
techniques shows
that the splitting of Helsing is sufficient to efficiently
utilize the FFT to compute the modal Green's function
(see Section \ref{se:epstein}), while the 
splitting of Gedney and Mittra requires further techniques to evaluate
the non-singular integral.
In fact, the fastest algorithm utilizing the splitting of Gedney and Mittra,
published by Vaessen et al. \cite{vaessen},
split the non-singular integral again into a part which is smooth as $\alpha\to $
and a part which is not. It turns out that this final splitting is
essentially equivalent to Helsing's.
Because the splitting of Helsing is utilized in the fastest algorithm for 
both techniques, we first summarize the splitting of Helsing by examining a
recent fast implementation, then conclude our review by summarizing
the fastest implementation of Gedney and Mittra's splitting. Lastly, we demonstrate
that it is algorithmically equivalent to the fastest implementation of Helsing's
splitting.

%We analyze each splitting via summarizing a recent implementation.
%
%
%
%
\subsubsection{Method of Epstein et al., a Recent Implementation of Helsing's Kernel Splitting}
\label{se:epstein}
In Epstein et al. \cite{epstein}, the modal Green's functions are computed 
using a Fast Fourier Transform (FFT)-based method, with the 
kernel splitting of Helsing \cite{helsing}.
In the following, the definitions for $m$, $\kappa$, $\alpha$, and $R_0$ are
identical to those used in Section \ref{sec:mgf}.

The authors divide the evaluation of the modal Green's function for Fourier modes 
$-M, -M+1,\dots,M-1,M$
into two cases: 
one where the source and target are well-separated
($0 \leq \alpha \leq 1/1.005$), 
and one where the source and target are close ($1/1.005 \leq \alpha < 1$).

In the former case, the integrand is relatively smooth, and the modal Green's
functions are computed using an $L$-point FFT, obtaining near double precision
accuracy when $L \geq 4|\kappa|$.
When $|\kappa| \leq 256$, a
1024 point FFT is used. The parameter $L$ must be chosen such that $L > 2M$,
but in practical situations $2M$ is usually smaller than the $L$ chosen 
via this heuristic. 

For the near-singular case, $\alpha \approx 1$, the authors follow 
\cite{helsing} by first rewriting (\ref{eq:mgfexp}) as
\begin{align} \label{eq:mgfhelsing}
\hspace{-3em}
G_m(\bm{x},\bm{x}') = \frac{1}{2\pi} \int_{-\pi}^\pi
\frac{\cos (\kappa \sqrt{1 - \alpha \cos \phi})
+ i \sin (\kappa\sqrt{1 - \alpha \cos \phi})} 
{4\pi R_0 \sqrt{1 - \alpha \cos \phi}}
e^{-im\phi} d\phi
\end{align}
(see \cite{helsing} Section 3, formula (9)).
The integrand of (\ref{eq:mgfhelsing}) is split into a smooth sine term, $H^s$, 
and a near-singular 
cosine term, $H^c$, where $H^s$ and $H^c$ are given by
\begin{align}
\hspace{-3em}
H^s(\phi;\kappa,\alpha) = \frac{\sin(\kappa \sqrt{1 - \alpha \cos \phi})}
{ \sqrt{1 - \alpha \cos \phi}},\quad
H^c(\phi;\kappa, \alpha) = \frac{\cos(\kappa \sqrt{1 - \alpha \cos \phi})}
{\sqrt{1 - \alpha \cos \phi}}. 
\end{align}
The Fourier modes of $H^c$ are computed as the linear convolution of the Fourier 
modes of $\cos(\kappa \sqrt{1 - \alpha \cos \phi } )$ and the Fourier modes of
$1/\sqrt{1 - \alpha \cos \phi}$. The Fourier modes 
of $\cos(\kappa \sqrt{1 - \alpha \cos \phi})$ are
computed via the FFT, while the Fourier modes of 
$1/\sqrt{1-\alpha \cos \phi}$ are known to be 
proportional to $Q_{m-1/2} (\chi)$ (see \cite{cohl}), 
where $Q_{m-1/2}$ is the Legendre function 
of the second kind of half-order, with $\chi$ given by 
\begin{align}
\chi= \frac{r^2+r'^2+(z-z')^2}{2rr'} = \frac{1}{\alpha}.
\end{align}
Note that $\chi \approx 1$ when $\alpha \approx 1$ (i.e., when the 
minimum distance between the source and target is very small).
The authors complete their algorithm by computing $Q_{m-1/2}(\chi)$ 
via a recurrence, which has cost that grows as $O(1/\beta_-)$, where
$\beta_- = \Delta/\rho_0$. 
Thus, their recurrence has poor performance for $\chi\approx 1$
(i.e., when the target and source are close). 
We note that a fast algorithm was recently introduced by Bremer in \cite{bremer},
which evaluates $Q_{m-1/2}(\chi)$ 
in constant run-time independent of $m$. 
Bremer's algorithm for evaluating the Legendre function of the second kind
of half-order \cite{bremer} 
is practically useful, not only as an improvement to \cite{epstein}, but as 
an ingredient in a potential $O(1)$ evaluator for an arbitrary mode of the Green's
function for the Laplace equation (see also Section \ref{se:O1Laplace} for an
alternative algorithm).
We are now ready to discuss the total computational cost of Epstein
et al.'s algorithm. Recall that $R = \sqrt{1 - \alpha \cos \phi}$.
After performing the splitting of Helsing, the $\sin(R)/R$ term is evaluated
in $O(L \log L)$ time with the FFT, where $L$ is the maximum of $4\kappa$ and $M$.
The $\cos(R)/R$ term is evaluated as the convolution of 
the Fourier coefficients of the $1/R$ (the Laplace term) and 
Fourier coefficients of $\cos(R)$. 
The Fourier coefficients of $\cos(R)$ are evaluated in $O(L \log L)$ time, and
the coefficients of $1/R$ term are evaluated in $O(1/\beta_-)$ time, 
where $\beta_- = \Delta/\rho_0$.
Lastly, the convolution of the coefficients of $\cos(R)$ and the
coefficients of $1/R$ is evaluated
in $O(\kappa M)$ time.
Finally, we summarize the Epstein et al.'s algorithm for the modal Green's 
function and its cost as
\begin{align}
\hspace{-4.5em}
\underbrace{\mathcal{F}\Bigg(\cos \big(\kappa 
\sqrt{1-\alpha \cos\phi} \big)\Bigg)}_{O(L \log L)}
       \underbrace{ \vphantom{ \Bigg) } \star}_{O(\kappa M)} 
       \underbrace{\mathcal{F} \Bigg( \frac{1}{\sqrt{1-\alpha\cos\phi}} \Bigg)}
       _{O(1/\beta_-)}
        + 
        \underbrace{\mathcal{F}\Bigg( \frac{\sin(\kappa  \sqrt{1 - \alpha\cos\phi})}
             {\sqrt{1 - \alpha\cos\phi}}  \Bigg)}_{O(L \log L)},
\end{align}
where $\star$ is the discrete convolution operator, $\mathcal{F}$ is 
the discrete Fourier transform (with its cost 
denoted by its implementation via the FFT), $L=\max(4\kappa,M)$,
and $\beta_-$ is the scaled minimum source-to-target distance given
by $\beta_-=\Delta/\rho_0$.
Hence, the cost of Epstein et al.'s algorithm for the modal Green's function is 
\begin{align}
O(L \log L) + O(\kappa M) + O(1/\beta_-).
\end{align}
Epstein et al.'s algorithm can be improved by the application of an
$O(1)$ evaluator for the modal Green's function for the Laplace equation, 
resulting in a cost of
\begin{align} \label{eq:epsteinCost1}
O(L \log L) + O(\kappa M).
\end{align}
Lastly, recall from Section \ref{se:numNeed} that the number of Fourier
coefficients needed when the source and target are close is
\begin{align}\label{eq:numNeed2}
M = O\Big(\frac{1}{\beta_-} + \kappa\Big).
\end{align}
\subsubsection{Method of Vaessen, a Modern Implementation of Gedney and Mittra's
Kernel Splitting}
\label{se:vaessen}
In Vaessen et al. \cite{vaessen}, the authors compute the modal Green's functions 
via an FFT-based kernel-splitting method, using the kernel-splitting
of Gedney and Mittra \cite{gedney}.
However, to integrate the non-singular term, they subsequently split
it again (i.e., they perform two splittings).
This second splitting is actually the same splitting which was later used 
by Helsing \cite{helsing}. 
In the summary below, we depart from the authors' notation 
to make it consistent with our summarization of Epstein et al.'s algorithm.
The authors follow \cite{gedney}, and begin by adding and subtracting the 
term $1/\sqrt{1-\alpha \cos \phi }$ to the integrand of 
(\ref{eq:mgfexp}), 
then split 
the integral into a near-singular and a non-singular term,
resulting in the splitting
\begin{align}
\hspace{-2em}
G_m=  \underbrace{ \vphantom{ \cos(m\phi) \Bigg( \frac{e^{-i\kappa \sqrt{1-\alpha \cos \phi} }-1}
             {  \sqrt{1-\alpha \cos \phi} }
    \Bigg) }
\int_0^\pi \frac{\cos(m\phi)}{ \sqrt{1-\alpha \cos \phi} } d\phi}_
{\textstyle g_{m1}} 
 +
 \underbrace{\int_0^\pi 
 \cos(m\phi) \Bigg( \frac{e^{-i\kappa \sqrt{1-\alpha \cos \phi} }-1}
             {  \sqrt{1-\alpha \cos \phi} }
    \Bigg)  d\phi}_{\textstyle g_{m2}},
\end{align}
where the integral corresponding to the $g_{m1}$ term is the near-singular portion,
and the integral corresponding to the $g_{m2}$ term is the non-singular portion.

To compute the $g_{m1}$ term, the authors use a recurrence inspired by the
recurrence published in \cite{gedney}.
The authors improved on \cite{gedney} by reversing the direction
of the recurrence when $\alpha \approx 1$ (i.e., when the source and target are 
close). 
However, we note that the $g_{m1}$ term is the modal Green's function
of the Laplace equation. Because Section \ref{se:epstein} discusses a potential
fast
$O(1)$ evaluator for the Laplace equation, we do not reproduce Vaessen's
method here.

The authors show that directly applying the FFT to $g_{m2}$ is inefficient
when $\alpha \approx 1$. The cost of accurately computing each 
Fourier coefficient
of $g_{m2}$
grows as $O(1/\beta_-)$, where $\beta_-$ is the scaled minimum distance between
the source and the target given by $\beta_-= \Delta/\rho_0$.

To evaluate the $g_{m2}$ term, the authors split the integral again, 
resulting in the splitting
\begin{align}
\label{eq:vaessenmgf}
\hspace{-6.5em}
\begin{split}
g_{m2} &=
\int_0^\pi 
 \cos(m\phi) \Bigg( \frac{e^{-i\kappa \sqrt{1-\alpha \cos \phi} }-1}
             {  \sqrt{1-\alpha \cos \phi} }
    \Bigg)  d\phi\\
\hspace{-4.5em}    
&=  \int_0^\pi
\cos(m\phi) \Bigg( \frac{\cos(\kappa   \sqrt{1-\alpha \cos \phi} )
      + i \sin(\kappa \sqrt{1-\alpha \cos \phi}) -1}
             {  \sqrt{1-\alpha \cos \phi} }
    \Bigg)  d\phi \\
\hspace{-4.5em}    
&= \int_0^\pi \cos(m\phi) \Bigg( 
    \frac{\cos(\kappa   \sqrt{1-\alpha \cos \phi} ) -1 }
    {  \sqrt{1-\alpha \cos \phi} }
\Bigg)
d\phi
 +i \int_0^\pi  \cos(m\phi) \Bigg( \frac{\sin(\kappa   \sqrt{1-\alpha \cos \phi}) }
    {  \sqrt{1-\alpha \cos \phi} }
  \Bigg) d\phi.
\end{split}
\end{align}
Formula (\ref{eq:vaessenmgf}) is almost identical to 
Epstein et al.'s formula for $G_m$ after
the authors applied the splitting of Helsing
(see formula (\ref{eq:mgfhelsing})
in Section \ref{se:epstein}).
Furthermore, both Epstein et al. and Vaessen et al. evaluate 
formula (\ref{eq:vaessenmgf}) by applying the FFT to compute
the Fourier coefficients associated with the cosine and sine 
terms. 
A superficial difference between the two techniques is that Vaessen et al.
use the fact that the Fourier coefficients of 
\begin{align} \label{eq:mgflaplaceV}
\frac{1}{\sqrt{1 -\alpha \cos \phi}}
\end{align}
are identical to the values of $g_{m1}$ (recall 
that this is the definition of the modal Green's function for the Laplace equation),
which they evaluate with a recurrence, while Epstein et al. compute these
Fourier coefficients by evaluating the associated Legendre functions of half order
(see Section \ref{se:epstein}).
Since the modal Green's functions of the Laplace equation are expressible in terms
of associated Legendre functions of half order, the two methods are equivalent.
As we noted in Section \ref{se:epstein}, due to a recent
algorithm by Bremer, each coefficient can be evaluated
easily with a fast algorithm in $O(1)$ time.
The computational cost of Vaessen et al.'s algorithm is the same as Epstein et al.'s
algorithm.

\section{Preliminaries}
\subsection{Chebyshev Polynomials}
The Chebyshev polynomials are a collection of polynomials on the unit
interval $[-1,1]$, denoted by $T_n(x)$, which are orthogonal to the
weight function $1/\sqrt{1-x^2}$. The $n$th Chebyshev polynomial is given by
the formula
  \begin{align}
T_n(x) = \cos(n \arccos(x))
  \end{align}
(see~\cite{abram}). The extension of $T_n(z)$ to the complex plane is given
by the same formula, only with $z$ replacing $x$. 
\subsection{The Chebyshev Polynomials Evaluated on 
the Bernstein Ellipse} \label{se:chebyshevEvalOnBE}
Recall that the $m$th order Chebyshev polynomial with
complex argument, $T_m(z)$, is given by
\begin{align} \label{eq:chebyshev}
T_m(z) = \cos m \theta,
\end{align}
where $\theta = \arccos (z)$.
An equivalent form of (\ref{eq:chebyshev}) is often used for applications
on ellipses (see, for example, \cite{trefethen}), given by
\begin{align} \label{eq:chebyshevTref}
T_m(z) = \frac{w^m - w^{-m}}{2},
%T_n(z) = \cos n\theta = \frac{e^{in\theta} + e^{in\theta}}{2},
\end{align}
where $z = \frac{1}{2}(w + w^{-1})$, for all $z \in \mathbb{C}$.
This form can be conveniently rewritten in terms of the Joukouwski transformation, 
defined as
\begin{align} \label{eq:joukowski}
J(\zeta) = \frac{1}{2} \big(\zeta + \frac{1}{\zeta} \big).
\end{align} 
Rewriting the Chebyshev polynomial 
using (\ref{eq:joukowski}) we have
\begin{align}
T_m(z) = \frac{(J^{-1}(z))^m - (J^{-1}(z))^{-m}}{2}.
%%%T_n(z) = \cos n\theta = \frac{e^{in\theta} + e^{in\theta}}{2}.
\end{align}
This immediately implies the useful formula
\begin{align} \label{eq:chebyofJ}
T_m(J(z)) = \frac{z^m + z^{-m}}{2},
\end{align}
for the composition of the Chebyshev polynomial with the Joukowski transformation.

Let $C_\rho$ denote a circle of radius $\rho$.
The Joukowski transformation of the family of circles $C_\rho$ with $\rho \neq 1$
has special significance in approximation theory 
and are named the Bernstein ellipses, denoted $E_\rho$, given by
\begin{align} \label{eq:bern}
\begin{split}
E_\rho(\theta) = J\big(C_\rho(\theta)\big) = J(\rho e^{i\theta})
= \frac{1}{2} ( \rho e^{i\theta} + \rho^{-1} e^{-i\theta}) \\
  = \frac{1}{2}(\rho \cos \theta + i\rho \sin \theta + \rho^{-1}\cos \theta 
 -i\rho^{-1} \sin \theta),
\end{split}
\end{align}
where we used the standard parametrization of the
circle, $C_\rho(\theta)=\rho e^{i\theta}$.
Note that both $C_\rho$ and $C_{1/\rho}$ under the Joukowski 
transformation yield the same Bernstein ellipse, that is, $E_\rho= E_{1/\rho}$. We 
adopt the convention in the literature (see, for example, \cite{mason, trefethen})
of parameterizing 
the Bernstein ellipses by $\rho > 1$.
Formula (\ref{eq:bern}) can be simplified into the familiar form of an ellipse,
albeit with the minor axis in the complex plane, given by
\begin{align}
E_\rho(\theta) = a \cos \theta + i b \sin \theta,
\end{align}
where
\begin{align}
a = \frac{1}{2}(\rho + \rho^{-1}), \qquad b = \frac{1}{2}(\rho - \rho^{-1}).
\end{align}
Because the Bernstein ellipses are the Joukowski transformations of circles, and
the Chebyshev polynomials can be defined in terms of the inverse of the Joukowski
transformation,
combining (\ref{eq:bern}) and (\ref{eq:chebyofJ}) leads to a formula
for the composition of a Chebyshev polynomial and a Bernstein ellipse, given by
\begin{align} \label{eq:chebyofBern}
T_m\big(E_\rho(\theta)\big) = T_m\big(J(C_\rho(\theta))\big) 
    = \frac{ \rho^m e^{im\theta} + \rho^{-m} e^{-im\theta}}{2}.
\end{align}
Formula (\ref{eq:chebyofBern}) leads to a useful inequality,
\begin{align} \label{eq:chebyInequal}
\frac{1}{2}(\rho^m - \rho^{-m})
\leq |T_m(E_\rho(\theta))| \leq 
 \frac{1}{2}  (\rho^m + \rho^{-m}),
\end{align}
for $\rho > 1$.
\subsection{Recurrence for a Certain Integral  
involving a Monomial Divided by a $\sqrt{a\tau^2+b}$} \label{se:certainintegral}
In Gustafsson (see \cite{gustafsson}, equations (25) and (26)),
a recurrence relation is given for the integral
of an $n$th degree monomial divided by the square root of a pure quadratic,
\begin{align}
\int \frac{\tau^n}{\sqrt{a\tau^2+b}} d\tau =
  \frac{\tau^{n-1}\sqrt{a\tau^2+b}}{na} 
  - \frac{(n-1)^b}{na} \int \frac{\tau^{n-2}}{\sqrt{a\tau^2 +b}}d\tau,
\end{align}
for $n \geq 2$, with the base case given by the formula
\begin{align}
\int \frac{1}{\sqrt{a\tau^2 +b}}d\tau = 
 \frac{1}{\sqrt{a}}  \ln\Big(\tau \sqrt{a} + \sqrt{a\tau^2 +b}\Big),
\end{align}
which is stable when $|b| <|a|$.
\subsection{The Mapping Between a Chebyshev Expansion and a
Taylor Series} \label{se:cheby2taylor}
The following lemma describes the mapping from a Chebyshev expansion to its
corresponding Taylor series. It can be derived in a straightforward way 
from the formulas in~\cite{abram}.

\begin{lemma} \label{le:cheby2taylor}
Suppose that $c_0,c_1,\ldots,c_n \in \R$. Let $a_0,a_1,\ldots,a_n \in \R$
be given by the formula
  \begin{align}
a_0 = \sum_{j=0}^{\floor{n/2}} c_{2j} (-1)^j,
  \end{align}
and
  \begin{align}
a_i = \sum_{j=0}^{\floor{(n-i)/2}} c_{i+2j} \cdot \frac{(i+2j)(-1)^j}{
2i\cdot j! (i+j-1)_{j}},
  \end{align}
for $i=1,2,\ldots,n$, where $(\cdot)_n$ is the falling Pochhammer symbol.
Then
  \begin{align}
\sum_{i=0}^n a_i x^i = \sum_{i=0}^n c_i T_i(x),
  \end{align}
for all $x\in [-1,1]$.

\end{lemma}

To determine the Taylor series centered at another point on $[-1,1]$, the
following lemma can be used, after applying Lemma~\ref{le:cheby2taylor}.

\begin{lemma}

Suppose that $a_0,a_1,\ldots,a_n \in \R$. Let $b_0,b_1,\ldots,b_n \in \R$
be given by the formula
  \begin{align}
%%%b_l = \sum_{k=l}^n a_k \cdot \binom{k}{l} x_0^{k-l},
b_i = \sum_{j=i}^n a_j \cdot \binom{j}{i} x_0^{j-i},
  \end{align}
for all $i=0,1,\ldots,n$.  Then
  \begin{align}\label{eq:cheby2taylor}
\sum_{i=0}^n b_i (x-x_0)^i = \sum_{i=0}^n a_i x^i,
  \end{align}
for all $x$.

\end{lemma}
\subsection{Chebyshev Coefficients of Analytic functions} \label{se:chebyanal}
The following theorem states that, if a function $f(z)$ can be analytically
continued to the Bernstein ellipse $E_\rho$, then the decay of the
coefficients of its Chebyshev expansion can be nicely bounded. It can
be found in, for example, Chapter~8 of~\cite{nickapprox}.

\begin{theorem} \label{th:chebyanal}
Suppose that $f(z)$ is a analytic function on a neighborhood of the interior
of the Bernstein ellipse $E_\rho$, where it satisfies $\abs{f(z)} \le M$
for all $z\in E_\rho^o$, for some constant $M>0$. Suppose further that 
\begin{align}
f(z) = \sum_{k=0}^\infty a_k T_k(z),
\end{align}
for all $z \in [-1,1]$, where $T_k(z)$ is the Chebyshev polynomial of order $k$.
Then its Chebyshev expansion
coefficients $a_k$ satisfy
  \begin{align}
\abs{a_k} \le 2M \rho^{-k},
  \end{align}
for all $k\ge 1$.

\end{theorem}

\subsection{Contour Integral of a Monomial Divided by a First Degree Polynomial}
\label{se:contourpoly}
For any $k \geq 0$, note the elementary indefinite integral
\fboxit{
\begin{align}
\int_\gamma \frac{z^k}{z-x}dz = \sum_{i=0}^{k-1} \frac{z^{k-i}x^i}{k-i} +
x^k \log(z-x),
\end{align}
}
for all $x \in \mathbb{C}$
\subsection{The Numerical Solution of the Quadratic Equation}
\label{se:quadraticeq}
Suppose that $a,b,c\in \R$, and suppose that the quadratic equation
  \begin{align}
a x^2 + b x + c = 0
  \end{align}
has two distinct roots. The roots are given by either the formula
  \begin{align}
x = \frac{-b \pm \sqrt{b^2 -4ac}}{2a},
    \label{quadrat1}
  \end{align}
or, alternatively, by
  \begin{align}
x = \frac{2c}{-b \mp \sqrt{b^2 -4ac}},
    \label{quadrat2}
  \end{align}
where the root $x_1$ corresponding to the $+$ in~(\ref{quadrat1}) is the root
corresponding to $-$ in~(\ref{quadrat2}), and the root $x_2$ corresponding
to the $-$ in~(\ref{quadrat1}) is the root corresponding to $+$
in~(\ref{quadrat2}). To avoid cancellation error in the numerical evaluation
of the roots, the formula should be chosen based on the sign of $b$. For
example, if $x_1$ is sought, then formula~(\ref{quadrat1}) should be used
when $b < 0$; if $b \ge 0$, then formula~(\ref{quadrat2}) should be used.
\section{Analytical Apparatus}
\subsection{Steepest Descent Contour}
\label{se:steepest}
The modal Green's function is given by 
\begin{align} \label{eq:mgf}
G_m(\bm{x},\bm{x}')=\frac{1}{8\pi^2R_0}\int_{-\pi}^{\pi} 
\frac{e^{-i\kappa\sqrt{1 - \alpha\cos\phi}}} 
{\sqrt{1 - \alpha\cos\phi}} e^{-im\phi}d\phi ,
\end{align}
where $\kappa = kR_0$, $\alpha = 2rr'/R_0^2$, and $R_0^2=r^2+r'^2+(z-z')^2$.
Recall that (\ref{eq:mgf}) is the $m$th Fourier coefficient of the spherical wave
\begin{align} \label{eq:hw}
H^w(\phi) =  \frac{e^{-i\kappa\sqrt{1 - \alpha\cos\phi}}}
{\sqrt{1 - \alpha\cos\phi}} .
\end{align}
Observe that (\ref{eq:hw}) is an even function. 
Rewriting (\ref{eq:mgf}) using (\ref{eq:hw}) and applying the formula for the
$m$th Fourier coefficient of an even function, we have
\begin{align} \label{eq:mgfcos}
G_m = \frac{1}{8\pi^2R_0} \int_{-\pi}^\pi H^w(\phi)e^{-iw\phi} d\phi =
 \frac{1}{4\pi^2R_0} \int_{0}^\pi H^w(\phi) \cos(\phi) d\phi ,
\end{align}
where we have 
omitted rewriting the variables $\bm{x}, \bm{x'}$. In the form (\ref{eq:mgfcos}),
$G_m$ is understood to be a function
of four parameters: $R_0, \alpha, \kappa,$ and $m$.
Lastly, we denote the integrand of (\ref{eq:mgfcos}) by $H_m$, where
\begin{align} \label{eq:hmz}
H_m(\phi) = H^w(\phi) \cos(m\phi).
\end{align}

This leads to an abbreviated form of $G_m$, given by
\fboxit{
\begin{align}
G_m = \int_0^\pi H_m(\phi)d\phi.
\end{align}
}
When $\kappa$ or $m$ are large, $H_m(\phi)$ is highly oscillatory along the real axis. 
However, $H_m(\phi)$ decays to zero in quadrant~IV of the complex plane
for complex arguments with 
sufficiently large positive imaginary components, provided 
that $0 < \Re(\phi) < \pi$.
This suggests that contour integration may be used to avoid evaluating
the oscillatory segment along the real axis.
The integrand is analytic on a neighborhood of $[0,\pi]$, 
so Cauchy's integral theorem can 
be used to deform the integration contour to complex valued $\phi$. 

Applying Cauchy's integral theorem,  we have
\begin{align} \label{eq:mgfcauchy}
\oint_\Gamma H_m(z) dz = 0,
\end{align}
where $\Gamma$ is some closed contour passing along the interval $[0,\pi]$ on
the real axis,
and extending into quadrant~IV in the complex plane. 
We rearrange (\ref{eq:mgfcauchy}) into an
expression for $G_m$, given by
\begin{align} \label{eq:mgfcauchysetminus}
G_m = -\oint_{\Gamma \setminus [0,\pi]} H_m(z)dz.
\end{align}
Determining an appropriate contour $\Gamma \setminus [0,\pi]$ 
is the subject of the subsequent section.
Ideally, one would construct a contour on which $H_m(z)$ undergoes
a finite number of oscillations independent of both $\kappa$ and $m$.
Unfortunately, this is not possible for the general case when
both $\kappa>0$ and $m>0$.
Although it is not always possible to construct a contour on which $H_m(z)$ 
(given by formula (\ref{eq:hmz}))
has a finite number of oscillations, it is always possible to construct a 
contour on which the spherical wave component $H^w(\phi)$ (given by (\ref{eq:hw}))
has exactly
one oscillation, regardless of $\kappa$, $\alpha$, or $R_0$. 
\subsubsection{Gustafsson's Contours} \label{se:gustafsson}
Gustafsson \cite{gustafsson} proposed using contour integration to evaluate
the modal Green's functions by selecting
a contour on which the spherical wave component (\ref{eq:hw})
is non-oscillatory. 
The spherical wave component with complex argument is given by
\begin{align} \label{eq:hwz}
H^w(\phi) =  \frac{e^{-i\kappa\sqrt{1 - \alpha\cos \phi}}}
{\sqrt{1 - \alpha\cos \phi}}, 
\end{align}
where $\phi \in \mathbb{C}$, with $\alpha$ and $\kappa$ defined the same as in
(\ref{eq:mgf}).
Recall that our goal is to construct a contour $\Gamma \setminus [0,\pi]$ which begins
at the point $\phi=0$, travels down into the complex plane sufficiently low,
traverses parallel to the real axis, then travels up  
to the point
$\phi=\pi$.
An adequate contour has the property that it decays (or grows) monotonically during
the first and last segments, which we name $\gamma_1$ and $\gamma_2$ respectively.
The contour parallel to the real axis connecting $\gamma_1$
and $\gamma_2$, corresponds to an integral which by design evaluates to zero. 
We assign this segment the label $\gamma_c$ for ``connecting."
Because it is noncontributory we do not 
derive its expression. We split the integral in (\ref{eq:mgfcauchy}) into 
\begin{align} \label{eq:mgfsplitup}
\int_{[0,\pi]} H(\phi)d\phi
+\int_{\gamma_1} H(\phi)d\phi
+\int_{\gamma_c} H(\phi) d\phi + \int_{\gamma_2} H(\phi) d\phi =0.
\end{align}
We then combine (\ref{eq:mgfsplitup}) with  (\ref{eq:mgfcauchysetminus}) to  
obtain the formula
\begin{align} \label{eq:mgftwocontours}
G_m=-\int_{\gamma_1} H(\phi)d\phi
- \int_{\gamma_c} H(\phi) d\phi
- \int_{\gamma_2} H(\phi) d\phi ,
\end{align} 
where $\gamma_1$, $\gamma_2$ are constructed below, with $\gamma_c$
as a contour connecting $\gamma_1$ and $\gamma_2$.

To construct $\gamma_1$, we choose a curve which intersects $\phi=0$ on which
$H^w(z)$ does not oscillate.
This occurs when $\Re(1 - \alpha \cos z )$ is constant.
Because $\gamma_1$ must intersect $x=\phi$, this contour is defined by
\begin{align} \label{eq:gamma1}
\gamma_1 = \{(x,y): \Re( \sqrt{1 - \alpha \cos (x + iy) )}=
 \sqrt{1 - \alpha} ,\quad y >0 \},
\end{align}
To convert (\ref{eq:gamma1}) into a parametric equation, we perform a change of
variables $\cos \phi = x+ iy$, giving the equation for $\gamma_1$,  
\begin{align} \label{eq:realcontour}
\Re( \sqrt{1 - \alpha (x + iy) })=\sqrt{1 - \alpha}. 
\end{align}
Recall that the formula for the square root of a complex number with negative 
imaginary part is 
\begin{align} \label{eq:negsqrtcomplex}
\sqrt{a +ib} = \sqrt{\frac{\sqrt{a^2+b^2}+a}{2}} -i\sqrt{\frac{\sqrt{a^2+b^2}-a}{2}}.
\end{align}
We solve (\ref{eq:realcontour}) by substituting for the left hand side 
the formula (\ref{eq:negsqrtcomplex}), then squaring both sides,
giving the equation 
\begin{align} \label{eq:realcontoursq}
\sqrt{(1-\alpha x)^2 + \alpha^2 y^2} + 1 - \alpha x = 2(1 - \alpha).
\end{align}
We further simplify (\ref{eq:realcontoursq}) by subtracting $(1-\alpha x)$ from 
both sides, and then square both sides. After solving for $x$,
(\ref{eq:realcontoursq}) becomes
\begin{align} \label{eq:contourparametric}
x = \frac{y^2}{4(1/\alpha - 1)} + 1. 
\end{align}

To construct $\gamma_2$, we choose a contour in a similar fashion,
except that 
$\gamma_2$ must intersect the point $\phi = \pi$.
The same procedure used to arrive at (\ref{eq:contourparametric}) results in
an equation 
for $\gamma_2$ in the $\cos\phi$-plane, given by 
\begin{align} \label{eq:contourparametric2}
x = \frac{y^2}{4(1/\alpha + 1)} - 1. 
\end{align}
Integration on these contours requires the change of variables $z = x+iy = \cos\phi$.
Thus, $dz = -\sin \phi d\phi = -\sqrt{1-z^2}$. 
Recalling that the Chebyshev polynomial of the first kind has the formula 
\begin{align}
T_m(z) = \cos(m\arccos(z))  ,
\end{align}
the $\cos{m\phi}$ term with the above substitution becomes $T_m(z)$.
It is not difficult to show that for $\Re(\kappa) > 0$ the integrand vanishes
as $\Im(z) \to +\infty$ provided that $0\leq\Re(\phi) \leq \pi$. 
Thus, if we construct $\gamma_1$ and $\gamma_2$ to
travel sufficiently high into the complex plane, we have 
\begin{align} \label{eq:mgfconnect}
\int_{\gamma_c} H(z)dz \to 0,
\end{align}
where $\gamma_c$ is the contour connecting $\gamma_1$ and $\gamma_2$.
After this change of variables, we arrive at a formula for $G_m$ where the
integrand has a non-oscillatory spherical wave component, given by
\begin{align} \label{eq:mgfz}
\hspace{-3em}
G_m =
\int_{\gamma_1} \frac{e^{-i\kappa \sqrt{1 - \alpha z}} }
  {\sqrt{1-\alpha z} \sqrt{1-z^2} } T_m(z) dz +
\int_{\gamma_2} \frac{e^{-i\kappa \sqrt{1 - \alpha z}}}
  {\sqrt{1-\alpha z} \sqrt{1-z^2} } T_m(z)  dz,
\end{align}
where we have used (\ref{eq:mgfconnect}) to omit the integral
corresponding to $\gamma_c$.

Our formula (\ref{eq:mgfz}) departs from the form given in
\cite{gustafsson} (see \cite{gustafsson}, formula (19))
in that (\ref{eq:mgfz}) is a
formula for all $m$, while the formula appearing in \cite{gustafsson} is for
the special case $m=1$. 
Although the integrand in (\ref{eq:mgfz}) has a spherical wave component which
monotonically decays on $\gamma_1$ and $\gamma_2$, the rest of the 
integrand oscillates and grows along $\gamma_1$ and $\gamma_2$. In the subsequent
section, we characterize the growth, oscillation, and sign behavior
of the integrand on these contours. 
We then demonstrate that this results in concomitant cancellation
error from integrating the form in (\ref{eq:mgfz}).
\subsubsection{Cancellation Error on Gustafsson's Contours}
\label{se:cancelgustafsson}
We consider the integrand
as the product of three terms, $H^w(z), T^m(z)$, and $H^r(z)$, with
$H^w$ and $H^r$ given as  
\begin{align}
H^w(z) =  \frac{e^{-i\kappa \sqrt{1 - \alpha z}}}
  {\sqrt{1-\alpha z}}, \qquad
  H^r(z) = \frac{1}{\sqrt{1-z^2}}.
\end{align}
On both contours $\gamma_1$ and $\gamma_2$, for points distant from the real axis
(with large imaginary component), the exponential term in $H^w(z)$ decays 
far faster than $T_m(z)$ grows,
meaning the integrand decays to zero as $\Im(z) \to +\infty$.
However, for points on $\gamma_1$ and $\gamma_2$ near the real axis, $T_m(z)$ can
be far larger than $1/H^w(z)$, meaning that the integrand takes on values
with large magnitude, particularly when evaluating the modal Green's function
for large values of $m$ and small values of $\kappa$. 

Being the Fourier coefficient of an analytic function, 
$G_m$ exhibits geometric decay in $m$ but equals
the sum of two integrals, each of which exhibit geometric growth in $m$.
We summarize this behavior with the formula 
\begin{align} \label{eq:cancelerror}
O(a^{-m}) \approx G_m 
    = \int_{\gamma_1} H(z)dz + \int_{\gamma_2}H(z)dz
      \approx O(a^m) + O(a^m),
\end{align}
which is only possible if the integrals have opposite sign. 
Therefore, integrating the form in (\ref{eq:mgfz}) incurs cancellation error
which grows geometrically with $m$.
\subsection{Rational Function Approximation of the Chebyshev Polynomial}
\label{se:ratfunapprox}
Integration of (\ref{eq:mgfz}) incurs cancellation error which grows geometrically
in $m$, due to the growth of the Chebyshev polynomial away from the 
real axis. In this section, we characterize its growth, then propose a rational
function approximation which approximately equals the Chebyshev polynomial
on the interval $[-1,1]$ but instead decays in the complex plane.
\subsubsection{The Growth of the Chebyshev Polynomial in the Complex Plane}
It is helpful to characterize the growth of the Chebyshev polynomial in 
the complex plane.
Recall that the formula for the Bernstein ellipse indexed by parameter $\rho$ is
\begin{align}
E_\rho(\theta) = a \cos \theta + i b \sin \theta,
\end{align}
where
\begin{align}
a = \frac{1}{2}(\rho + \rho^{-1}), \qquad b = \frac{1}{2}(\rho - \rho^{-1}).
\end{align}
Recall that (\ref{eq:chebyInequal}) provides a useful bound
\begin{align} \label{eq:chebyBernIn2}
\frac{1}{2}(\rho^m - \rho^{-m})
\leq |T_m(E_\rho(\theta))| \leq 
 \frac{1}{2}  (\rho^m + \rho^{-m}),
\end{align}
%
%for $\rho > 1$.
%
characterizing the growth of $T_m(E_\rho(\theta))$.  
Note that (\ref{eq:chebyBernIn2}) can be immediately extended to any point 
$z$ in the interior of the Bernstein ellipse $E_\rho$. Thus,
\begin{align} \label{eq:chebyBernIn3}
\frac{1}{2}(\rho^m - \rho^{-m})
\leq |T_m(z))| \leq 
 \frac{1}{2}  (\rho^m + \rho^{-m}),
\end{align}
for all $z \in E_\rho^\circ$, where $E_\rho^\circ \in \mathbb{C}$ denotes
the interior of the
region bounded by $E_\rho$.
\subsubsection{Choice of the Bernstein Ellipse Parameter $\rho$ for an $m$th
order Chebyshev Polynomial} \label{se:bernsteinrule}
Recall that by convention, the parameter $\rho > 1$. 
Hence, by (\ref{eq:chebyBernIn3}), 
\begin{align} \label{eq:chebyBernIn4}
|T_m(z)| \leq \frac{1}{2} (\rho^m + \rho^{-m} )< \rho^m,
\end{align}
for all $z \in E_\rho^\circ$.
Thus, to bound the $m$th order Chebyshev polynomial by an arbitrary 
constant $M$, we pick $\rho$ with the formula,
\begin{align} \label{eq:rhoChoice}
\rho = M^{\frac{1}{m}},
\end{align}
which by (\ref{eq:chebyBernIn4}) bounds $|T_m(z)| < M$ for $z \in E_{\rho}^\circ$,
where $E_{\rho}^\circ \in \mathbb{C}$ denotes
the interior of the
region bounded by $E_{\rho}$.
\subsubsection{Rational Function Approximation of the Chebyshev Polynomial via the
Cauchy Integral Formula} \label{se:ratfuncon}
In this section, we construct a rational function approximation which
is approximately equal to $T_m(z)$ on the interval $[-1,1]$, but, instead of
exhibiting polynomial growth in the complex plane, decays.

The Chebyshev polynomial, $T_m(z)$, is analytic everywhere in the complex plane.
Thus, by Cauchy's integral formula
\begin{align} \label{eq:cif}
T_m(z) = \frac{1}{2\pi i} \oint\limits_{\Gamma} \frac{T_m(v)}{v-z} dv,
\end{align}
where $\Gamma$ is any simple closed contour, and $z$ is a point in the 
interior of $\Gamma$.
Let $\Gamma$ be a Bernstein ellipse with parameter $\rho$, denoted by $E_\rho$. 
Then (\ref{eq:cif}) is given by
\begin{align} \label{eq:cifEr}
T_m(z) = \frac{1}{2\pi i} 
\int\limits_{0}^{2\pi } \frac{T_m(E_\rho(\theta)) E_\rho'(\theta)}
{E_\rho(\theta)-z} d\theta.
\end{align}
Suppose that the integral in (\ref{eq:cifEr}) can be efficiently estimated with 
a quadrature rule, given by the nodes $\theta_1, \theta_2, \dots , \theta_n$ 
and weights $w_1, w_2, \dots, w_n$. 
Then, $T_m(z) \approx R_m(z)$, where
\begin{align} \label{eq:rfapproxT}
R_m(z) = \frac{1}{2\pi i} \sum_{i=1}^n
\frac{T_m(E_\rho(\theta_i))E_\rho'(\theta_i)}
{ E_\rho(\theta_i)- z }  w_i.
\end{align}
Recall from Section \ref{se:chebyshevEvalOnBE} that 
\begin{align} 
T_m\big(E_\rho(\theta)\big) = T_m\big(J(C_\rho(\theta))\big) 
    = \frac{ \rho^m e^{im\theta} + \rho^{-m} e^{-im\theta}}{2}.
\end{align}
Thus, we rewrite (\ref{eq:rfapproxT}) as
\begin{align} \label{eq:tmrff}
R_m(z) = \frac{1}{2\pi i} \sum_{i=1}^n \frac{a_i}{v_i - z} ,
\end{align}
where
\begin{align}
a_i = \big( \frac{ \rho^m e^{im\theta_i} 
        + \rho^{-m} e^{-im\theta_i}}{2} \big) E_\rho'(\theta_i)
        w_i, \qquad v_i = E_\rho(\theta_i),
\end{align}
for $i =1, 2, \dots, n$.
\subsection{The Number of Terms in the Chebyshev Expansions of Analytic Functions}
\label{sec:numterms}
The following theorem states that
the number of Chebyshev polynomials required
to represent $f(z)$ which is $\leq L$ and analytic on the interior of $E_\rho$,
with $\rho= M^{1/m}$,
can be bounded in terms of $M$ and $L$.

\begin{corollary}\label{co:approx}
Suppose that $M > 1$, and let $\rho = M^{1/m}$, for some integer $m> 1$.
Suppose further that $f(z)$ is an analytic function on the interior of 
the Bernstein ellipse $E_\rho$, where it satisfies $\abs{f(z)} \le M$
for all $z\in E_\rho^o$, for some constant $L>0$. 
Suppose further that 
\begin{align}
f(z) = \sum_{k=0}^\infty a_k T_k(z),
\end{align}
for all $z \in [-1,1]$, where $T_k(z)$ is the Chebyshev polynomial of order $k$.
Finally, let $0<\epsilon
\ll 1$ be some small real number. Then, if
  \begin{align}\label{eq:chebycoeff}
k_0 = m{(\log(2L) - \log(\epsilon))/\log(M)},
  \end{align}
then $\abs{a_k} \le \epsilon$ for all $k\ge k_0$.

\end{corollary}

\begin{proof}
The proof follows in a straightforward way from Theorem~\ref{th:chebyanal}.

\end{proof}

Clearly, $k_0 = O(m)$. If, for example, $M=100$, $L=2$, and $\epsilon =
10^{-16}$, then the analytic function $f(z)$, bounded by $L$ in $E_\rho^o$,
could be approximated by a Chebyshev expansion with only $k_0 \approx {8.3
m}$ terms.

\begin{remark} \label{re:trefanalcont} 
We point out, without proving in detail, that this corollary extends to
analytic functions on contours in the complex plane. Suppose that $f(z)$ is
defined on a contour $C$ of length $2$, and can be analytically continued
onto some neighborhood $\Omega$ of $C$, where it stays nicely bounded.
Suppose that the nearest points on $\partial \Omega$ to the ends of $C$ are
at a distance approximately $(\log(M)/m)^2$ away, and the nearest points to
the middle of $C$ are approximately $\log(M)/m$ away. If the curve is quite
smooth, then the arc length parameterization $z(s)\colon [-1,1]\to C$ of $C$
is a conformal mapping from a neighborhood of $[-1,1]$ to a neighborhood of
$C$. If we construct a Bernstein ellipse with $\rho=M_2^{1/m}$ around the
interval $[-1,1]$ in the arc length parameter, then the distance from
$\partial E_\rho$ to the ends will be $(\log(M_2)/m)^2$, and the distance to
the middle will be $\log(M_2)/m$ (see Section~\ref{se:geometrybern}). 
For some $M_2 \approx
M$, the image of that ellipse will be inside $\Omega$.  Thus, the function
$f(z(s))$ will be representable by an $O(m)$-term Chebyshev expansion by 
Corollary~\ref{co:approx}.  In fact, the number of terms will also be given
by formula~(\ref{eq:chebycoeff}).

\end{remark}

\subsection{The Geometry of the Bernstein Ellipse} \label{se:geometrybern}
Recall from Section \ref{se:bernsteinrule} that, for the $m$th order
Chebyshev polynomial, we choose the Bernstein ellipse parameter $\rho$
using the formula
\begin{align}
\rho = M^{\frac{1}{m}},
\end{align}
where $M > 1$ is an arbitrary constant.
In this section, we demonstrate that the distances from Gustafsson's contours
to their intersections with $E_\rho$ are well-behaved.  
\subsubsection{Approximations for the Major and Minor Axes as a Function of $m$}
Recall from Section \ref{se:chebyshevEvalOnBE}
that the axes of the Bernstein ellipse are given by
\begin{align}
a = \frac{1}{2} \big(\rho + \frac{1}{\rho}\big), \qquad 
b = \frac{1}{2} \big(\rho - \frac{1}{\rho}\big),
\end{align}
where $a$ is the semi-major axis (along the real axis)
and $b$ is the semi-minor axis (along the imaginary axis).
For convenience we analyze the case
where $M=e$, giving  $\rho = e^{1/m}$. 
The Taylor expansion of $\rho$ is
\begin{align}
\rho = 1 + \frac{1}{m} + \frac{1}{m^2 2!} + O(\frac{1}{m^3})
\end{align}
Consequently,
\begin{align}
\frac{1}{\rho} =\frac{1}{ 1 + \frac{1}{m} + \frac{1}{m^2 2!} + O(\frac{1}{m^3})} \, .
\end{align}
Recall the formulae for the geometric series for $0 \le |r| < 1$,  
\begin{align} \label{eq:geometricseries}
\hspace{-2em}
\frac{1}{1-r} = 1 + r + r^2 + r^3\cdots , \qquad  
  \frac{1}{1+r} = 1 - r + r^2 - r^3,\cdots
\,\,.
\end{align}
Rewriting $1/\rho$ using the formula for a  
geometric series, we have that 
\begin{align}
\frac{1}{\rho} = 1 - \frac{1}{m} + \frac{1}{2m^2} + O\big(\frac{1}{m^3} \big).
\end{align}
Hence, substituting the Taylor expansions of $\rho$ and $1/\rho$ for the semi-major
and semi-minor axes, we have that
\begin{align} \label{eq:taylora_m}
a = \frac{1}{2}\big(\rho + \frac{1}{\rho}\big) = 
    1 + \frac{1}{2m^2} +  O\big(\frac{1}{m^3}\big).
\end{align}
Likewise, the minor axis is  
\begin{align} \label{eq:taylorb_m}
b = \frac{1}{2}\big(\rho - \frac{1}{\rho}\big) = 
     \frac{1}{m}  + O\big(\frac{1}{m^3}\big).
\end{align}
\subsubsection{The Distances from the Points $z=1$ and $z=-1$ to 
the Bernstein Ellipse as a Function of $m$}
\label{se:geomdist}
Recall that Gustafsson's two contours have origins located at $z=-1$ and $z=1$, 
which are the foci of the Bernstein ellipses.
For each focus, we are interested in two quantities: the quantity 
$a-1$, where $a$ is the semi-major axis, and the $y$-coordinate of the intersection
of the line $x=1$ with the Bernstein ellipse $E_\rho$.
Formula (\ref{eq:taylora_m}) immediately yields
\begin{align}
a - 1 = \frac{1}{2m^2}  + O(\frac{1}{m^3}).
\end{align}
The intersection point of $x=1$ with $E_\rho$
is approximated by substituting the Taylor series expansions of the semi-major
and semi-minor axes into the formula for the Bernstein ellipse, and solving 
for resulting $y$.

Recall the formula for the ellipse,
\begin{align}
\frac{x^2}{a^2} + \frac{y^2}{b^2} = 1.
\end{align}
Substituting in the Taylor expansions from Section \ref{se:geometrybern}
for the semi-major and semi-minor axes, we have that
\begin{align}
\frac{x^2}{\big(1 + \frac{1}{2m^2}+ O(\frac{1}{m^3})\big)^2} +
\frac{y^2}{\big(\frac{1}{m} + O(\frac{1}{m^3})\big)^2} = 1.
\end{align}
Setting $x^2 = 1$, we arrive at an equation for $y$, given by
\begin{align}
\frac{1}{\big(1 + \frac{1}{2m^2}+ O(\frac{1}{m^3})\big)^2} +
\frac{y^2}{\big(\frac{1}{m} + O(\frac{1}{m^3})\big)^2} = 1.
\end{align}
\begin{align}
\frac{m^4}{\big(m^2 + \frac{1}{2} + O(\frac{1}{m})\big)^2} +
\frac{y^2m^2}{\big(1 + O(\frac{1}{m^3})\big)^2} = 1.
\end{align}
We then solve for $y$.
\begin{align} \label{eq:y^2O}
\begin{split}
\frac{y^2}{\big(1 + O(\frac{1}{m})\big)^2} 
&=\frac{1}{m^2}-\frac{m^2}{\big(m^2 + 1+ O(\frac{1}{m})\big)^2} \\
\frac{y^2}{\big(1 + O(\frac{1}{m})\big)^2} &=\frac{1}{m^2}-\frac{1}{m^2} \frac{1}
    {\big(1 + \frac{1}{m^2} + O(\frac{1}{m^3}) \big)^2}\\
\frac{y^2}{\big(1 + O(\frac{1}{m})\big)^2} &=
    \frac{1}{m^2}-\frac{1}{m^2} \Big(1 -\frac{1}{m^2} + O(\frac{1}{m^3})  \Big)^2\\  
\frac{y^2}{\big(1 + O(\frac{1}{m})\big)} &=
    \frac{1}{m^2}-\frac{1}{m^2} \Big(1 -\frac{2}{m^2}
    + O\big(\frac{1}{m^3} \big)  \Big)\\    
\frac{y^2}{\big(1 + O(\frac{1}{m})\big)} &=
    \frac{1}{m^2}-\frac{1}{m^2}  +\frac{2}{m^4} + O\Big(\frac{1}{m^5}\Big)\\
y^2&=
     \bigg( \frac{2}{m^4} + O\Big(\frac{1}{m^5}\Big) \bigg)
    \bigg( (1 + O(\frac{1}{m})\bigg)\\
y^2&=
\sqrt{\frac{2}{m^4} }  \sqrt{1 + O(\frac{1}{m})}
\end{split}
\end{align}
Recall that the Taylor series of $\sqrt{1 +x}$ is
\begin{align} \label{eq:sqrt1+x}
\sqrt{1 +x} = 1 + x\frac{1}{2} - x^2\frac{1}{8} +\cdots \,.
\end{align}
Substituting (\ref{eq:sqrt1+x}) into (\ref{eq:y^2O}), we have that
\begin{align}\label{eq:yheight}
y &=   \frac{\sqrt{2}}{m^2} \big(1 + O(\frac{1}{m}) \big)\\
y &=   \frac{\sqrt{2}}{m^2} + O(\frac{1}{m^3}) \label{eq:orderY}
\end{align}
where we have used the formula for a geometric series.
Hence, the vertical distance from $|z|=1$ to the Bernstein ellipse is on
the order of $1/m^2$ (see Figure \ref{fig:geom1}).
\begin{figure}[h]
\centering
\includegraphics[width=0.45\textwidth]{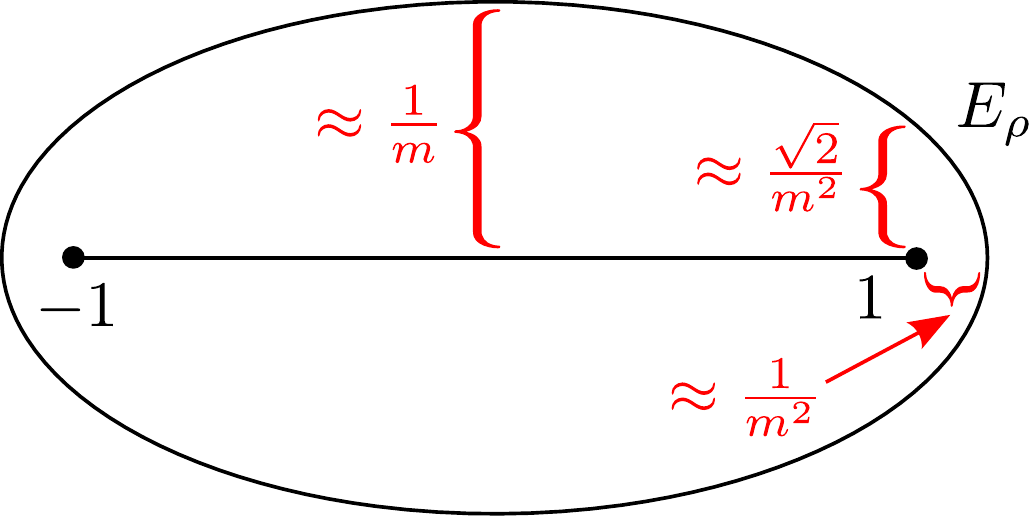}
\caption[Distances of interest associated with the Bernstein ellipse]
{ { \bf Distances $1-a$, $b$, and the intersection of $x=1$ with
the Bernstein ellipse as a function of $m$.}
The distance from $z=1$ to the intersection of $E_\rho$
with the $\Re(\cos\phi)$ axis is $\approx 1/m^2$, and is equal to
$1-a$, where $a$ is the semi-major axis of $E_\rho$. 
The intersection of $x=1$ with $E_\rho$ has a distance of 
$\approx \sqrt{2}/m^2$ to the point $z=1$.
The vertical distance from $z=0$ to $E_\rho$ is $\approx 1/m$, and
is equal to $b$, the semi-minor axis of $E_\rho$.
}
\label{fig:geom1}
\end{figure}
\subsubsection{Length of Gustafsson's 
Contours within the Bernstein Ellipse}
Recall from Section \ref{se:gustafsson} that Gustafsson's contours $\gamma_1$
and $\gamma_2$ can be parameterized as
\begin{align} \label{eq:geomgamma}
\tilde{\gamma_1}(\tau) = \tau^4  +2i\beta_- \tau^2+1,\\
\tilde{\gamma_2}(\tau) = \tau^4 + 2i\beta_+\tau^2 -1.
\end{align}
Consider the sets $\Gamma_1$, $\Gamma_2$, consisting of all possible 
$\gamma_1$, and $\gamma_2$, respectively, defined as
\begin{align} 
\Gamma_1 = \{ \gamma_1: 0 < \beta_- < \infty  \}, \qquad
\Gamma_2 = \{ \gamma_2: 1 < \beta_+ < \infty  \}.
\end{align}
The boundary of $\Gamma_2$, denoted as $\partial \Gamma_2$, is given by $\gamma_2$ 
associated with $\beta_+ = 1$ and the $\gamma_1$ associated with $\beta_+ = \infty$.
We observe that in the limit as $\beta_- \to \infty$, formula (\ref{eq:orderY})
resembles a vertical line (see Figure \ref{fig:geom2}).
Together with bounds from Section \ref{se:geomdist}, 
it is clear that the angle that
$\partial \Gamma_2$ makes with $E_\rho$ is bounded from below.

\begin{figure}[h]
\centering
\includegraphics[width=0.45\textwidth]{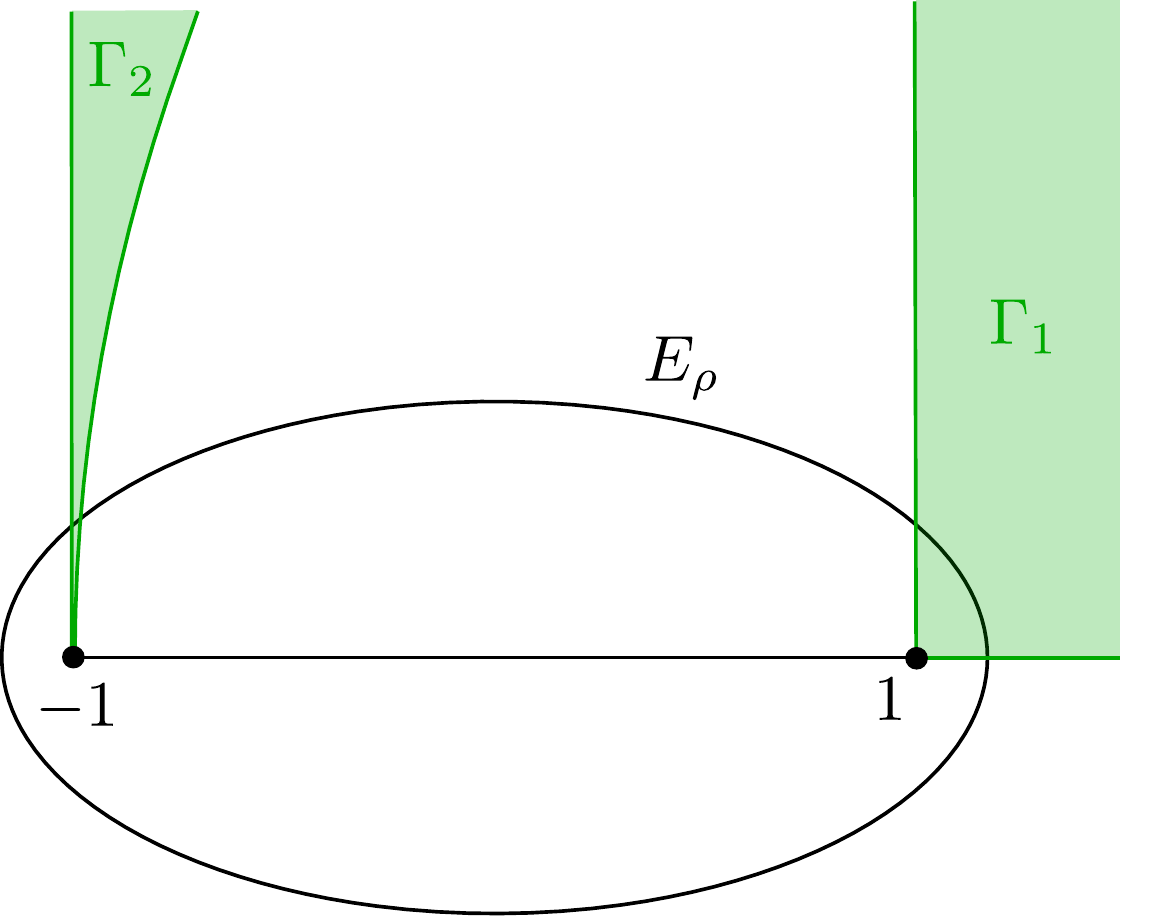}
\caption[Intersections of Gustafsson's contours with the Bernstein ellipses
]{ {\bf The intersection of all possible Gustafsson contours with
the Bernstein ellipse $E_\rho$
in the $z=\cos\phi$ plane.} Recall that Gustafsson's contours are denoted
$\gamma_1$ and $\gamma_2$, where $\gamma_1$ begins at the point $z=1$ and
$\gamma_2$ begins at the point $z=-1$.
The set of all possible $\gamma_1$ is denoted $\Gamma_1$. The
set of all possible $\gamma_2$ is denoted $\Gamma_2$. 
}
\label{fig:geom2}
\end{figure}
\subsection{Evaluating the Modal Green's Function} \label{se:mgfFinal}
After the variable substitution of $z= \cos \phi$, $dz = -\sin{\phi} \, d\phi$,
the formula for the modal Green's function, $G_m$, is 
\fboxit{
\begin{align}
G_m = \int\displaylimits_{[-1,1]}
\frac{e^{-i\kappa \sqrt{1 - \alpha z}} }
  {\sqrt{1-\alpha z} \sqrt{1-z^2} } 
T_m(z)dz.
\end{align}
}
Our rational function approximation, $R_m(z)$, is approximately equal to $T_m(z)$ 
on the interval $[-1,1]$. 
Therefore, substituting $R_m(z)$ for $T_m(z)$, we arrive at a formula for $G_m$,
\begin{align} \label{eq:mgfRF}
G_m \approx \int\displaylimits_{[-1,1]} 
\frac{e^{-i\kappa \sqrt{1 - \alpha z}} }
  {\sqrt{1-\alpha z} \sqrt{1-z^2} } 
R_m(z)dz.
\end{align}
The integrand of (\ref{eq:mgfRF}) is analytic everywhere in the
complex plane except for a
finite number of poles,
so the integral can be deformed. By Cauchy's residue theorem,
\begin{align} \label{eq:mgfres}
\hspace{-3em}\oint_\Gamma \frac{e^{-i\kappa \sqrt{1 - \alpha z}} }
  {\sqrt{1-\alpha z} \sqrt{1-z^2} } 
R_m(z)dz
= 2\pi i \sum_{k=1}^n \underset{z=z_k}{\text{Res}} 
  \Bigg( 
       \frac{e^{-i\kappa \sqrt{1-\alpha z}}}
            {\sqrt{1-\alpha z}\sqrt{1-z^2} } R_m(z) \Bigg),
\end{align}
where $z_1,...,z_n$ are the poles inside $\Gamma$.
Thus, if $\Gamma$ is a closed contour containing the interval $[-1,1]$,
we have that
\begin{align} \label{eq:mgfre}
\hspace{-5em}G_m \approx -\int\displaylimits_{\Gamma\setminus[-1,1]} 
\frac{e^{-i\kappa \sqrt{1 - \alpha z}} }
  {\sqrt{1-\alpha z} \sqrt{1-z^2} } 
R_m(z)dz
+ 2\pi i \sum_{k=1}^n \underset{z=z_k}{\text{Res}} 
  \Bigg( 
       \frac{e^{-i\kappa \sqrt{1-\alpha z}}}
            {\sqrt{1-\alpha z}\sqrt{1-z^2} } R_m(z) \Bigg),
\end{align}
where $\Gamma$ is a contour starting at $z=1$ and ending at $z=-1$.
We select $\Gamma\setminus[-1,1]$ to be the Gustafsson 
contour $\gamma_1 + \gamma_c + \gamma_2$, which we described 
in Section \ref{se:gustafsson}. Since the integrand vanishes over $\gamma_c$, 
we have that
\begin{align} \label{eq:mgfzfinal}
\begin{split}
\hspace{-3em}G_m \approx & 
\int_{\gamma_1} \frac{e^{-i\kappa \sqrt{1 - \alpha z}} }
  {\sqrt{1-\alpha z} \sqrt{1-z^2} } R_m(z) dz   
 +\int_{\gamma_2} \frac{e^{-i\kappa \sqrt{1 - \alpha z}}}
  {\sqrt{1-\alpha z} \sqrt{1-z^2} } R_m(z)  dz \\
 & \qquad +2\pi i \sum_{k=1}^n 
%%% \text{Res}_{z=z_k} 
  \underset{z=z_k}{\text{Res}}
 \Bigg( 
       \frac{e^{-i\kappa \sqrt{1-\alpha z}}}
            {\sqrt{1-\alpha z}\sqrt{1-z^2} } R_m(z) \Bigg).
\end{split}
\end{align} 
\subsection{Removing the Singularity} \label{se:removesingularity}
Recall that the integral in (\ref{eq:mgfz}) corresponding to the $\gamma_1$
contour has the formula 
\begin{align} \label{eq:gamma_1TermFact}
\hspace{-2em}\int_{\gamma_1} \frac{e^{-i\kappa \sqrt{1 - \alpha z}} }
  {\sqrt{1-\alpha z} \sqrt{1-z^2}} 
  T_m(z)  dz=
\int_{\gamma_1} \frac{e^{-i\kappa \sqrt{1 - \alpha z}} }
  {\sqrt{1-\alpha z} \sqrt{1-z}\sqrt{1+z}} 
  T_m(z)  dz.
\end{align}
Observe that the integrand in (\ref{eq:gamma_1TermFact}) has square-root
singularities at $z=1$ and $z=-1$. Furthermore, when $\alpha\approx 1$, 
the product of the terms, 
\begin{align}
\frac{1}{\sqrt{1-\alpha z} \sqrt{1-z}} \approx \frac{1}{1-z},
\end{align}
meaning that the integrand will have a $1/z$-type singularity at $z=1$. 
By careful reparameterization of the contour $\gamma_1$, the singularities 
in (\ref{eq:gamma_1TermFact}) can be removed. 
The variable substitutions and analysis of the singularities
in this section are unchanged when $R_m(z)$ is substituted for $T_m(z)$.
Recall from Section \ref{se:gustafsson}
that the contour $\gamma_1$ can be parameterized as
\begin{align} \label{eq:gamma_1param}
\gamma_1(t) = \frac{t^2}{4(1/\alpha - 1)} + 1 + it,
\end{align}
for $t >0$.
For convenience, we introduce the parameter $\beta_-$, defined as 
\begin{align}
\beta_- = \sqrt{1/\alpha -1}, 
\end{align}
and we observe that, since $0 \leq \alpha < 1$, we have $0 < \beta_- < \infty$.
We then follow \cite{gustafsson} and perform the substitution $t=2\beta_- \tau^2$ 
and reparameterize the contour $\gamma_1$ as $\tilde{\gamma_1}$, given by
\fboxit{
\begin{align} \label{eq:gamma_1kirill}
\tilde{\gamma_1}(\tau) = \gamma_1(2\beta_- \tau^2) = \tau^4  +2i\beta_- \tau^2+1.
\end{align}
}
Gustafsson showed (see \cite{gustafsson}, equations (15) and (16)) that, after  
substituting $z = \tilde{\gamma_1}(\tau)$, $dz = \tilde{\gamma_1}'(\tau)d\tau$,
\fboxit{
\begin{align} \label{eq:dzsqrt}
dz = 4\tau(\tau^2 +i\beta_-) d\tau, \\
\label{eq:dzsqrt2}
\sqrt{1-\alpha z} = -i\sqrt{\alpha}(\tau^2 + i\beta_-).
\end{align}
}
Thus, with the parameterization $z=\tilde{\gamma_1}(\tau)$, formula
(\ref{eq:gamma_1TermFact}) becomes 
\begin{align} \label{eq:gamma_1TildeTerm}
 \frac{-4i}{\sqrt{\alpha}}\,  \int_{0}^\infty
\frac{e^{-i\kappa \sqrt{1 - \alpha \tilde{\gamma_1}(\tau)}}}
{\sqrt{1 - \tilde{\gamma_1}(\tau)}  \sqrt{1 + \tilde{\gamma_1}(\tau)} }
T_m(\tau) d\tau,
\end{align}
where we have used (\ref{eq:dzsqrt}) and (\ref{eq:dzsqrt2})
to cancel the $\sqrt{1-\alpha z}$ term. 
The integrand in (\ref{eq:gamma_1TildeTerm}) has a square-root singularity
near $z=1$. 
Substituting (\ref{eq:gamma_1kirill}) into (\ref{eq:gamma_1TildeTerm}), we have
\begin{align} 
%\hspace{-6em}
\frac{-4i}{\sqrt{\alpha}} \int
  \frac{e^{-i\kappa \sqrt{1-\alpha \tilde{\gamma_1}(\tau)}}}
        {\sqrt{-(\tau^4 + 2i\beta_-\tau^2 )}  \sqrt{1 + \tilde{\gamma_1}(\tau)} } 
        T_m(\tilde{\gamma_1}(\tau)) \tau d\tau,
\end{align}
which can be simplified to
\fboxit{
\begin{align}\label{eq:reparamfinal}
\frac{-4}{\sqrt{\alpha}} \int_0^\infty
  \frac{e^{-i\kappa \sqrt{1-\alpha \tilde{\gamma_1}(\tau)}}}
        {\sqrt{\tau^2 + 2i\beta_-}  \sqrt{1 + \tilde{\gamma_1}(\tau)} } 
        T_m(\tilde{\gamma_1}(\tau)) d\tau.
\end{align}
}
Note that the integrand of (\ref{eq:reparamfinal}) is the product of
a smooth function and the function $1/\sqrt{\tau^2 +2i\beta_-}$. 
Let $F_1(\tilde{\gamma_1}(\tau))$  be the smooth term, given by the formula
\begin{align} \label{eq:smoothF1}
F_1(\tau) = \frac{e^{-i\kappa \sqrt{1 - \alpha \tilde{\gamma_1}_1(\tau)}}}
  {\sqrt{1+\tilde{\gamma_1}(\tau)}} T_m(\tilde{\gamma_1}(\tau)) 
\end{align}
We now rewrite (\ref{eq:reparamfinal}) using (\ref{eq:smoothF1}), so that
\begin{align} \label{eq:mgfF1}
\frac{-4}{\sqrt{\alpha}} \int_0^\infty \frac{F_1(\tau)}
  {\sqrt{\tau^2 + i\beta_-}} d\tau.
\end{align}
The variable substitutions for the integral corresponding to the $\gamma_2$
contour are similar. 
Recall that $\gamma_2$ can be parameterized as 
\begin{align}
\gamma_2(t) = \frac{t^2}{1/\alpha +1} -1 + it.
\end{align}
For the $\gamma_2$ contour, we introduce the parameter 
$\beta_+$, defined as
\begin{align}
\beta_+ =  \sqrt{1/\alpha +1}, 
\end{align}
and we observe that, since $0 \leq \alpha < 1$, we have $1 < \beta_+ < \infty$.
We reparameterize $\gamma_2(t)$ as $\tilde{\gamma_2}(t)$, given by the formula
\fboxit{
\begin{align}
\tilde{\gamma_2}(\tau) = \tau^4 + 2i\beta_+\tau^2 -1,
\end{align}
}
By proceeding as before, we arrive at the formula
for $F_2(\tau)$, 
\begin{align}
F_2(\tau) = \frac{e^{-i\kappa \sqrt{1 -\alpha\tilde{\gamma_2}(\tau)}}}
  {\sqrt{1-\tilde{\gamma_2}(\tau)}} T_m(\tilde{\gamma_2}(\tau)).
    \label{eq:smoothF2}
\end{align}
The formula for the integral corresponding to the $\gamma_2$ contour is thus
\fboxit{
\begin{align} \label{eq:mgfF2}
\frac{4i}{\sqrt{\alpha}} \int_0^\infty \frac{F_2(\tau)}
  {\sqrt{\tau^2 +2i \beta_+}} d\tau.
\end{align}
}
We combine (\ref{eq:mgfF1}) and (\ref{eq:mgfF2}) to write a formula for
the $m$th modal Green's function, 
\begin{align} \label{eq:mgfF1F2}
G_m = 
\frac{-4}{\sqrt{\alpha}} \int_0^\infty \frac{F_1(\tau)}
  {\sqrt{\tau^2 +2i \beta_-}} d\tau + 
\frac{4i}{\sqrt{\alpha}} \int_0^\infty \frac{F_2(\tau)}
  {\sqrt{\tau^2 +2i \beta_+}} d\tau.
\end{align}
Because $\beta_+$ is bounded from below by $1$, the denominator in 
(\ref{eq:mgfF2}) is always greater than $1$. In contrast, when $\alpha\approx 1$,
we have that $\beta_- \approx 0$, which means that the denominator 
in (\ref{eq:mgfF1}) 
$\approx \sqrt{\tau^2}=\tau$.
\subsection{Intersection of the Bernstein Ellipse with the Gustafsson Contour}
\label{se:intersect}
It is natural to split each contour integral into two segments, one
within the Bernstein ellipse and one beyond the ellipse. 
In this section, we solve for the locations where the Gustafsson 
contour, $\gamma_1 \cup \gamma_2$ 
(introduced in \ref{se:gustafsson}),
intersects the Bernstein ellipse in the $z=\cos\phi$-plane.
We derive formulae in terms of the Bernstein ellipse's parameter and in terms
of the Gustafsson contours' parameter.
\subsubsection{Intersection in Terms of the Bernstein Ellipse's Parameter}
\label{se:intersection}
Recall from Section \ref{se:chebyshevEvalOnBE} that the Bernstein 
ellipse, $E_\rho$, is parameterized by the formula
\begin{align} \label{eq:erho2}
E_\rho(\theta) = a \cos \theta + i b \sin \theta,
\end{align}
for $\theta \in [0,2\pi)$, where
\begin{align}
a = \frac{1}{2}(\rho + \rho^{-1}), \qquad b = \frac{1}{2}(\rho - \rho^{-1}).
\end{align}
Also recall from Section \ref{se:gustafsson} that Gustfasson's contours, 
$\gamma_1$ and $\gamma_2$, can be reparameterized as 
\begin{align} \label{eq:geomgamma2}
\tilde{\gamma_1}(\tau) = \tau^4  +2i\beta_- \tau^2+1,\\
\tilde{\gamma_2}(\tau) = \tau^4 + 2i\beta_+\tau^2 -1,
\end{align}
where
\begin{align} 
\beta_- = \sqrt{1/\alpha -1} , \qquad \beta_+ = \sqrt{1/\alpha +1}.
\end{align}
To solve for the parameter $\theta$ for which $E_\rho(\theta)$ intersects
$\tilde{\gamma}_1$, we substitute the real and imaginary parts of $E\rho$ into
$\tilde{\gamma}_1$, to arrive at a quadratic equation in theta.

Let $s$ be the larger of the two roots of
\fboxit{
\begin{align} \label{eq:intersectTheta}
\frac{b^2}{4\beta_+^2}
s^2 + as + \Big(  1 + \frac{b^2}{4\beta_+^2} \Big) = 0.
\end{align}
}
Then,
\fboxit{
\begin{align}
\theta = \arccos(s)
\end{align}
}
A similar procedure is used to solve for $\tau$ such that 
of $\tilde{\gamma}_2(\tau)$ intersects $E_\rho$, resulting in the formula 

\fboxit{
\begin{align} \label{eq:intersectTheta2}
\frac{b^2}{4\beta_-^2}
s^2 + as + \Big( - 1 + \frac{b^2}{4\beta_-^2} \Big) = 0.
\end{align}
}
Then
\fboxit{
\begin{align}
\theta = \arccos(s)
\end{align}
}
\subsubsection{Intersection in Terms of the Gustafsson's Contours' Parameters}

We now solve for parameter $\tau$ for which $\tilde{\gamma}_1(\tau)$
intersects $E_\rho$.

Let $s>0$ be the positive root of

\fboxit{
\begin{align}
s^2 + \Big( \frac{4a^2\beta_-^2}{b^2} +2 \Big)s + (1 - a^2) = 0.
\end{align}
}
Then, 
\fboxit{
\begin{align}
\tau = s^{\frac{1}{4}}.
\end{align}
}
The parameter $\tau$ for which $\tilde{\gamma}_2(\tau)$ 
intersects $E_\rho$ is solved
in a similar fashion.

\fboxit{
\begin{align}
s^2 + \Big( \frac{4a^2\beta_+^2}{b^2} -2 \Big)s + (1 - a^2) = 0.
\end{align}
}
Then,
\fboxit{
\begin{align}
\tau = s^{\frac{1}{4}}.
\end{align}
}
\section{Algorithm}
Recall that the method of Epstein et al. \cite{epstein} has computational cost
which scales with both $|\kappa|$ and $1/(1-\alpha)$, and cannot be easily
parallelized (see Section \ref{se:epstein}).
In contrast, the method of Gustafsson \cite{gustafsson} has computational cost
independent of $\kappa$ and $\alpha$, but incurs cancellation error which 
grows geometrically in $m$ (see Section \ref{se:cancelgustafsson}).

Our technique is to compute the modal Green's function by integrating along 
Gustafsson's contours
using a rational function approximation in place of
the Chebyshev polynomial. 
Because the spherical wave term in the integrand monotonically decays,
our algorithm's order is completely independent of $\kappa$.
Unlike the method of Gustafsson, because our rational function
approximation $R_m(z)$ is bounded
by our choice of Bernstein ellipse $\rho$, our approach does not have cancellation
error which geometrically grows in $m$. 
This comes at the price of having to evaluate the residues of $R_m$ on the 
boundary of 
the corresponding Bernstein ellipse $E_\rho$, which scales with $m$.
We also use the same technique as Gustafsson to evaluate the Green's function
when $\alpha \approx 1$, in time independent of of $\alpha$.
Consequently, our algorithm's computational cost depends only on $m$ and is
independent of both $\kappa$ and $\alpha$, and scales as $O(m)$.
\subsection{Choice of the Rational Function Approximation} \label{se:choiceRat}
Recall from Section \ref{se:ratfuncon}
that the Chebyshev polynomial $T_m(z)$ can be approximated on the interval 
$[-1,1]$ with a rational function, $R_m(z)$, constructed via an application of Cauchy's 
integral formula
followed by the application of a quadrature rule.
This rational function approximation decays quickly in the complex plane.
In this section, we introduce a different approximation, also denoted $R_m(z)$,
which is the
sum of a Cauchy integral and a rational function.

By Cauchy's integral formula, $T_m(z)$ can be expressed as the contour integral,
\begin{align} \label{eq:cif1}
T_m(z) = \frac{1}{2\pi i} \oint\limits_{\Gamma} \frac{T_m(v)}{v-z} dv,
\end{align}
where $\Gamma$ is any simple closed contour, and $z$ is a point in the 
interior of $\Gamma$. Similarly,
\begin{align} \label{eq:cif0}
\frac{1}{2\pi i} \oint_{\Gamma} \frac{T_m(v)}{v-z} dv=0,
\end{align}
for all $z$ outside $\Gamma$.
Recall from \ref{se:bernsteinrule} that for any $m$th order Chebyshev
polynomial, there is an
associated $\rho$ such that, within the Bernstein ellipse
$E_{\rho}$, $T_m(z)$ is bounded by the constant $M$. 
Furthermore, within the interior of $E_\rho$, the Chebyshev polynomial
oscillates exactly once along any possible Gustafsson 
contour (see Section \ref{se:geometrybern}).  
Note that the parameter $\rho$ associated with the Bernstein ellipse $E_\rho$
is a function of $m$, but we denote it simply as $\rho$.
We also denote a scaled copy of $E_\rho$ by $E_{\tilde{\rho}}$, where 
$E_{\tilde{\rho}}$ has twice the major axis and twice the minor axis of 
$E_\rho$ (see Figure \ref{fig:choiceRat}). 
Note that $E_{\tilde{\rho}}$ is not a Bernstein ellipse.

Let $C_\rho$ denote the part of the Bernstein ellipse $E_\rho$ 
between the contours
$\gamma_1$ and $\gamma_2$, represented as the blue arc between $p_1$
and $p_2$ in Figure \ref{fig:choiceRat}, where $p_1 \in \mathbb{C}$ and 
$p_2 \in \mathbb{C}$ are the intersection points of $\gamma_1$ and $\gamma_2$
with $E_\rho$, respectively.
We split the Cauchy integral into two parts,
\begin{align}
T_m(z) = \frac{1}{2\pi i} \int\displaylimits_{C_\rho} \frac{T_m(v)}{v-z}  dv
  + \frac{1}{2\pi i} \int\displaylimits_{E\rho\setminus C_\rho} \frac{T_m(v)}{v-z} dv
\end{align}
\begin{figure}[h]
\centering
\includegraphics[width=0.55\textwidth]{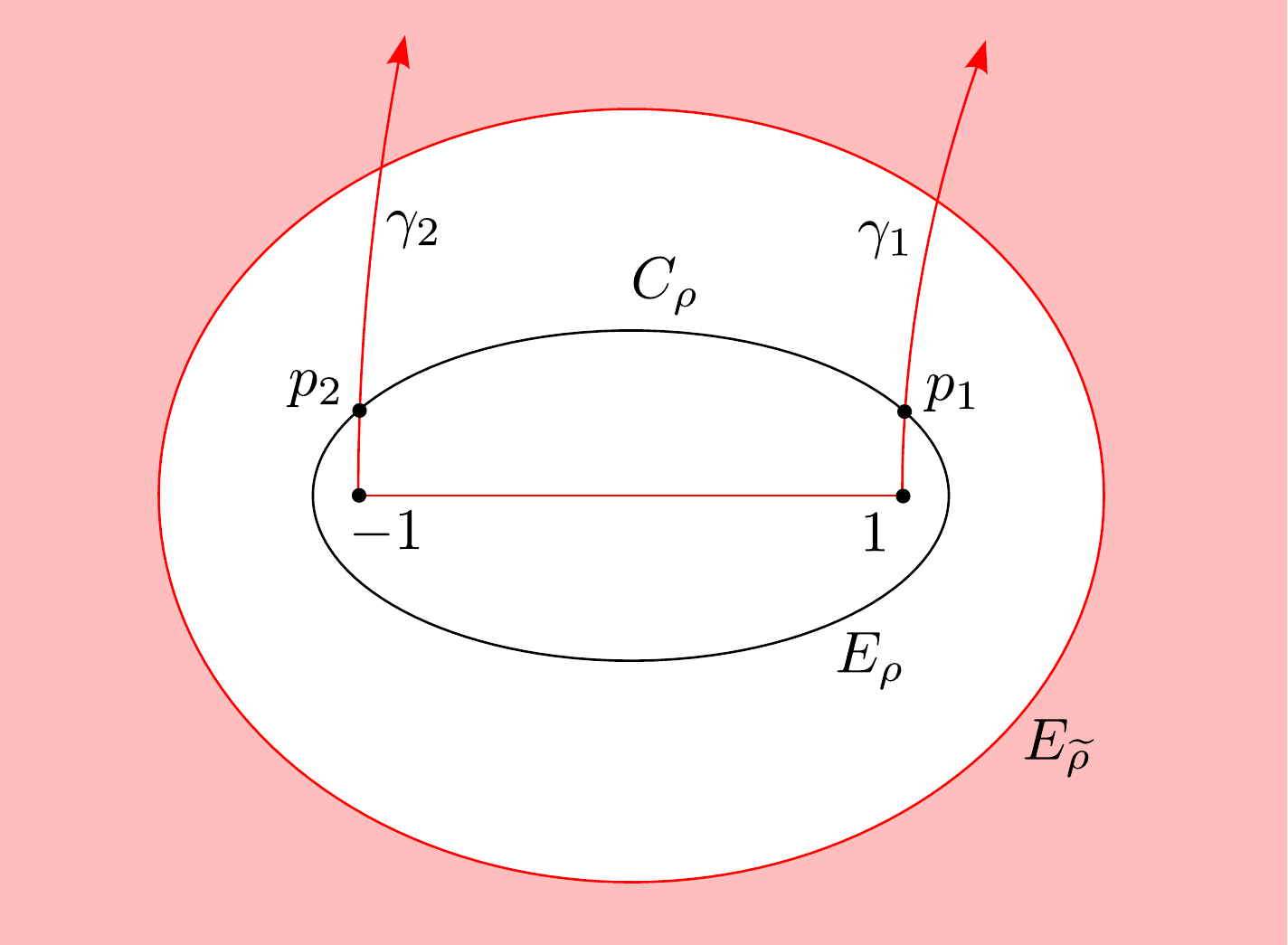}
\caption[Deformation of the contour]
{ {\bf Contours of interest with respect to the function $R_m(z)$ 
in the $z=\cos\phi$ plane.} Gustafsson's
contours are labeled as $\gamma_1$ and $\gamma_2$. The inner Bernstein ellipse
is denoted by $E_\rho$. The outer ellipse is denoted by $E_{\tilde{\rho}}$.
The intersection of $\gamma_1$ with $E_\rho$ is denoted by $p_1$, and the intersection
of $\gamma_2$ with $E_\rho$ is denoted by $p_2$. The arc of $E_\rho$ between $p_2$ and
$p_1$ is denoted by $C_\rho$. 
The contours highlighted
in red and region shaded in red correspond to the
values of $z$ on which the quadrature in (\ref{eq:R_mquadrule})
must be accurate, in the sense of (\ref{eq:r_munit})-(\ref{eq:r_mC}).
}
\label{fig:choiceRat}
\end{figure}

Now, suppose that $\theta_1,\ldots,\theta_n$, $w_1,\ldots,w_n$ are
the nodes and weights of
a quadrature formula such that
\begin{align} \label{eq:R_mquadrule}
\frac{1}{2\pi i} \int\limits_{C_\rho} \frac{T_m(v)}{v-z} dv \approx 
\frac{1}{2\pi i}  \sum_{i=1}^n \frac{T_m(v_i)}{v_i-z} dv_i w_i,
\end{align}
where
$v_i = E_\rho(\theta_i)$, $dv_i = E_\rho'(\theta_i)$, 
and the quadrature is accurate
to precision $\epsilon > 0$ for all $z \in [-1,1]$, 
$z \in \gamma_1 \cap E_{\tilde{\rho}}^o$, 
$z \in \gamma_2\cap E_{\tilde{\rho}}^o$, 
$z \in \mathbb{C} \setminus E_{\tilde{\rho}}^o$, where  
$E_{\tilde{\rho}}^o$ is the interior of $E_{\tilde{\rho}}$
(see Figure \ref{fig:choiceRat}).
Now, let $R_m(z)$ be defined by
\begin{align}
R_m(z) = \frac{1}{2\pi i} \int\displaylimits_{E\rho\setminus C_\rho}
\frac{T_m(v)}{v-z} dv
+ \frac{1}{2\pi i} \sum_{i=1}^N \frac{T_m(v_i)}{v_i-z} dv_i w_i.
\end{align}
We observe that, due to formula (\ref{eq:cif1}),
we have that 
\begin{align} \label{eq:r_munit}
|T_m(z) - R_m(z)| < \epsilon,
\end{align}
for $z \in [-1,1]$. 
We also observe that, due to formula (\ref{eq:cif1}), we have that
\begin{align} \label{eq:r_mgamma1}
|T_m(z) - R_m(z)| < \epsilon,
\end{align}
for $z \in  \gamma_1 \cap E_{\rho}^o$ and 
$z \in \gamma_2 \cap E_{\rho}^o$. 
Likewise, due to formula (\ref{eq:cif0}), 
\begin{align} \label{eq:r_mgamma2}
|R_m(z)| < \epsilon,
\end{align}
for $z \in  \gamma_1 \setminus E_{\rho}^o$
and $z \in  \gamma_2 \setminus E_{\rho}^o$.
We also observe that, due to formula (\ref{eq:cif0}),
\begin{align}  \label{eq:r_mC}
|R_m(z)| < \epsilon,
\end{align}
for $z \in \mathbb{C} \setminus E_{\tilde{\rho}}^o$.
\subsubsection{Deformation of the Contour} \label{se:deformation}
Recall from Section \ref{se:mgfFinal}
that
after the variable substitution of $z= \cos \phi$, $dz = -\sin{\phi} \, d\phi$,
the formula for the modal Green's function, $G_m$, is 
\fboxit{
\begin{align}
G_m = \int\displaylimits_{[-1,1]}
\frac{e^{-i\kappa \sqrt{1 - \alpha z}} }
  {\sqrt{1-\alpha z} \sqrt{1-z^2} } 
T_m(z)dz.
  \label{eq:mgfTF2}
\end{align}
}
Our approximation, $R_m(z)$, by formula (\ref{eq:cif1}), is approximately 
equal to $T_m(z)$ 
on the interval $[-1,1]$. 
Therefore, substituting $R_m(z)$ for $T_m(z)$, we arrive at a formula for $G_m$,
\begin{align} \label{eq:mgfRF2}
G_m \approx \int\displaylimits_{[-1,1]} 
\frac{e^{-i\kappa \sqrt{1 - \alpha z}} }
  {\sqrt{1-\alpha z} \sqrt{1-z^2} } 
R_m(z)dz.
\end{align}
The integrand of (\ref{eq:mgfRF2}) is analytic everywhere in the
complex plane except for a
finite number of poles,
so the integral can be deformed. 
Recall that, for any closed contour $\Gamma \in \mathbb{C}$,
by Cauchy's residue theorem,
\begin{align} \label{eq:mgfres2}
\hspace{-3em}\oint_\Gamma \frac{e^{-i\kappa \sqrt{1 - \alpha z}} }
  {\sqrt{1-\alpha z} \sqrt{1-z^2} } 
R_m(z)dz
= 2\pi i \sum_{k=1}^n \underset{z=z_k}{\text{Res}} 
  \Bigg( 
       \frac{e^{-i\kappa \sqrt{1-\alpha z}}}
            {\sqrt{1-\alpha z}\sqrt{1-z^2} } R_m(z) \Bigg),
\end{align}
where $z_1,...,z_n$ are the poles inside $\Gamma$.
For brevity, let the portion of the integrand in (\ref{eq:mgfres2}) 
corresponding to the spherical
wave component 
be represented by the function $H(z)$, given by
\begin{align} \label{eq:HSphere}
H(z) = \frac{e^{-i\kappa \sqrt{1 - \alpha z}} }
  {\sqrt{1-\alpha z} \sqrt{1-z^2} } . 
\end{align}
If $\Gamma$ is a closed contour containing the interval $[-1,1]$,
we have that
\begin{align} \label{eq:mgfre2}
\hspace{-1em}G_m \approx -\int\displaylimits_{\Gamma\setminus[-1,1]} 
H(z)R_m(z)dz
+ 2\pi i \sum_{k=1}^n \underset{z=z_k}{\text{Res}} 
  \bigg( H(z) 
            R_m(z) \bigg),
\end{align}
where we have substituted formula (\ref{eq:HSphere}) for the spherical wave term. 
We select $\Gamma\setminus [-1,1]$ to be Gustafsson's contours within the
outer ellipse, $ E_{\tilde{\rho}}$, with both segments connected by a short segment 
$\gamma_c \subset \mathbb{C} \setminus E_{\tilde{\rho}}^o$ (see Figure
\ref{fig:choiceRatgammaC}).
\begin{figure}[h]
\centering
\includegraphics[width=0.45\textwidth]{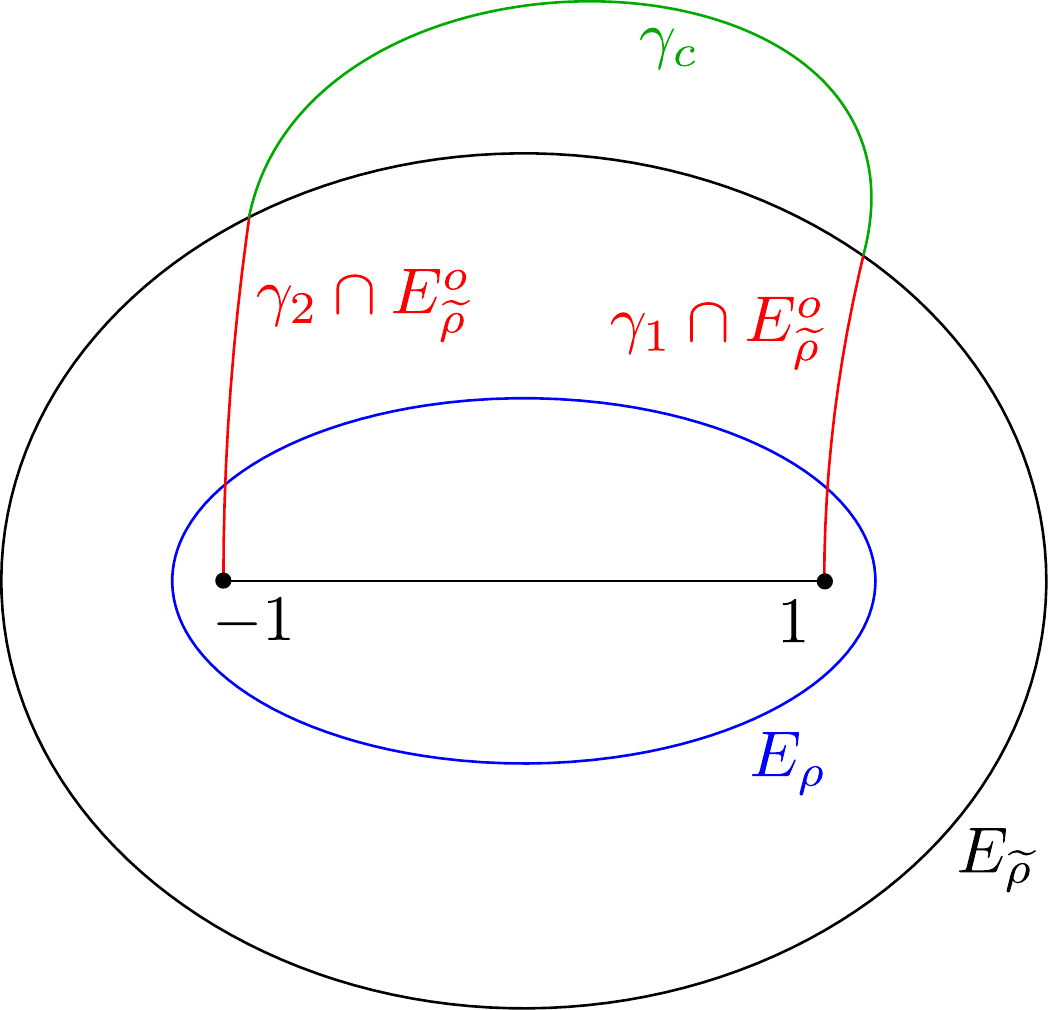}
\caption[Contours of interest for $R_m(z)$]{
{ \bf Contours used in formula (\ref{eq:mgfz2}) in the $z=\cos\phi$ plane.}
The interior Bernstein ellipse is denoted by $E_\rho$ and drawn in blue.
The exterior ellipse is denoted by $E_{\tilde{\rho}}$.
Gustafsson's contours within the exterior ellipse are denoted by 
$\gamma_2 \cap E_{\tilde{\rho}}^o$ 
and $\gamma_2 \cap E_{\tilde{\rho}}^o$ and drawn in red. The contour 
$\gamma_c \subset \mathbb{C} \setminus E_{\tilde{\rho}}^o$, connecting 
the $\gamma_1$ and $\gamma_2$ segments, is drawn in green.
The intersection of $\gamma_1$ with $E_\rho$ is denoted by $p_1$, and the intersection
of $\gamma_2$ with $E_\rho$ is denoted by $p_2$. 
}
\label{fig:choiceRatgammaC}
\end{figure}
Substituting this choice of $\Gamma\setminus [-1,1]$ into (\ref{eq:mgfre2}), we have
\begin{align} \label{eq:mgfz2}
\begin{split}
\hspace{-2em}
G_m \approx 
\int\displaylimits_{\gamma_1 \cap E_{\tilde{\rho}}^o} 
 H(z) R_m(z) dz +
\int\displaylimits_{\gamma_2 \cap E_{\tilde{\rho}}^o} 
  H(z) R_m(z)  dz+
\int\displaylimits_{\gamma_c} 
  H(z) R_m(z)  dz \\
  + 2\pi i \sum_{k=1}^n \underset{z=z_k}{\text{Res}} 
  \Bigg( 
        H(z) R_m(z) \Bigg),
\end{split}
\end{align}
where $\gamma_1$ and $\gamma_2$ are Gustafsson's contours as described in Section
\ref{se:gustafsson}, and $E_{\tilde{\rho}}^o$ is the interior of the scaled
Bernstein ellipse introduced earlier
%%%in Section \ref{se:choiceRat}
(see Figure \ref{fig:choiceRatgammaC}).
We split the integral corresponding to the $\gamma_1$ contour into 
\begin{align} \label{eq:gamma_1expand}
\hspace{-2em}
\int\displaylimits_{\gamma_1} H(z) R_m(z) dz =
\int\displaylimits_{\gamma_1 \cap E_{\rho}^o } H(z) R_m(z) dz +
\int\displaylimits_{\gamma_1 \setminus E_{\rho}^o} H(z) R_m(z) dz .
\end{align}
Recall that by formula (\ref{eq:r_mgamma1}), $R_m(z) \approx T_m(z)$ for  
$z\in \gamma_1 \cap E_{\rho}^o$ and for $z \in \gamma_2 \cap E_{\rho}^o$. 
Also, recall that by formula (\ref{eq:r_mgamma2}), $R_m(z) \approx \epsilon$ for  
$z\in \gamma_1 \setminus E_{\rho}^o$ and for 
$z\in \gamma_2 \setminus E_{\rho}^o$. 
Substituting (\ref{eq:r_mgamma1}) and (\ref{eq:r_mgamma2}) into 
(\ref{eq:gamma_1expand}),
we arrive at a formula for the $\gamma_1$ contour within the 
interior of $E_{\tilde{\rho}}$, 
\begin{align} \label{eq:gamma_1simp}
\hspace{-2em}
\int\displaylimits_{\gamma_1  \cap E_{\tilde{\rho}}^o } H(z) R_m(z) dz \approx 
\int\displaylimits_{\gamma_1 \cap E_{\rho}^o } H(z) T_m(z) dz.
\end{align}
Likewise, 
the formula for the $\gamma_2$ contour within the 
interior of $E_{\tilde{\rho}}$ is
\begin{align} \label{eq:gamma_2simp}
\hspace{-2em}
\int\displaylimits_{\gamma_2 \cap E_{\tilde{\rho}}^o  } H(z) R_m(z) dz \approx 
\int\displaylimits_{\gamma_2 \cap E_{\rho}^o } H(z) T_m(z) dz.
\end{align}
We also observe that, due to formula (\ref{eq:r_mC}),
the integral corresponding to $\gamma_c$ evaluates to zero.
We now substitute our formulas for
the $\gamma_1 \cap E_{\tilde{\rho}}^o $,
$\gamma_2  \cap  E_{\tilde{\rho}}^o $, and $\gamma_c$ contours into
(\ref{eq:mgfz2}) to arrive at
\begin{align} \label{eq:mgfz3}
\begin{split}
\hspace{-3em}
G_m \approx 
\int\displaylimits_{\gamma_1 \cap E_{\rho}^o } \frac{e^{-i\kappa \sqrt{1 - \alpha z}} }
  {\sqrt{1-\alpha z} \sqrt{1-z^2} } T_m(z) dz +
\int\displaylimits_{\gamma_2 \cap E_{\rho}^o } \frac{e^{-i\kappa \sqrt{1 - \alpha z}}}
  {\sqrt{1-\alpha z} \sqrt{1-z^2} } T_m(z)  dz\\
  + 2\pi i \sum_{k=1}^n \underset{z=z_k}{\text{Res}} 
  \Bigg( 
       \frac{e^{-i\kappa \sqrt{1-\alpha z}}}
            {\sqrt{1-\alpha z}\sqrt{1-z^2} } R_m(z) \Bigg).
            \end{split}
\end{align}
\subsubsection{Interpretation of the Residues in
Formula (\ref{eq:mgfz3}) as a Quadrature Formula for the Contour $C_\rho$}
\label{se:interpret}
By Cauchy's integral theorem,
\begin{align} \label{eq:cauchyGM}
\hspace{-3em}
G_m \approx 
\int\displaylimits_{\gamma_1 \cap E_{\rho}^o }H(z) T_m(z) dz +
\int\displaylimits_{\gamma_2 \cap E_{\rho}^o }H(z) T_m(z) dz +
\int\displaylimits_{C_\rho }H(z) T_m(z) dz,
\end{align}
where $\gamma_1$, $\gamma_2$, $E_{\rho}^o $, and $C_\rho$ are described in 
Section \ref{se:choiceRat}.

Subtracting (\ref{eq:mgfz3}) from (\ref{eq:cauchyGM}), and rearranging, we arrive
at a formula for the $C_\rho$ contour,
\begin{align} \label{eq:gmAlmostQuad}
\hspace{-3em}
\int\displaylimits_{C_\rho }  \frac{e^{-i\kappa \sqrt{1-\alpha z}}}
            {\sqrt{1-\alpha z}\sqrt{1-z^2} } 
 T_m(z) dz \approx
2\pi i \sum_{k=1}^n \underset{z=z_k}{\text{Res}} 
  \Bigg( 
       \frac{e^{-i\kappa \sqrt{1-\alpha z}}}
            {\sqrt{1-\alpha z}\sqrt{1-z^2} } 
            R_m(z) \Bigg).
\end{align}
Recall from Section \ref{se:choiceRat} that
\begin{align} \label{eq:RmzIntme}
R_m(z) =\frac{1}{2\pi i} \int\displaylimits_{E\rho\setminus C_\rho} 
\frac{T_m(v)}{v-z} dv
+ \frac{1}{2 \pi i} \sum_{i=1}^n \frac{T_m(v_i)}{v_i-z} dv_i w_i,
\end{align}
where $v_i = E_\rho(\theta_i)$, $dv_i = E_\rho'(\theta_i)$,
and $\theta_1,\ldots,\theta_n$, $w_1,\ldots,w_n$ are nodes and weights of 
the quadrature constructed
in (\ref{eq:R_mquadrule}).
Thus, the residues $z_1,\ldots,z_n$ in (\ref{eq:gmAlmostQuad}) correspond 
to the points $v_1,\ldots,v_n$ and 
\begin{align} \label{eq:resIsQuad}
\hspace{-4.2em}
2\pi i \sum_{i=1}^n \underset{z=z_i}{\text{Res}} 
  \Bigg( 
       \frac{e^{-i\kappa \sqrt{1-\alpha z}}}
            {\sqrt{1-\alpha z}\sqrt{1-z^2} } 
            R_m(z) \Bigg) = 
 \sum_{i=1}^n 
       \frac{e^{-i\kappa \sqrt{1-\alpha v_i}}}
            {\sqrt{1-\alpha v_i}\sqrt{1-v_i^2} } 
            T_m(v_i)  dv_i w_i.
\end{align}
Substituting (\ref{eq:resIsQuad}) into (\ref{eq:gmAlmostQuad}), we have that
\begin{align} \label{eq:quadForCrho}
 \sum_{i=1}^n 
       \frac{e^{-i\kappa \sqrt{1-\alpha v_i}}}
            {\sqrt{1-\alpha v_i}\sqrt{1-v_i^2} } 
            T_m(v_i) dv_i w_i   
\approx
\int\displaylimits_{C_\rho }  \frac{e^{-i\kappa \sqrt{1-\alpha z}}}
            {\sqrt{1-\alpha z}\sqrt{1-z^2} } 
 T_m(z) dz ,
\end{align}
which resembles a quadrature formula for the
contour integral on $C_\rho$. 
Substituting formula (\ref{eq:quadForCrho}) 
into formula (\ref{eq:mgfz3}), we arrive at
\begin{align} \label{eq:mgfzQuad}
\begin{split}
\hspace{-3em}
G_m \approx 
\int\displaylimits_{\gamma_1 \cap E_{\rho}^o } \frac{e^{-i\kappa \sqrt{1 - \alpha z}} }
  {\sqrt{1-\alpha z} \sqrt{1-z^2} } T_m(z) dz +
\int\displaylimits_{\gamma_2 \cap E_{\rho}^o } \frac{e^{-i\kappa \sqrt{1 - \alpha z}}}
  {\sqrt{1-\alpha z} \sqrt{1-z^2} } T_m(z)  dz\\
+ \sum_{i=1}^n 
       \frac{e^{-i\kappa \sqrt{1-\alpha v_i}}}
            {\sqrt{1-\alpha v_i}\sqrt{1-v_i^2} } 
            T_m(v_i) dv_i w_i ,
            \end{split}
\end{align}
where $v_i = E_\rho(\theta_i)$, $dv_i = E_\rho'(\theta_i)$,
and $\theta_1,\ldots,\theta_n$, $w_1,\ldots,w_n$ are 
the nodes and weights of the quadrature constructed
in (\ref{eq:R_mquadrule}).
\subsection{Evaluation of the Integral on Gustafsson's Contour
when $\alpha \approx 1$}
\label{se:evalalpha}
Recall from Section \ref{se:interpret} that the formula for
the $m$th modal Green's function is
\begin{align} \label{eq:mgfzQuad2}
\begin{split}
\hspace{-3em}
G_m \approx 
\int\displaylimits_{\gamma_1 \cap E_{\rho}^o } \frac{e^{-i\kappa \sqrt{1 - \alpha z}} }
  {\sqrt{1-\alpha z} \sqrt{1-z^2} } T_m(z) dz +
\int\displaylimits_{\gamma_2 \cap E_{\rho}^o } \frac{e^{-i\kappa \sqrt{1 - \alpha z}}}
  {\sqrt{1-\alpha z} \sqrt{1-z^2} } T_m(z)  dz\\
+ \sum_{i=1}^n 
       \frac{e^{-i\kappa \sqrt{1-\alpha v_i}}}
            {\sqrt{1-\alpha v_i}\sqrt{1-v_i^2} } 
            T_m(v_i) dv_i w_i ,
            \end{split}
\end{align}
where $v_i = E_\rho(\theta_i)$, $dv_i = E_\rho'(\theta_i)$,
and $\theta_1,\ldots,\theta_n$, $w_1,\ldots,w_n$ are 
the nodes and weights of the quadrature constructed
in (\ref{eq:R_mquadrule}).
Recall also from Section \ref{se:removesingularity} that the integrals 
in (\ref{eq:mgfzQuad2}) can be written as 
\begin{align} \label{eq:mgfF1F2v2}
\begin{split}
G_m \approx 
\frac{-4}{\sqrt{\alpha}} \int_0^{\tau_1} \frac{F_1(\tau)}
{\sqrt{\tau^2 +2i \beta_-}} d\tau + 
\frac{4}{\sqrt{\alpha}} \int_0^{\tau_2} \frac{F_2(\tau)}
  {\sqrt{\tau^2 +2i \beta_+}} d\tau\\
  + \sum_{i=1}^n 
       \frac{e^{-i\kappa \sqrt{1-\alpha v_i}}}
            {\sqrt{1-\alpha v_i}\sqrt{1-v_i^2} } 
            T_m(v_i) dv_i w_i ,
            \end{split}
\end{align}
where $F_1(\tau)$ and $F_2(\tau)$ are smooth functions corresponding to the 
$\gamma_1$ and $\gamma_2$ contours, respectively
(see Section \ref{se:removesingularity},
equations (\ref{eq:mgfF1}) and (\ref{eq:mgfF2})), 
$\tau_1$ and $\tau_2$ are positive parameters such that  
$\gamma_1(\tau_1)$ and $\gamma_2(\tau_2)$ intersect $E_\rho$, respectively, 
and  
\begin{align}
\beta_- = \sqrt{1/\alpha -1} , \qquad \beta_+ = \sqrt{1/\alpha +1}.
\end{align}
When $\alpha \approx 1$,
the parameter $\beta_+ \approx 2$, meaning that 
the integrand in (\ref{eq:mgfF1F2v2}) corresponding
$\gamma_2$ remains a smooth function of $\tau$ for all values
of $0 \leq \alpha < 1$, and can be evaluated efficiently with 
a Gauss-Legendre quadrature. 
In contrast, 
when $\alpha \approx 1$, the parameter $\beta_- \approx 0$.
Consequently, for $\alpha\approx 1$, the integrand in (\ref{eq:mgfF1F2v2})
corresponding to
the $\gamma_1$ contour 
resembles a $1/\tau$ singularity at $\tau=0$. 
\subsubsection{Evaluation of the Integral on the Contour $\gamma_1$
when $\alpha \approx 1$}
\label{se:evalgamma1}
We integrate along the contour $\gamma_1$ using the following procedure.  
Observe that for $\tau$ sufficiently large, the integrand is smooth.
Thus we split the integral into two parts,
\begin{align}
\hspace{-4em}
\frac{-4}{\sqrt{\alpha}} \int_0^{\tau_1} \frac{F_1(\tau)}
  {\sqrt{\tau^2 +i \beta_-}} d\tau =
\frac{-4}{\sqrt{\alpha}} \int_0^{\tau_0} \frac{F_1(\tau)}
  {\sqrt{\tau^2 +i \beta_-}} d\tau +
\frac{-4}{\sqrt{\alpha}} \int_{\tau_0}^{\tau_1} \frac{F_1(\tau)}
  {\sqrt{\tau^2 +i \beta_-}} d\tau.
\end{align}
The integral corresponding to the interval $[\tau_0,\tau_1]$ can be efficiently
computed using a Gauss-Legendre quadrature. 
The integral corresponding to the interval $[0,\tau_0]$ is evaluated
with a specialized quadrature based on the technique 
used by Gustafsson (see \cite{gustafsson}, Section 4.2),
described below.
Recall that $F_1(\tau)$ is smooth, given by the formula
\begin{align}
F_1(\tau) = \frac{e^{-i\kappa \sqrt{1 - \alpha \tilde{\gamma_1}(\tau)}}}
  {\sqrt{1+\tilde{\gamma_1}(\tau)}} T_m(\tilde{\gamma_1}(\tau)),
\end{align}
where $\tilde{\gamma_1}(\tau)$ is  
\begin{align}
\tilde{\gamma_1}(\tau) =  \tau^4  + i\beta_- \tau^2 +1.
\end{align}
We expand $F_1(\tau)$ in k terms of its Taylor series about the point $\tau=0$,
given by the formula  
\begin{align} \label{eq:taylorF1}
F_1(\tau) \approx \sum_{n=0}^k a_n \tau^n. 
\end{align}
We compute the coefficients $a_n$ by first forming a Chebyshev expansion 
of $F_1(\tau)$ on the interval $[0,\tau_0]$, and then using the mapping
(\ref{eq:cheby2taylor})
described in Section \ref{se:cheby2taylor}. This mapping takes the
Chebyshev expansion coefficients and returns the corresponding Taylor expansion
at $\tau=0$.

We substitute (\ref{eq:taylorF1}) into the integral corresponding 
to the interal $[0,\tau_0]$ from (\ref{eq:mgfF1F2v2}), resulting in
\begin{align} \label{eq:mgfF1Taylor}
\frac{-4}{\sqrt{\alpha}} \int_0^{\tau_0} \frac{F_1(\tau)}
  {\sqrt{\tau^2 +2i \beta_-}} d\tau &\approx  
\frac{-4}{\sqrt{\alpha}} \int_0^{\tau_0} \frac{  
 \sum_{n=0}^k a_n\tau^n}
  {\sqrt{\tau^2 +i \beta_-}} d\tau \\
  &= \frac{-4}{\sqrt{\alpha}} \sum_{n=0}^k a_n
  \int_0^{\tau_0} \frac{\tau^n}{\sqrt{\tau^2 + 2i\beta_-}} d\tau
\end{align}
Recall from Section \ref{se:certainintegral} that the integral of
$\tau^n$ divided by $\sqrt{a\tau^2 +b}$ has the recurrence relation
\begin{align} \label{eq:recurrencecertain}
\int \frac{\tau^n}{\sqrt{a\tau^2+b}} d\tau =
  \frac{\tau^{n-1}\sqrt{a\tau^2+b}}{na} 
  - \frac{(n-1)^b}{na} \int \frac{\tau^{n-2}}{\sqrt{a\tau^2 +b}}d\tau,
\end{align}
for $n \geq 2$, where the base case has the formula
\begin{align}
\int \frac{1}{\sqrt{a\tau^2 +b}}d\tau = 
 \frac{1}{\sqrt{a}}  \ln\Big(\tau \sqrt{a} + \sqrt{a\tau^2 +b}\Big),
\end{align}
and this recurrence is known to be stable when $|b| < |a|$. 
We observe that $\beta_- \ll 1$ when
$\alpha \approx 1$, meaning that the recurrence
given by (\ref{eq:recurrencecertain}) is stable 
when applied to (\ref{eq:mgfF1Taylor}).   
\subsection{Construction of the Quadratures to Evaluate the Integral over 
the Contour $C_\rho$}
\label{se:constructboth}
Recall from Section \ref{se:choiceRat} that our approximation, $R_m(z)$, of the
$m$th order Chebyshev polynomial,
$T_m(z)$,
has the formula
\begin{align} \label{eq:R_mQuad}
R_m(z) =\frac{1}{2\pi i} \int\displaylimits_{E\rho\setminus C_\rho} 
\frac{T_m(v)}{v-z} dv
+ \frac{1}{2 \pi i} \sum_{i=1}^n \frac{T_m(v_i)}{v_i-z} dv_i w_i,
\end{align}
where $v_i = E_\rho(\theta_i)$, $dv_i = E_\rho'(\theta_i)$, 
$C_\rho$ is the region of the Bernstein ellipse $E_\rho$ between
and the contours $\gamma_1$ and $\gamma_2$, and
$\theta_1,\ldots,\theta_n$ and $w_1,\ldots ,w_n$ are the
nodes and weights of a quadrature 
such that 
\begin{align} \label{eq:R_mquadrule2}
\frac{1}{2\pi i} \int\limits_{C_\rho} \frac{T_m(v)}{v-z} dv \approx 
\frac{1}{2\pi i}  \sum_{i=1}^n \frac{T_m(v_i)}{v_i-z} dv_i w_i,
\end{align}
for 
$z \in [-1,1]$, 
$z \in \gamma_1 \cap E_{\tilde{\rho}}^o$, 
$z \in \gamma_2\cap E_{\tilde{\rho}}^o$, 
$z \in \mathbb{C} \setminus E_{\tilde{\rho}}^o$, where  
$E_{\tilde{\rho}}^o$ is the interior of $E_{\tilde{\rho}}$.
Recall also from Section \ref{se:choiceRat} that, by Cauchy's integral formula,
\begin{align}
T_m(z) = \oint_{E_\rho} \frac{T_m(v)}{v-z}dv,
\end{align}
for $z \in E_\rho$,
where $E_\rho$ is the Bernstein ellipse described in Section \ref{se:choiceRat}.
 Finally, recall from Section~\ref{se:bernsteinrule} that, if
$\rho=M^{1/m}$, then $\abs{T_m(z)} < M$ for all $z \in E_\rho^o$.
For the sake of simplicity, we first assume that $M=e$, and then consider the
case for general $M$ in Sections~\ref{sec:errrmz} and~\ref{sec:numquadr}.

We summarize the contours on which $R_m(z)$ approximates $T_m(z)$ point-wise
by stating that
\begin{align} \label{eq:pwsplitme}
\Big| \int_{C_\rho} \frac{T_m(v)}{v-z}dv -
  \sum_{i=1}^n \frac{T_m(v_i)dv_iw_i}{v_i -z} \Big| < \epsilon,
\end{align}
for all 
$z \in [-1,1]$, 
$z \in \gamma_1 \cap E_{\tilde{\rho}}^o$, 
$z \in \gamma_2\cap E_{\tilde{\rho}}^o$, 
$z \in \mathbb{C} \setminus E_{\tilde{\rho}}^o$, where  
$E_{\tilde{\rho}}^o$ is the interior of $E_{\tilde{\rho}}$.
Let $p_1 \in \mathbb{C}$ and $p_2 \in \mathbb{C}$
denote the intersections of $\gamma_1$ and $\gamma_2$ with $E_\rho$, respectively
(see Figure \ref{fig:crho}).
\begin{figure}[h]
\centering
\includegraphics[width=0.45\textwidth]{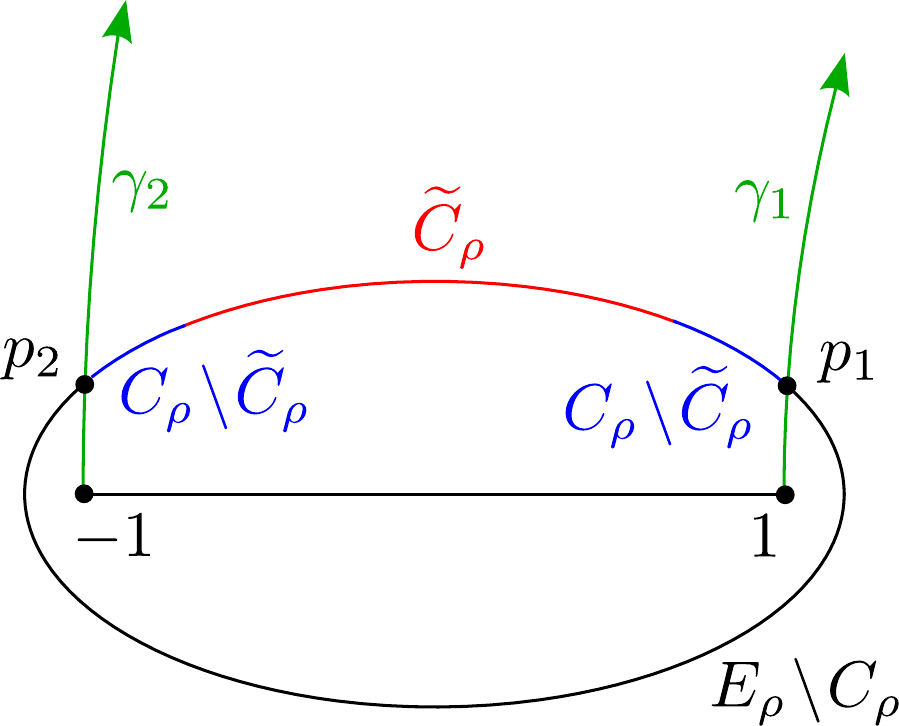}
\caption[Splitting of $C_\rho$ to construct separate quadratures]
{ { \bf Splitting of the Bernstein ellipse into $\tilde{C}_\rho$ and
$C_\rho \setminus \tilde{C}_\rho$ based on proximity to Gustafsson's contours
in the $\cos\phi$ plane.}
Gustafsson's contours are denoted as $\gamma_1$ and $\gamma_2$, and drawn
in green. The points where $\gamma_1$ and $\gamma_2$ intersect $E_\rho$ are denoted
as $p_1$ and $p_2$, respectively. The region of the ellipse bounded by 
the intersections $p_1$ and $p_2$ defines the segment $C_\rho$. 
The segment of $C_\rho$ not close to the points $p_1$ and $p_2$ is denoted as
$\tilde{C}_\rho$ and drawn in red. The segments of $C_\rho$ which are close
to the points $p_1$ and $p_2$ are denoted as $C_\rho \setminus \tilde{C}_\rho$
and drawn in blue.
The remainder of the ellipse is denoted as $E_\rho \setminus C_\rho$ and drawn 
in black. 
}
\label{fig:crho}
\end{figure}
Let $\tilde{C}_\rho \subset C_\rho$ denote the portion of $C_\rho$ 
no closer than $1/m^2$ from the
points $p_1$ and $p_2$, defined by
\begin{align} \label{eq:C_rho}
\tilde{C}_\rho = \{z: z \in C_\rho , |p_1 -z| > 1/m^2,  |p_2 -z| > 1/m^2 \}.
\end{align}
We split the integral in (\ref{eq:R_mquadrule}) into $\tilde{C}_\rho$
and $C_\rho$, arriving at
\begin{align} \label{eq:splitCrho}
\int_{C_\rho} \frac{T_m(v)}{v-z}dv =
\int_{\tilde{C}_\rho} \frac{T_m(v)}{v-z}dv + 
\int_{C_\rho \setminus \tilde{C}_\rho} \frac{T_m(v)}{v-z}dv.
\end{align}
The domain of integration $\tilde{C}_\rho$
is relatively well-separated
for all values of $z$ on which the quadrature rule in (\ref{eq:pwsplitme})
must hold.
In contrast, 
the domain of integration $C_\rho \setminus \tilde{C}_\rho$ is not
well-separated. Hence, we split the task of constructing the quadratures
on $C_\rho$ into two tasks.
\subsubsection{Quadratures for the Portion of $C_\rho$
Away From Gustafsson's Contours}
\label{se:quadfar}
Recall from formula (\ref{eq:C_rho}) that, by construction, 
$\tilde{C}_\rho$ is separated from the poles near its
end points by $1/m^2$.
Recall also from Section \ref{se:geometrybern} that, in the middle, $C_\rho$
is separated from the interval $[-1,1]$ by $\approx 1/m$. 
Hence, by Remark \ref{re:trefanalcont}, 
\begin{align} \label{eq:cifCt}
\int_{\tilde{C}_\rho} \frac{T_m(v)}{v-z} dv
\end{align}
is well approximated by an $O(m)$ Gauss-Legendre quadrature, 
for all $z \in [-1,1]$, 
$z \in \gamma_1 \cap E_{\tilde{\rho}}^o$, 
$z \in \gamma_2\cap E_{\tilde{\rho}}^o$, 
$z \in \mathbb{C} \setminus E_{\tilde{\rho}}^o$, where  
$E_{\tilde{\rho}}^o$ is the interior of $E_{\tilde{\rho}}$.
\subsubsection{Quadratures for the Portions of $C_\rho$ Near Gustafsson's Contours}
\label{se:quadnear}
In this section, we present the construction of a quadrature rule which approximates 
the contour integral
\begin{align} \label{eq:cifCEdge}
\int_{C_\rho\setminus \tilde{C}_\rho} \frac{T_m(v)}{v-z} dv,
\end{align}
for $z \in [-1,1]$, 
$z \in \gamma_1 \cap E_{\tilde{\rho}}^o$, 
$z \in \gamma_2\cap E_{\tilde{\rho}}^o$, 
$z \in \mathbb{C} \setminus E_{\tilde{\rho}}^o$, where  
$E_{\tilde{\rho}}^o$ is the interior of $E_{\tilde{\rho}}$.
Since $[-1,1]$ is well-separated from $C_\rho \setminus \tilde{C}_\rho$,
we focus only on $z \in (\gamma_1 \cap E_{\tilde{\rho}}) \cup
( \gamma_2  \cap E_{\tilde{\rho}})$.
Observe that $C_\rho \setminus \tilde{C}_\rho$ consists of two disjoint segments
(see Figure \ref{fig:crho}). 
One segment
of $C_\rho \setminus \tilde{C}_\rho$ 
approaches the point $p_1$, which denotes the intersection of
$\gamma_1$ (associated with $z=1$) with $E_\rho$, 
the other segment of  $C_\rho \setminus \tilde{C}_\rho$ approaches the
point $p_2$, which denotes the
intersection of $\gamma_2$ (associated with $z=-1$) with $E_\rho$.
We denote the points where $C_\rho\setminus \tilde{C}_\rho$ ends 
and $\tilde{C}_\rho$
begins with the points $\tilde{p}_1$ and $\tilde{p}_2$, where $\tilde{p}_1$
is the point closer to $p_1$, and $\tilde{p}_2$ is the point closer to
$p_2$.
We analyze the segment of $C_\rho \setminus \tilde{C}_\rho$ 
near $p_2$, with the understanding that the
Bernstein ellipse is symmetric and an identical argument applies to the 
segment of $C_\rho \setminus \tilde{C}_\rho$ near $p_1$.
\begin{figure}[h]
\centering
\includegraphics[width=0.48\textwidth]{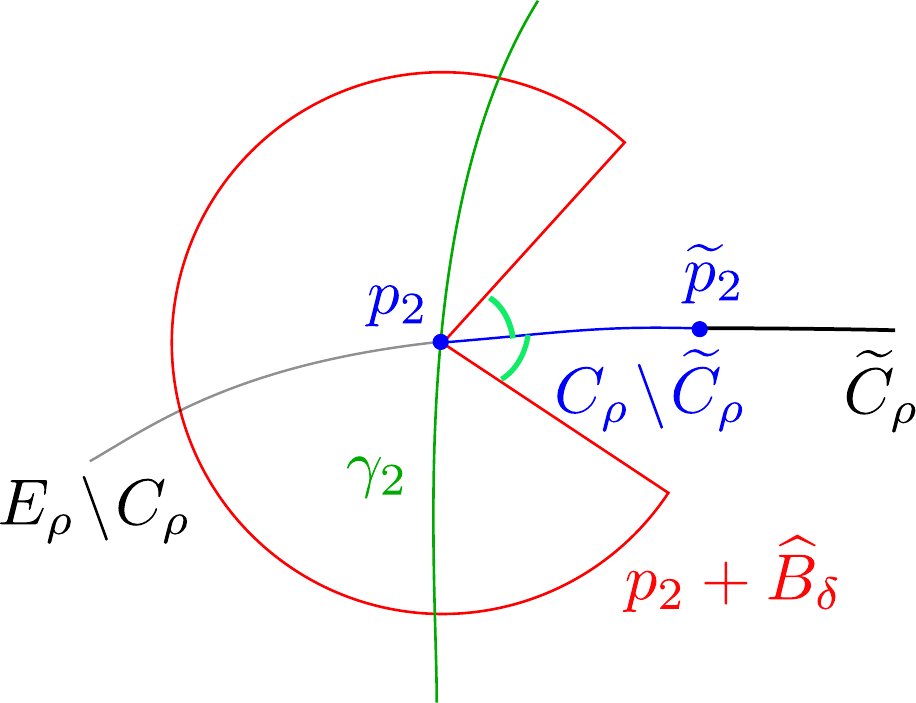}
\caption[Region in the vicinity of $C_\rho$'s endpoint where the quadrature 
formula must hold]{
{ \bf
Region $p_2 + \widehat{B}_\delta$ in which the quadrature in 
formula (\ref{eq:R_mquadrule2})
must accurately 
evaluate the integral over the contour $C_\rho\setminus \tilde{C}_\rho$
for $z\in\gamma_2$.}
The values of $z$ for which the quadrature 
must be accurate in the sense of formula (\ref{eq:l1v4}) are the
interior of the region denoted by $p_2 + \widehat{B}_\delta$, whose boundary is
drawn in red. Note that
the angle that $\widehat{B}_d$ makes with $C\rho\setminus \tilde{C}_\rho$ is
$\pi/6$ from above and $\pi/6$ from below.  
Gustafsson's contour which begins at $z=-1$ is denoted $\gamma_2$ and is drawn 
in green. The intersection of $\gamma_2$ with the Bernstein Ellipse is
denoted by $p_2$.
The Bernstein ellipse, $E_\rho$, is drawn as three contiguous segments.
The left segment, colored grey, corresponds to the Bernstein ellipse
which is not in $C_\rho$, and is denoted $E_\rho\setminus C_\rho$. 
The middle segment,
denoted $C_\rho\setminus \tilde{C}_\rho$, is drawn in blue.
The right segment,
colored black, corresponds to $\tilde{C}_\rho$. The point where
$C_\rho\setminus \tilde{C}_\rho$ ends and where $\tilde{C}_\rho$ begins is 
denoted by $\tilde{p_2}$.
}
\label{fig:leftBeta}
\end{figure}
\begin{figure}[h]
\centering
\includegraphics[width=0.30\textwidth]{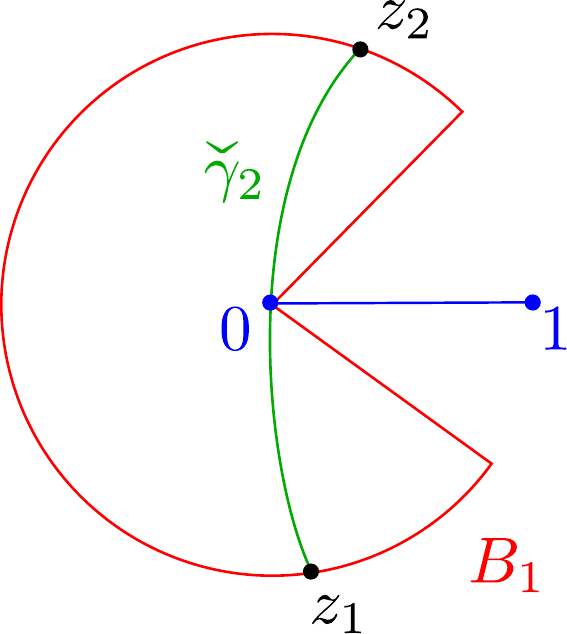}
\caption[Rescaling of a region where the quadrature must hold 
in the vicinity of $C_\rho$'s endpoint]
{ { \bf Rescaling and rotation of region of interest depicted in Figure \ref{fig:leftBeta}.}
Region $B_1$ is a translation, rotation, and rescaling of $\widehat{B}_\delta$ such
that $B_1$ has radius $1$.
For $z\in B_1$, the quadrature formula must satisfy (\ref{eq:l1}).
}
\label{fig:leftBeta1}
\end{figure}
We define $B_\delta$ as
\begin{align}
B_\delta = \big\{z: |\text{Arg}(z)|  \geq
\frac{\pi}{6}, |z| \leq \delta \big\}.
\end{align}
Recall from Section \ref{se:geometrybern} that $\gamma_2$, in the vicinity
of $p_2$, always lies in $p_2+\widehat{B}_\delta$, where
$\widehat{B}_\delta$ is a rotated version of $B_\delta$, such that
the opening in $B_\delta$ is bisected by $C_\rho$ (see Figure \ref{fig:leftBeta}).
We note that, due to the same argument in Section \ref{se:quadfar},
for $z$ outside of $B_\delta$ but elsewhere where the quadrature
must hold, $z$ is well-separated from the domain of integration 
$C_\rho \setminus \tilde{C}_\rho$
, meaning that
a Gauss-Legendre quadrature accurately approximates the
integral. Hence, for the remainder of this section we exclusively focus
on developing a quadrature rule which approximates (\ref{eq:cifCEdge})
for $z \in p_2 + \widehat{B}_\delta$.

For convenience, we rotate, translate, and rescale $C_\rho\setminus
\tilde{C}_\rho$ and
$p_2 + \widehat{B}_\delta$
(see Figure \ref{fig:leftBeta1}), so that the segment 
$C_\rho \setminus \tilde{C}_\rho$ is by approximated by the 
interval $[0,1]$ (i.e., it is translated by $p_2$, rotated, and
scaled by a factor of $m^2$).
Likewise, $\widecheck{\gamma_2}$ represents a similarly translated,
rotated, and scaled copy of $\gamma_2$.
Note that we associate $p_2$ with the point $x=0$ and the point $\tilde{p}_2$
with $x=1$
(see Figure \ref{fig:leftBeta1}).
Consider a quadrature rule
$x_i,\ldots ,x_n$ and
$w_1,\ldots ,w_n$ such that
\begin{align} \label{eq:pointwiseAway}
\bigg| \int_0^1 \frac{\rho(x)}{x-z}dx 
  - \sum_{i=1}^n \frac{\rho(x_i)}{x_i - z}w_i\bigg| < \epsilon
\end{align}
for all $z \in B_1$, where $\rho(x)$ is smooth. 
Such a quadrature, if used to approximate (\ref{eq:cifCEdge}), will
be accurate to precision $\epsilon$, for all $z \in p_2 + \widehat{B}_\delta$.
However, recall that we are integrating $R_m(z)$ given by (\ref{eq:RmzIntme})
over $\gamma_1 \cap E_{\tilde{\rho}}^o$ and $\gamma_2 \cap E_{\tilde{\rho}}^o$.
In the rotated and rescaled coordinates, this means that we are integrating
$z \in \widecheck{\gamma}_2 \subset B_1$, where $\widecheck{\gamma}_2$ 
starts at $z_1 \in \partial B_1$ and ends at $z_2 \in \partial B_1$,
with $|z_1| = |z_2| = 1$.
We thus relax the requirement (\ref{eq:pointwiseAway}) to hold in 
$L^1(\widecheck{\gamma}_2 \cap B_1)$, meaning that the integral
and the quadrature approximation in (\ref{eq:pointwiseAway}) can 
disagree on a set of measure $\epsilon$.

This allows us to relax (\ref{eq:pointwiseAway}) to the condition
\begin{align} \label{eq:l1}
\bigg| \int_0^1 \frac{\rho(x)}{x-z}dz 
  - \sum_{i=1}^n \frac{\rho(x_i)}{x_i - z}w_i\bigg| < \frac{\epsilon}{|z|},
\end{align}
for $z \in B_1$. Thus, for each $\delta > 0$, if $z \in B_\delta$, 
then the quadrature is accurate to within an error $\epsilon/\delta$.
Since the length of $\widecheck{\gamma}_2 \cap B_\delta$ is on the order $\delta$,
the $L^1$ error in the quadrature is $\delta \cdot \epsilon/\delta = \epsilon$.

We can construct this quadrature by first sampling $z_i \in \partial B_1 $,
and then computing a generalized Gaussian quadrature (see \cite{bremer2})
on $x \in [0,1]$,
where (\ref{eq:l1}) is enforced on all the sampled $z_i$'s.
By Cauchy's theorem, if (\ref{eq:l1}) holds on $\partial B_1$, then
it will also hold on $B_1$. However, this still results in a quadrature rule
with several hundred nodes. It turns out that far fewer nodes can be used, due
to the following observation.

Recall that, in the integrand of (\ref{eq:mgfzQuad2}), $R_m(z)$ is multiplied 
by a spherical wave term, denoted as $H(z)$, given by
\begin{align} \label{eq:HSphere2}
H(z) = \frac{e^{-i\kappa \sqrt{1 - \alpha z}} }
  {\sqrt{1-\alpha z} \sqrt{1-z^2} } . 
\end{align}
Because $H(z)$ is smooth near $z=p_2$, we only need that 
\begin{align} \label{eq:l1v2}
\bigg| \int_{\widecheck{\gamma}_2\cap B_1} \sigma(z) \int_0^1 
    \frac{\rho(x)}{x-z}dx dz 
-   \int_{\widecheck{\gamma}_2\cap B_1} \sigma(z) 
    \sum_{i=1}^n \frac{\rho(x_i)}{x_i -z}w_i dz \bigg| < \epsilon,
\end{align}
for all sufficiently smooth functions $\sigma(z)$.
Since $\sigma(z)$ is smooth, it can be represented by a Taylor series
of a small order $k$, so that
\begin{align}
\sigma(z) \approx \sum_{j=0}^k a_j z^j.
\end{align}
Thus, inequality (\ref{eq:l1v2}) becomes
\begin{align} \label{eq:l1v3}
\bigg| \int_{\widecheck{\gamma}_2\cap B_1} z^j \int_0^1 \frac{\rho(x)}{x-z}dx dz 
-   \int_{\widecheck{\gamma}_2\cap B_1} z^j 
    \sum_{i=1}^n \frac{\rho(x_i)}{x_i -z}w_i dz \bigg| < \epsilon,
\end{align}
for each $j = 0,1,\ldots ,k$. Exchanging the order of integration,
\begin{align} \label{eq:l1v4}
\bigg| \int_0^1 \rho(x) \int_{\widecheck{\gamma}_2\cap B_1} \frac{z^j}{x-z}dz dx 
-    \sum_{i=1}^n  \rho(x_i)  \int_{\widecheck{\gamma}_2\cap B_1} 
 \frac{z^j}{x_i -z} dz\, w_i \bigg| < \epsilon.
\end{align}
Recall from Section \ref{se:contourpoly} that 
\begin{align} 
\int_{\widecheck{\gamma}_2\cap B_1} \frac{z^j}{x-z}dz =
    \phi(x) + \psi(x) \log\Big(\frac{x-z_1}{x-z_2} \Big),
\end{align}
where $\phi$ and $\psi$ are polynomials of order $j$, and $z_1$ and $z_2$
are the endpoints of $\widecheck{\gamma}_2\cap B_1$.
Due to the geometry of $B_1$, we have that 
\begin{align}
|z_1 - x| \geq \frac{1}{2},  \qquad |z_2 -x| \geq \frac{1}{2},
\end{align}
for all $x \in [0,1]$.
We also observe that the branch cut of 
\begin{align} \label{eq:logsmooth}
\log \bigg( \frac{x-z_1}{x-z_2} \bigg)
\end{align}
does not intersect $[0,1]$, so (\ref{eq:logsmooth})  smooth on $[0,1]$.
Since $\rho(x)$ is smooth and (\ref{eq:logsmooth}) is smooth,
we observe that the integrand in (\ref{eq:l1v4}), given by
\begin{align}
\rho(x) \int_{\widecheck{\gamma}_2\cap B_1} \frac{z^j}{x-z}dz,
\end{align}
is a smooth function of $x$ for $x \in [0,1]$. 
Hence, a Gauss-Legendre quadrature with $O(1)$ points will satisfy (\ref{eq:l1v4}).

Because (\ref{eq:l1v4}) is satisfied, (\ref{eq:l1v2}) is satisfied, and so
the contour deformation argument 
presented in Section \ref{se:deformation} can be carried out without change,
using a Gauss-Legendre quadrature with $O(1)$ points 
on $C_\rho\setminus \tilde{C}_\rho$ and $O(m)$ points on 
$\tilde{C}_\rho$.

\subsubsection{The Error in the Approximation $R_m(z)$}
\label{sec:errrmz}

In order to derive the approximation~(\ref{eq:mgfzQuad}) to the Green's
function $G_m$, we approximated the Chebyshev polynomial $T_m(z)$ by the
function $R_m(z)$, defined by~(\ref{eq:RmzIntme}) (see
also~(\ref{eq:mgfTF2}) and~(\ref{eq:mgfRF2})). Recall from
Section~\ref{se:bernsteinrule} that, when $\rho=M^{1/m}$, we have that
$\abs{T_m(z)} \approx M$ for all $z\in E_\rho$. If the formula for $R_m(z)$
is evaluated numerically, then the integrand and summand in that formula
will both have size approximately $M$, while the sum, $R_m(z)$, will have
size approximately one for $z\in [-1,1]$.  Thus, due to cancellation error,
$\abs{R_m(z) - T_m(z)} \approx M\epsilon$  for all $z\in [-1,1]$, where
$\epsilon$ is equal to machine precision. This means that, for
$\rho=M^{1/m}$, the approximation for $G_m$ given by
formula~(\ref{eq:mgfzQuad}) has an error of $M \epsilon$.

\subsubsection{The Number of Quadrature Nodes on $\tilde{C}_\rho$}
\label{sec:numquadr}

In Section~\ref{se:quadnear}, we demonstrated that only $O(1)$ nodes are
required on $C_\rho \setminus \tilde{C}_\rho$. In Section~\ref{se:quadfar},
we showed that $O(m)$ nodes are required on $\tilde{C}_\rho$ by pointing out
that the distance from $\tilde{C}_\rho$ to the nearest pole is $1/m^2$ at
its endpoints and $\approx 1/m$ in the middle. We then used
Corollary~\ref{co:approx} and Remark~\ref{re:trefanalcont} to state that the
number of terms required to expand the integrand in~(\ref{eq:cifCt}) in
Chebyshev polynomials is $O(m)$, which means that $O(m)$ nodes are
needed in the corresponding quadrature formula.

In fact, Corollary~\ref{co:approx} and Remark~\ref{re:trefanalcont}  provide
a quantitative estimate for how many terms are required. If a function on a
contour of length $2$ is analytic and bounded by $L$ on a region containing
the contour, where the boundary of the region is separated from the contour
by a distance of $(\log(M)/m)^2$ at the endpoints and $\log(M)/m$ in the
middle, then the number of Chebyshev expansion coefficients required to
approximate that function on the contour to precision $\epsilon$ is
  \begin{align}
k_0 \approx m{(\log(2L) - \log(\epsilon))/\log(M)}.
    \label{k0ez}
  \end{align}

When $\rho=M^{1/m}$, a straightforward modification of the argument in
Section~\ref{se:quadfar} shows that the distance from the contour
$\tilde{C}_\rho$ to the nearest pole is $(\log(M)/m)^2$ at its endpoints and
$\approx \log(M)/m$ in the middle.  Thus, if we take a slightly smaller
region, say, $90\%$ the size, then $L\approx 10 M$. Replacing $\epsilon$ in
formula~(\ref{k0ez}) by $M\epsilon$, since this is the minimum error we can
hope to achieve (see Section~\ref{sec:errrmz}), the estimate for the number
of terms in the Chebyshev expansion of the integrand of~(\ref{eq:cifCt})
becomes
  \begin{align}
k_0 \approx \frac{m}{0.9}{(\log(2\cdot 10M) - \log(M\epsilon))/\log(M)},
  \end{align}
which simplifies to
  \begin{align}
k_0 \approx 1.11 m {(\log(20) - \log(\epsilon))/\log(M)}.
  \end{align}
Taking $\epsilon = 10^{-16}$, we compute the number of terms $k_0$ in the
Chebyshev expansion required to approximate the integrand
of~(\ref{eq:cifCt}) to precision $M\epsilon$ for
$M=10,100,\ldots,10^{12}$ (see Table~\ref{ta:k0needed}).  Likewise,
taking $\epsilon = 10^{-34}$, we compute the number of terms $k_0$ for
$M=10,100,\ldots,10^{15}$ (see Table~\ref{ta:k0needed_quad}). We note
that, if $k_0$ Chebyshev expansion coefficients are required to approximate
the integrand, then the integral~(\ref{eq:cifCt}) can be evaluated using
a Gauss-Legendre quadrature with approximately $k_0/2$ points, since that a
Gauss-Legendre quadrature integrates approximately twice as many polynomials
as the number of quadrature points.

\begin{table}[H]
\begin{center}
\begin{tabular}{cccc}
\multirow{2}{*}{$$}
$M$& $M\epsilon$ & $k_0$ & $k_0/2$ \\
\midrule
\addlinespace[.5em]
10   & $10^{-15}$ & 19.2m & 9.6m \\
\addlinespace[.25em]
100   & $10^{-14}$ & 9.6m & 4.8m \\
\addlinespace[.25em]
$10^3$   & $10^{-13}$ & 6.4m & 3.2m \\
\addlinespace[.25em]
$10^6$   & $10^{-10}$ & 3.2m & 1.6m \\
\addlinespace[.25em]
$10^9$   & $10^{-7}$ & 2.13m & 1.07m \\
\addlinespace[.25em]
$10^{12}$   & $10^{-4}$ & 1.6m & 0.8m
 
\end{tabular}
\caption[The required number of Gauss-Legendre nodes on $\tilde{C}_\rho$ to
approximate (\ref{eq:cifCt}), in double precision]{
{\bf The required number of Gauss-Legendre nodes on $\tilde{C}_\rho$
to approximate ~(\ref{eq:cifCt}), in double precision.} In this table, 
$k_0/2$ is the required number of
nodes, $\epsilon=10^{-16}$, and $\rho=M^{1/m}$.
}
\label{ta:k0needed}
\end{center}
\end{table}

\begin{table}[H]
\begin{center}
\begin{tabular}{cccc}
\multirow{2}{*}{$$}
$M$& $M\epsilon$ & $k_0$ & $k_0/2$ \\
\midrule
\addlinespace[.5em]
10   & $10^{-33}$ & 39.2m & 19.6m \\
\addlinespace[.25em]
100   & $10^{-32}$ & 19.6m & 9.8m \\
\addlinespace[.25em]
$10^3$   & $10^{-31}$ & 13.1m & 6.53m \\
\addlinespace[.25em]
$10^6$   & $10^{-28}$ & 6.53m & 3.27m \\
\addlinespace[.25em]
$10^9$   & $10^{-25}$ & 4.35m & 2.18m \\
\addlinespace[.25em]
$10^{12}$   & $10^{-22}$ & 3.27m & 1.63m \\
\addlinespace[.25em]
$10^{15}$   & $10^{-19}$ & 2.61m & 1.31m
 
\end{tabular}
\caption[The required number of Gauss-Legendre nodes on $\tilde{C}_\rho$ to
approximate (\ref{eq:cifCt}), in quadruple precision]{
{\bf The required number of Gauss-Legendre nodes on $\tilde{C}_\rho$
to approximate~(\ref{eq:cifCt}), in quadruple precision.} In this table, $k_0/2$ is
the required number of nodes, $\epsilon=10^{-34}$, and $\rho=M^{1/m}$.
}
\label{ta:k0needed_quad}
\end{center}
\end{table}

Finally, we observe that, in practice, we can place a single $O(m)$
Gauss-Legendre quadrature with $k_0/2$ nodes on the entire contour $C_\rho$, rather than
placing two $O(1)$ quadratures on each part of $C_\rho\setminus
\tilde{C}_\rho$ and one $O(m)$ quadrature on $\tilde{C}_\rho$.  Also, we
note that, in practice, the minimum number of quadrature nodes required to
achieve the accuracy $M\epsilon$ matches the estimates in
Tables~\ref{ta:k0needed} and~\ref{ta:k0needed_quad} very closely.

\subsection{Summary of the Algorithm}
\label{se:sumalg}

Recall from Section \ref{sec:mgf} that $G_m$ is a function 
of $\kappa$, $m$, and $\alpha$. 
Recall also that $\alpha$ can be determined
from $\beta_-$ (see formula \ref{eq:betaBoth})), and vice versa. 
We consider $G_m$ as
a function of
$\kappa$, $m$, and $\beta_1$.
We compute $G_m$ as follows.
Recall from Section \ref{se:evalalpha} the formula for $G_m$,
\begin{align} 
\begin{split} \label{eq:mastersum}
G_m \approx 
\frac{-4}{\sqrt{\alpha}} \int_0^{\tau_1} \frac{F_1(\tau)}
{\sqrt{\tau^2 +2i \beta_-}} d\tau + 
\frac{4i}{\sqrt{\alpha}} \int_0^{\tau_2} \frac{F_2(\tau)}
  {\sqrt{\tau^2 +2i \beta_+}} d\tau\\
  + \sum_{i=1}^n 
       \frac{e^{-i\kappa \sqrt{1-\alpha v_i}}}
            {\sqrt{1-\alpha v_i}\sqrt{1-v_i^2} } 
            T_m(v_i) dv_i w_i ,
            \end{split}
\end{align}
where $F_1(\tau)$ and $F_2(\tau)$ are smooth functions corresponding to the 
$\gamma_1$ and $\gamma_2$ contours, respectively
defined by (\ref{eq:smoothF1}) and (\ref{eq:smoothF2}), 
$\tau_1$ and $\tau_2$ are positive parameters such that  
$\gamma_1(\tau_1)$ and $\gamma_2(\tau_2)$ intersect $E_\rho$
(see Section \ref{se:intersect})
, respectively, 
$T_m$ is the $m$th order Chebyshev polynomial,
and  
\begin{align}
\beta_- = \sqrt{1/\alpha -1} , \qquad \beta_+ = \sqrt{1/\alpha +1}.
\end{align}
Recall from Section \ref{se:geometrybern} that both $\gamma_1 \cap E_\rho$
and $\gamma_2 \cap E_\rho$ 
have length $\approx 1/m^2$. Hence, $T_m(z)$ oscillates at most once along
each contour.
By construction, on Gustafsson's contours 
(see Section \ref{se:gustafsson}), the spherical wave portion of the
integrand does not oscillate. Hence, the entire integrand oscillates at
most once.
By the argument in Section \ref{se:evalalpha} the integrand associated 
with the $\gamma_2$
contour is always smooth and hence can be evaluated with 
an $O(1)$ Gauss-Legendre
quadrature.

The integrand associated with the contour $\gamma_1$ has a 
singularity for $\beta_- \approx 0$. For this case, we follow
the method in Section \ref{se:evalalpha} and evaluate the portion near the
singularity by expanding the function $F_1(\tau)$ into its Taylor series,
then use the recurrence 
described in Section \ref{se:evalgamma1}. Due to the smoothness of $F_1(\tau)$,
this integral is computed with an $O(1)$ Gauss-Legendre quadrature.
The remainder of
the integral is smooth and oscillates at most once, and hence is evaluated
with an $O(1)$ Gauss-Legendre quadrature. Hence, both integrals in 
(\ref{eq:mastersum}) are evaluated in $O(1)$ operations.

The remaining term in (\ref{eq:mastersum}) is a sum of residues evaluated
on $C_\rho$, where $C_\rho$ denotes the portion of a Bernstein ellipse connecting
$\gamma_1$ and $\gamma_2$ (see Section \ref{se:choiceRat}).
We select the residues $v_1,\ldots,v_n$ and weights $w_1,\ldots,w_n$ by
constructing a quadrature which approximates
\begin{align}
\oint_{C_\rho} \frac{T_m(v)}{v-z} dv ,
\end{align}
which holds for values of $z$ relevant to the evaluation of $G_m$
(see Section \ref{se:quadnear}).
By the argument in Section \ref{se:constructboth}, this is accomplished
using $O(m)$ Gauss Legendre nodes on $C_\rho$.

Therefore, the entire cost of our algorithm for $G_m$ is $O(m)$ and completely
independent of both $\kappa$ and $\beta_-$. Lastly, since the algorithm
is entirely quadrature based, it is
embarrassingly parallelizable. Finally, we note that implementing 
this algorithm requires certain numerical issues to be treated with care,
which we describe in
Section \ref{se:numericmisc}.
\subsection{Numerical Miscellanea}
\label{se:numericmisc}
This section contains various facts required for the accurate evaluation of
some
of the quantities and formulas used by the numerical algorithm of this
manuscript.
\subsubsection{Evaluating the Semi-major and Semi-minor Axes of the Bernstein
Ellipse}
We will need to compute the quantities $a-1$ and $b$, where $a$ is the
semi-major axis of the Bernstein ellipse $E_\rho$ described in 
Section~\ref{se:chebyshevEvalOnBE}
and $b$ is the semi-minor axis. When $\rho \approx 1$, we have that $a\approx
1$, so computing $a-1 \approx 0$ directly from $a$ will result in a large
cancellation error. Likewise, when $\rho \approx 1$, we have that $b \approx
0$, so computing it from formula~(\ref{eq:bern}) will also result in cancellation
error.  Instead of computing
  \begin{align}
\rho = M^{\frac{1}{m}},
  \end{align}
as in formula~(\ref{eq:rhoChoice}), we instead compute the value of the semi-minor
axis $b$ directly using the formula
  \begin{align}
&\hspace*{-2em} 
b = \sinh(\log(M)/m) = \frac{1}{2} (e^{\log(M)/m} - e^{-\log(M)/m})
= \frac{1}{2} (\rho - \frac{1}{\rho}).
  \end{align}
which can be done stably even when $m$ is very large.  We then compute the
value of the semi-major axis $a$ from $b$ using the formula
  \begin{align}
a=\sqrt{b^2 + 1} = \cosh(\log(M)/m) = \frac{1}{2}(\rho + \frac{1}{\rho}).
  \end{align}
The value of $a-1$ is also given by the formula
  \begin{align}
a-1 = \frac{b^2}{a+1},
  \end{align}
and is likewise derived from identities involving the hyperbolic functions. 
\subsubsection{The Evaluation of $\beta_-$ when $\alpha \approx 1$}
When $\alpha\approx 1$, the quantity $\beta_- = \sqrt{1/\alpha - 1}$ will be
computed with a very large cancellation error. Thus, instead of using
$\alpha$ as an input parameter to our algorithm, we use $\beta_-$.  The
quantity $\alpha$ can be obtained from $\beta_-$ by the formula
$\alpha=1/(\beta_-^2 +1)$, and $\beta_-$ can be evaluated to full relative
precision from~(\ref{eq:beta_-eff}).

\subsubsection{The Evaluation of the Quantity $\sqrt{1-\alpha z}$ when $\alpha
\approx 1$ and $z\approx 1$}

Sometimes we will need to evaluate the quantity $\sqrt{1-\alpha z}$ on 
Gustafsson's contours when $\alpha \approx 1$ and $z \approx 1$. As mentioned in
Section~\ref{sec:mgf}, we use $\beta_-$ as an input parameter to prevent a loss of
accuracy. With the parameterization $z=\tilde \gamma_{1}(\tau)$ of the
contour $\gamma_1$, given by~(\ref{eq:gamma_1kirill}), we have
  \begin{align}
\sqrt{1-\alpha \tilde \gamma_{1}(\tau)} = -i \sqrt{\alpha}(\tau^2 +
i\beta_{-}),
  \end{align}
as stated in~(\ref{eq:dzsqrt}). This formula can be evaluated to relative precision when
$\beta_-\approx 0$ and $\tau \approx 0$ (equivalently, when $\alpha \approx
1$ and $z \approx 1$).

\subsubsection{The Evaluation of the Intersection Points $p_1$ and $p_2$}

In Section~\ref{se:intersect}, we determine the intersection points of the Gustafsson
contours $\gamma_1$ and $\gamma_2$ with the Bernstein ellipse $E_\rho$, in
both the Bernstein ellipse parameter $\theta$ and Gustafsson's contours'
parameter $\tau$.  These formulas all involve solving a quadratic equation.
To solve it accurately, we use the observation in Section~\ref{se:quadraticeq}.

\subsubsection{The Evaluation of $\arccos(s)$ for $s\approx 1$}

In the construction of the intersection points of the Gustaffson contours
$\gamma_1$ and $\gamma_2$ with the Bernstein ellipse $E_\rho$, in the
Bernstein ellipse parameter $\theta$, it is sometimes the case that $\theta
= \arccos(s) \approx 0$ and $s\approx 1$ in formula~(\ref{eq:intersectTheta}).  
The condition
number of $\arccos(s)$ becomes infinite near $s = 1$, so a straightforward
application of the formula results in a loss of accuracy.  We observe that
the function $\arccos(1+z)$ can be evaluated accurately for $z\approx 0$ (by,
for example, Taylor series). Thus, instead of solving the quadratic equation
for $s$, we solve for $s-1$, and then evalate $\arccos(1+z)$ for $z=s-1$.

\subsubsection{The Evaluation of $T_m(z)$ when $z\approx \pm 1$}

Since we use the parameterizations $z=\tilde \gamma_1(\tau)$ 
and $z=\tilde \gamma_2(\tau)$ for 
Gustafsson's contours, we are able to 
evaluate $z-1$ and $z+1$
to full relative precision on  $\gamma_1$ and $\gamma_2$, respectively.
However, the formula
  \begin{align}
T_m(z) = \cos(m \arccos(z))
  \end{align}
requires the evaluation of $\arccos(z)$ near $z=1$, where its condition
number is infinite. Instead, we observe that $\arccos(1+z)$ can be evaluated
to full relative accuracy near $z=0$ (using, for example, Taylor series). 
Thus, we evaluate 
  \begin{align}
T_m(z) = \cos(m \arccos(1+w))
  \end{align}
accurately for $w=z-1 \approx 0$ with $z\in \gamma_1$. Likewise, we use the fact that
$\arccos(-z)=\pi-\arccos(z)$ to evaluate
  \begin{align}
T_m(z) = (-1)^m \cos(m \arccos(1+w))
  \end{align}
accurately for $w=-z-1 \approx 0$ with $z \in \gamma_2$.

\subsubsection{The Limits of Integration on Gustafsson's Contours}
\label{se:limitsofint}

In order to approximate the modal Green's function $G_m$ using
formula~(\ref{eq:mastersum}), it is necessary to evaluate the integrals
  \begin{align} 
&\hspace*{-4em} \int_0^{\tau_1} \frac{F_1(\tau)}
{\sqrt{\tau^2 +2i \beta_-}} d\tau, \qquad \text{with} \qquad 
F_1(\tau)=\frac{e^{-i\kappa \sqrt{1 - \alpha \tilde{\gamma_1}(\tau)}}}
  {\sqrt{1+\tilde{\gamma_1}(\tau)}} T_m(\tilde{\gamma_1}(\tau)),
    \label{eq:F1lim}
  \end{align}
and
  \begin{align}
&\hspace*{-4em} \int_0^{\tau_2} \frac{F_2(\tau)}
  {\sqrt{\tau^2 +2i \beta_+}} d\tau, \qquad \text{with} \qquad 
F_2(\tau) = \frac{e^{-i\kappa \sqrt{1 -\alpha\tilde{\gamma_2}(\tau)}}}
  {\sqrt{1+\tilde{\gamma_2}(\tau)}} T_m(\tilde{\gamma_2}(\tau)),
    \label{eq:F2lim}
  \end{align}
where $\tilde\gamma_1(\tau_1)$ and $\tilde\gamma_2(\tau_2)$ are,
respectively, the intersection points of $\gamma_1$ and $\gamma_2$ with
$E_\rho$ (see~(\ref{eq:smoothF1}) and~(\ref{eq:smoothF2}).  The integrands
decay exponentially in $\tau$ at a rate proportional to $\kappa$. Thus, when
$\kappa$ is large, care must be taken to choose the domains of integration
when evaluating the integrals numerically.

In order to evaluate the integrals~(\ref{eq:F1lim}) and~(\ref{eq:F2lim}) to
within an error of $M\epsilon$ (see Section~\ref{sec:errrmz}), the integrals
only need to be evaluated over values of $\tau$ for which $F_1(\tau) \ge
M\epsilon$ and $F_2(\tau) \ge M\epsilon$, respectively. 
Since $\tilde\gamma_1(\tau) \in E_\rho^o$ for all $\tau \in [0,\tau_1)$ and
$\tilde\gamma_2(\tau) \in E_\rho^o$ for all $\tau \in [0,\tau_2)$,
by~(\ref{eq:chebyBernIn3}), it follows that, when $\rho=M^{1/m}$,
$\abs{T_m(\tilde\gamma_1(\tau))} < M$ for $\tau \in [0,\tau_1)$ and
$\abs{T_m(\tilde\gamma_2(\tau))} < M$ for $\tau \in [0,\tau_2)$. 
We observe then that
  \begin{align}
F_1(\tau) \approx M e^{-i \kappa \sqrt{1-\alpha \tilde\gamma_1(\tau)}}
\qquad \text{and} \qquad
F_2(\tau) \approx M e^{-i \kappa \sqrt{1-\alpha \tilde\gamma_2(\tau)}}.
  \end{align}
By~(\ref{eq:dzsqrt2}),
  \begin{align}
\sqrt{1-\alpha \tilde\gamma_1(\tau)} = -i\sqrt{\alpha} (\tau^2 + i\beta_-), 
  \end{align}
and, likewise,
  \begin{align}
\sqrt{1-\alpha \tilde\gamma_2(\tau)} = -i\sqrt{\alpha} (\tau^2 + i\beta_+). 
  \end{align}
Thus,
  \begin{align}
\abs{e^{-i \kappa \sqrt{1-\alpha \tilde\gamma_1(\tau)}}}
= \abs{e^{-\kappa \sqrt{\alpha}\tau^2}},
  \end{align}
and
  \begin{align}
\abs{e^{-i \kappa \sqrt{1-\alpha \tilde\gamma_2(\tau)}}}
= \abs{e^{-\kappa \sqrt{\alpha}\tau^2}}.
  \end{align}
Solving the equation
  \begin{align}
e^{-\kappa \sqrt{\alpha}\tau^2} = \epsilon
  \end{align}
for $\tau$, we arrive at the formula
  \begin{align}
\tau_c = \sqrt{ \frac{-\log(\epsilon)}{\kappa\sqrt{\alpha}} },
  \end{align}
from which we see that $F_1(\tau_c) \approx M\epsilon$ and $F_2(\tau_c)
\approx M\epsilon$. Thus, we evaluate the integrals~(\ref{eq:F1lim})
and~(\ref{eq:F2lim}) over Gustafsson's contours only on the intervals
$[0,\min(\tau_1,\tau_c))$ and $[0,\min(\tau_2,\tau_c))$, respectively.
Hence, rather than evaluate (\ref{eq:F1lim}) and (\ref{eq:F2lim}),
we instead evaluate
  \begin{align}
&\hspace*{-4em} \int_0^{\min(\tau_1, \tau_c)} \frac{F_1(\tau)}
{\sqrt{\tau^2 +2i \beta_-}} d\tau  \qquad \text{and} \qquad
\int_0^{\min(\tau_2,\tau_c)} \frac{F_2(\tau)}
  {\sqrt{\tau^2 +2i \beta_+}} d\tau.
 \end{align}
\section{Numerical Experiments}
\label{se:numex}
In Sections \ref{se:performbetam}-\ref{se:parallel}
we characterize the speed and accuracy of our method.
Importantly, as demonstrated below, we achieve full precision
for all possible ranges of 
$\beta_-$ and $\kappa$, and our algorithm's performance is
completely independent of $\beta_-$ and $\kappa$.
 
We use adaptive integration applied to (\ref{eq:mgf}) as the gold standard,
and measure the error of our algorithm by comparing the two results.
We use the change of variables $\phi = x^3$, $d\phi = 3x^2 dx$, to ensure
that adaptive integration is accurate when $\alpha \approx 1$.
We compute the $1-\alpha \cos(\phi)$ term using the double angle formula
to avoid cancellation error. 
The error in evaluating the modal Green's function for very large
$\kappa$ is not measured, as adaptive integration is too expensive
and no prior method can compute the modal 
Green's function for large $\kappa$.

An implementation of the previously described algorithm was written in
Fortran 77.  In our implementation, we chose $M=100$, and used $5m$
quadrature nodes on $C_\rho$ in double precision, and $11m$ quadrature nodes
on $C_\rho$ in extended precision (see Section~\ref{sec:numquadr}).
The timing and performance experiments in Sections
\ref{se:performbetam}--\ref{se:performfourier}
were performed using a consumer
laptop with a four-core 2.6 GHz
Intel i7 processor running a timing script in MATLAB 2018b with two threads.
The parallel computing experiment
in Section \ref{se:parallel} was run on a server with a 16-core Intel
Xeon 2.9 GHz processor.
\subsubsection{The Interpretation of $\beta_-$ and $\kappa$}
Recall from Section \ref{sec:mgf} that the modal Green's function 
can be thought of as a function of four parameters: $m$, $k$, $\alpha$, and $R_0$.
After the introduction of the parameters $\kappa$ and $\beta_-$ (see 
formula (\ref{eq:mgfexp})), the $R_0$ 
term exclusively appears as a $1/R_0$ scaling outside the integral. 
Hence, with this parameterization, $R_0$ is of no independent consequence to the
performance of our algorithm, so we only characterize our algorithm's performance
as a function of $\kappa$, $\beta_-$, and $m$.
Recall also that $\beta_-$ is defined as
\begin{align} \label{eq:beta_-eff2}
\beta_- = \frac{\Delta}{\rho_0},
\end{align}
where $\Delta$ is the minimum source-to-target distance and
$\rho_0 = 2rr'$, with $r$ and $r'$ being the radial distances of the 
source and target in cylindrical coordinates.
Recall finally from Section \ref{sec:mgf} that $\kappa$ is defined as
\begin{align}
\kappa = kR_0.
\end{align}

\subsection{Performance of the Algorithm with Varying Source-to-Target Distance}
\label{se:performbetam}
We examined the performance of our algorithm over a wide range of source-to-target
distances.
As shown in Table \ref{ta:betam1double} and Table \ref{ta:betam1quad}, our
algorithm's performance is independent of $\beta_-$.

%%%%%%%%%%%%%%%%%%%%%%%%%%%%%%%%%%%%%%%%%%%%%%%%%%
%   Varying Beta, double
%%%%%%%%%%%%%%%%%%%%%%%%%%%%%%%%%%%%%%%%%%%%%%%%%

\vspace{0.25in}
\begin{table}[H]
\begin{center}
\begin{tabular}{ccccc}
 & \multicolumn{2}{c}{$\kappa=10,000,\,m=10$} & \multicolumn{2}{c}{$\kappa=10,000,\,m=1000$}\\
\multirow{2}{*}{$$} 
%%%&  Function  & Absolute Error & Avg. Modal Function  & Absolute Error\\
$\beta_-$& Evaluation Time      &  Absolute Error               & Evaluation Time      & Absolute Error           \\
\midrule
\addlinespace[.5em]
$10^{15}$   & 4.66\e{-5} secs    & 1.47\e{-13}   & 1.33\e{-3} secs    & 7.29\e{-13} \\
\addlinespace[.25em]
$10^{12}$   & 4.76\e{-5} secs    & 1.53\e{-13}   & 1.34\e{-3} secs    & 7.29\e{-13} \\
\addlinespace[.25em]
$10^{9}$   & 4.73\e{-5} secs    & 1.46\e{-13}   & 1.34\e{-3} secs    & 7.29\e{-13} \\
\addlinespace[.25em]
$10^{6}$   & 5.00\e{-5} secs    & 2.55\e{-11}   & 1.41\e{-3} secs    & 2.11\e{-12} \\
\addlinespace[.25em]
$10^{3}$   & 4.89\e{-5} secs    & 6.03\e{-12}   & 1.41\e{-3} secs    & 2.64\e{-12} \\
\addlinespace[.25em]
$10^{0}$   & 3.76\e{-5} secs    & 3.34\e{-14}   & 1.44\e{-3} secs    & 4.71\e{-13} \\
\addlinespace[.25em]
$10^{-3}$   & 3.57\e{-5} secs    & 3.43\e{-14}   & 1.44\e{-3} secs    & 1.79\e{-12} \\
\addlinespace[.25em]
$10^{-6}$   & 3.51\e{-5} secs    & 3.92\e{-14}   & 1.44\e{-3} secs    & 3.43\e{-13} \\
\addlinespace[.25em]
$10^{-9}$   & 3.59\e{-5} secs    & 5.50\e{-14}   & 1.44\e{-3} secs    & 5.27\e{-13} \\
\addlinespace[.25em]
$10^{-12}$   & 3.52\e{-5} secs    & 3.33\e{-14}   & 1.44\e{-3} secs    & 5.28\e{-13} \\
\addlinespace[.25em]
$10^{-15}$   & 3.46\e{-5} secs    & 1.69\e{-14}   & 1.44\e{-3} secs    & 4.81\e{-13} \\
\addlinespace[.25em]
$10^{-18}$   & 3.45\e{-5} secs    & 3.95\e{-14}   & 1.44\e{-3} secs    & 4.84\e{-13} \\
\addlinespace[.25em]
$10^{-21}$   & 3.44\e{-5} secs    & 6.63\e{-14}   & 1.44\e{-3} secs    & 5.11\e{-13} \\
\end{tabular}
\caption[The evaluation of the modal Green's function for large wavenumber, in
double precision]{
{\bf The evaluation of the modal Green's function in double precision for a 
large wavenumber ($\kappa=10,000$)}.
The error is evaluated by using adaptive Gaussian quadrature as the gold standard.
}
\label{ta:betam1double}
\end{center}
\end{table}

%%%%%%%%%%%%%%%%%%%%%%%%%%%%%%%%%%%%%%%%%%%%%%%%%%
% Varying Beta, Quad
%%%%%%%%%%%%%%%%%%%%%%%%%%%%%%%%%%%%%%%%%%%%%%%%

\vspace{0.25in}
\begin{table}[H]
\begin{center}
\begin{tabular}{ccccc}
 & \multicolumn{2}{c}{$\kappa=10,000,\,m=10$} & \multicolumn{2}{c}{$\kappa=10,000,\,m=1000$}\\
\multirow{2}{*}{$$} 
%%%&  Function  & Absolute Error & Avg. Modal Function  & Absolute Error\\
$\beta_-$& Evaluation Time      &  Absolute Error               & Evaluation Time      & Absolute Error           \\
\midrule
\addlinespace[.5em]
$10^{15}$   & 6.62\e{-3} secs    & 1.03\e{-30}   & 2.14\e{-1} secs    & 2.11\e{-30} \\
\addlinespace[.25em]
$10^{12}$   & 6.51\e{-3} secs    & 1.73\e{-29}   & 2.10\e{-1} secs    & 2.24\e{-30} \\
\addlinespace[.25em]
$10^{9}$   & 6.51\e{-3} secs    & 1.58\e{-29}   & 2.11\e{-1} secs    & 1.44\e{-30} \\
\addlinespace[.25em]
$10^{6}$   & 6.88\e{-3} secs    & 8.94\e{-30}   & 2.10\e{-1} secs    & 1.58\e{-30} \\
\addlinespace[.25em]
$10^{3}$   & 6.88\e{-3} secs    & 1.78\e{-29}   & 2.11\e{-1} secs    & 1.99\e{-30} \\
\addlinespace[.25em]
$10^{0}$   & 5.82\e{-3} secs    & 2.89\e{-32}   & 2.11\e{-1} secs    & 6.59\e{-31} \\
\addlinespace[.25em]
$10^{-3}$   & 5.84\e{-3} secs    & 2.31\e{-32}   & 2.11\e{-1} secs    & 1.07\e{-30} \\
\addlinespace[.25em]
$10^{-6}$   & 5.74\e{-3} secs    & 2.06\e{-31}   & 2.11\e{-1} secs    & 1.82\e{-30} \\
\addlinespace[.25em]
$10^{-9}$   & 6.25\e{-3} secs    & 3.45\e{-33}   & 2.12\e{-1} secs    & 2.63\e{-31} \\
\addlinespace[.25em]
$10^{-12}$   & 6.27\e{-3} secs    & 2.07\e{-32}   & 2.12\e{-1} secs    & 1.55\e{-31} \\
\addlinespace[.25em]
$10^{-15}$   & 6.22\e{-3} secs    & 9.65\e{-32}   & 2.12\e{-1} secs    & 5.53\e{-31} \\
\addlinespace[.25em]
$10^{-18}$   & 6.06\e{-3} secs    & 1.58\e{-31}   & 2.11\e{-1} secs    & 5.60\e{-31} \\
\addlinespace[.25em]
$10^{-21}$   & 5.92\e{-3} secs    & 2.18\e{-31}   & 2.11\e{-1} secs    & 6.36\e{-31} \\

\end{tabular}
\caption[The evaluation of the modal Green's function for large wavenumber, in
quadruple precision]{
{\bf The evaluation of the modal Green's function in quadruple precision for a 
large wavenumber ($\kappa=10,000$)}.
The error is evaluated by using adaptive Gaussian quadrature as the gold standard.
}
\label{ta:betam1quad}
\end{center}
\end{table}

\subsection{Performance of the Algorithm with Varying $\kappa$}
\label{se:performkappa}
We examined the performance of our algorithm over a wide range of 
values for $\kappa$.
As shown in Tables \ref{ta:kappa1double}-\ref{ta:kappa2quad}, our
algorithm's performance is independent of $\kappa$.
\vspace{0.25in}
\begin{table}[H]
\begin{center}
\begin{tabular}{ccccc}
\multirow{2}{*}{$$}
%%%& Avg. Modal Function  & Absolute Error & Avg. Modal Function  & Absolute Error\\
& \multicolumn{2}{c}{$\beta_-=1\,,m=10$} & \multicolumn{2}{c}{$\beta_-=1\,,m=1000$}\\
$\kappa$ & Evaluation Time      & Absolute Error         & Evaluation Time      & Absolute Error                 \\
\midrule
\addlinespace[.5em]
$10^{-6}$   & 1.22\e{-4} secs    & 1.45\e{-13}   & 1.84\e{-3} secs    & 2.05\e{-12} \\
\addlinespace[.25em]
$10^{-3}$   & 6.26\e{-5} secs    & 1.50\e{-13}   & 1.64\e{-3} secs    & 2.05\e{-12} \\
\addlinespace[.25em]
$10^{0}$   & 6.08\e{-5} secs    & 1.61\e{-13}   & 1.66\e{-3} secs    & 2.02\e{-12} \\
\addlinespace[.25em]
$10^{1}$   & 1.36\e{-4} secs    & 2.71\e{-14}   & 2.50\e{-3} secs    & 1.83\e{-12} \\
\addlinespace[.25em]
$10^{2}$   & 1.02\e{-4} secs    & 4.94\e{-15}   & 2.43\e{-3} secs    & 2.23\e{-12} \\
\addlinespace[.25em]
$10^{3}$   & 4.56\e{-5} secs    & 1.30\e{-14}   & 1.73\e{-3} secs    & 1.51\e{-12} \\
\addlinespace[.25em]
$10^{4}$   & 3.89\e{-5} secs    & 3.34\e{-14}   & 1.69\e{-3} secs    & 1.03\e{-12} \\
\addlinespace[.25em]
$10^{5}$   & 3.94\e{-5} secs    & 2.25\e{-14}   & 1.70\e{-3} secs    & 5.05\e{-13} \\
\addlinespace[.25em]
$10^{6}$   & 4.15\e{-5} secs    & 2.75\e{-13}   & 1.75\e{-3} secs    & 3.32\e{-13} \\
\addlinespace[.25em]
$10^{7}$   & 3.78\e{-5} secs    & --   & 8.39\e{-4} secs    & -- \\
\addlinespace[.25em]
$10^{8}$   & 3.91\e{-5} secs    & --   & 8.33\e{-4} secs    & -- \\
\addlinespace[.25em]
$10^{9}$   & 4.46\e{-5} secs    & --   & 8.23\e{-4} secs    & -- \\
\addlinespace[.25em]
$10^{12}$   & 3.77\e{-5} secs    & --   & 8.14\e{-4} secs    & -- \\
\addlinespace[.25em]
$10^{15}$   & 4.46\e{-5} secs    & --   & 8.18\e{-4} secs    & -- \\
\addlinespace[.25em]
$10^{18}$   & 3.98\e{-5} secs    & --   & 8.33\e{-4} secs    & -- \\

\end{tabular}
\caption[The evaluation of the modal Green's function for varying wavenumber
when the source and target are not close, in
double precision]{
{\bf The evaluation of the modal Green's function in double precision for
varying $\kappa$ ($\beta_-=1$)}.
The error is evaluated by using adaptive Gaussian quadrature as the gold standard.
Note for $\kappa > 10^6$, the resource requirements of prior methods becomes
excessive.
}
\label{ta:kappa1double}
\end{center}
\end{table}

\vspace{0.25in}
\begin{table}[H]
\begin{center}
\begin{tabular}{ccccc}
\multirow{2}{*}{$$}
& \multicolumn{2}{c}{$\beta_-=1\,,m=10$} & \multicolumn{2}{c}{$\beta_-=1\,,m=1000$}\\
$\kappa$ & Evaluation Time      & Absolute Error  & Evaluation Time      & Absolute Error   \\
\midrule
\addlinespace[.5em]
$10^{-6}$   & 7.17\e{-3} secs    & 4.06\e{-31}   & 2.07\e{-1} secs    & 1.12\e{-30} \\
\addlinespace[.25em]
$10^{-3}$   & 6.43\e{-3} secs    & 3.99\e{-31}   & 2.07\e{-1} secs    & 1.12\e{-30} \\
\addlinespace[.25em]
$10^{0}$   & 6.92\e{-3} secs    & 3.94\e{-31}   & 2.09\e{-1} secs    & 1.20\e{-30} \\
\addlinespace[.25em]
$10^{1}$   & 7.09\e{-3} secs    & 1.34\e{-31}   & 2.11\e{-1} secs    & 9.06\e{-31} \\
\addlinespace[.25em]
$10^{2}$   & 6.93\e{-3} secs    & 7.74\e{-34}   & 2.10\e{-1} secs    & 1.32\e{-30} \\
\addlinespace[.25em]
$10^{3}$   & 6.81\e{-3} secs    & 9.49\e{-33}   & 2.11\e{-1} secs    & 1.01\e{-30} \\
\addlinespace[.25em]
$10^{4}$   & 5.79\e{-3} secs    & 2.89\e{-32}   & 2.12\e{-1} secs    & 6.59\e{-31} \\
\addlinespace[.25em]
$10^{5}$   & 5.78\e{-3} secs    & 1.59\e{-31}   & 2.12\e{-1} secs    & 6.05\e{-31} \\
\addlinespace[.25em]
$10^{6}$   & 5.76\e{-3} secs    & 2.23\e{-31}   & 2.11\e{-1} secs    & 4.69\e{-31} \\
\addlinespace[.25em]
$10^{7}$   & 5.76\e{-3} secs    & --   & 2.11\e{-1} secs    & -- \\
\addlinespace[.25em]
$10^{8}$   & 5.82\e{-3} secs    & --   & 1.52\e{-1} secs    & -- \\
\addlinespace[.25em]
$10^{9}$   & 5.72\e{-3} secs    & --   & 1.52\e{-1} secs    & -- \\
\addlinespace[.25em]
$10^{12}$   & 5.72\e{-3} secs    & --   & 1.52\e{-1} secs    & -- \\
\addlinespace[.25em]
$10^{15}$   & 5.70\e{-3} secs    & --   & 1.52\e{-1} secs    & -- \\
\addlinespace[.25em]
$10^{18}$   & 5.70\e{-3} secs    & --   & 1.52\e{-1} secs    & -- \\

\end{tabular}
\caption[The evaluation of the modal Green's function for varying wavenumber
when the source and target are not close, in
quadruple precision]{
{\bf The evaluation of the modal Green's function in quadruple precision for a
for varying $\kappa$ with large source-to-target distance  ($\beta_-=1$)}.
The error is evaluated by using adaptive Gaussian quadrature as the gold standard.
Note for $\kappa > 10^6$, the resource requirements of prior methods becomes
excessive.
}
\label{ta:kappa1quad}
\end{center}
\end{table}

\vspace{0.25in}
\begin{table}[H]
\begin{center}
\begin{tabular}{ccccc}
\multirow{2}{*}{$$}

& \multicolumn{2}{c}{$\beta_-=10^{-12}\,,m=10$} & \multicolumn{2}{c}{$\beta_-=10^{-12}\,,m=1000$}\\
$\kappa$ & Evaluation Time      & Absolute Error  & Evaluation Time      & Absolute Error   \\
\midrule
\addlinespace[.5em]
$10^{-6}$   & 4.49\e{-5} secs    & 3.08\e{-13}   & 1.34\e{-3} secs    & 2.90\e{-11} \\
\addlinespace[.25em]
$10^{-3}$   & 4.45\e{-5} secs    & 2.90\e{-13}   & 1.37\e{-3} secs    & 2.88\e{-11} \\
\addlinespace[.25em]
$10^{0}$   & 4.77\e{-5} secs    & 1.90\e{-13}   & 1.40\e{-3} secs    & 2.84\e{-11} \\
\addlinespace[.25em]
$10^{1}$   & 4.79\e{-5} secs    & 4.35\e{-14}   & 1.41\e{-3} secs    & 2.74\e{-11} \\
\addlinespace[.25em]
$10^{2}$   & 4.61\e{-5} secs    & 1.80\e{-14}   & 1.43\e{-3} secs    & 2.29\e{-11} \\
\addlinespace[.25em]
$10^{3}$   & 3.57\e{-5} secs    & 1.07\e{-14}   & 1.44\e{-3} secs    & 4.19\e{-12} \\
\addlinespace[.25em]
$10^{4}$   & 3.49\e{-5} secs    & 3.33\e{-14}   & 1.43\e{-3} secs    & 5.28\e{-13} \\
\addlinespace[.25em]
$10^{5}$   & 3.46\e{-5} secs    & 1.50\e{-13}   & 1.44\e{-3} secs    & 6.58\e{-13} \\
\addlinespace[.25em]
$10^{6}$   & 3.41\e{-5} secs    & 5.11\e{-13}   & 1.45\e{-3} secs    & 3.04\e{-13} \\
\addlinespace[.25em]
$10^{7}$   & 3.48\e{-5} secs    & --   & 7.88\e{-4} secs    & -- \\
\addlinespace[.25em]
$10^{8}$   & 3.43\e{-5} secs    & --   & 7.87\e{-4} secs    & -- \\
\addlinespace[.25em]
$10^{9}$   & 3.38\e{-5} secs    & --   & 7.87\e{-4} secs    & -- \\
\addlinespace[.25em]
$10^{12}$   & 3.22\e{-5} secs    & --   & 7.85\e{-4} secs    & -- \\
\addlinespace[.25em]
$10^{15}$   & 3.38\e{-5} secs    & --   & 7.87\e{-4} secs    & -- \\
\addlinespace[.25em]
$10^{18}$   & 3.33\e{-5} secs    & --   & 7.85\e{-4} secs    & -- \\

\end{tabular}
\caption[The evaluation of the modal Green's function for varying wavenumber
when the source and target are close, in
double precision]{
{\bf The evaluation of the modal Green's function in double precision for
varying $\kappa$ with small source-to-target distance
($\beta_-=10^{-12}$)}.
The error is evaluated by using adaptive Gaussian quadrature as the gold standard.
Note for $\kappa > 10^6$, the resource requirements of prior methods becomes
excessive.
}
\label{ta:kappa2double}
\end{center}
\end{table}

\vspace{0.25in}
\begin{table}[H]
\begin{center}
\begin{tabular}{ccccc}
\multirow{2}{*}{$$}
& \multicolumn{2}{c}{$\beta_-=10^{-12}\,,m=10$} & \multicolumn{2}{c}{$\beta_-=10^{-12}\,,m=1000$}\\
$\kappa$ & Evaluation Time      & Absolute Error  & Evaluation Time      & Absolute Error   \\
\midrule
\addlinespace[.5em]
$10^{-6}$   & 7.23\e{-3} secs    & 7.06\e{-31}   & 2.08\e{-1} secs    & 5.75\e{-29} \\
\addlinespace[.25em]
$10^{-3}$   & 7.53\e{-3} secs    & 7.12\e{-31}   & 2.08\e{-1} secs    & 5.75\e{-29} \\
\addlinespace[.25em]
$10^{0}$   & 7.16\e{-3} secs    & 6.54\e{-31}   & 2.10\e{-1} secs    & 5.70\e{-29} \\
\addlinespace[.25em]
$10^{1}$   & 7.53\e{-3} secs    & 8.03\e{-32}   & 2.11\e{-1} secs    & 5.56\e{-29} \\
\addlinespace[.25em]
$10^{2}$   & 7.09\e{-3} secs    & 5.00\e{-32}   & 2.11\e{-1} secs    & 4.55\e{-29} \\
\addlinespace[.25em]
$10^{3}$   & 6.25\e{-3} secs    & 6.41\e{-32}   & 2.12\e{-1} secs    & 7.93\e{-30} \\
\addlinespace[.25em]
$10^{4}$   & 6.23\e{-3} secs    & 2.07\e{-32}   & 2.11\e{-1} secs    & 1.55\e{-31} \\
\addlinespace[.25em]
$10^{5}$   & 6.21\e{-3} secs    & 1.43\e{-31}   & 2.12\e{-1} secs    & 2.21\e{-31} \\
\addlinespace[.25em]
$10^{6}$   & 6.21\e{-3} secs    & 2.37\e{-31}   & 2.12\e{-1} secs    & 3.88\e{-31} \\
\addlinespace[.25em]
$10^{7}$   & 6.25\e{-3} secs    & --   & 1.52\e{-1} secs    & -- \\
\addlinespace[.25em]
$10^{8}$   & 6.22\e{-3} secs    & --   & 1.53\e{-1} secs    & -- \\
\addlinespace[.25em]
$10^{9}$   & 6.19\e{-3} secs    & --   & 1.52\e{-1} secs    & -- \\
\addlinespace[.25em]
$10^{12}$   & 5.78\e{-3} secs    & --   & 1.52\e{-1} secs    & -- \\
\addlinespace[.25em]
$10^{15}$   & 5.72\e{-3} secs    & --   & 1.52\e{-1} secs    & -- \\
\addlinespace[.25em]
$10^{18}$   & 5.69\e{-3} secs    & --   & 1.52\e{-1} secs    & -- \\

\end{tabular}
\caption[The evaluation of the modal Green's function for varying wavenumber
when the source and target are close, in
quadruple precision]{
{\bf The evaluation of the modal Green's function in quadruple precision for
varying $\kappa$ with small source-to-target distance
($\beta_-=10^{-12}$)}.
The error is evaluated by using adaptive Gaussian quadrature as the gold standard.
Note for $\kappa > 10^6$, the resource requirements of prior methods becomes
excessive.
}
\label{ta:kappa2quad}
\end{center}
\end{table}
\subsection{Performance of the Algorithm with Varying Fourier Mode ($m$)}
\label{se:performfourier}
We examined the performance of our algorithm over a wide range of Fourier
modes (represented by the parameter $m$).
Because the number of points in the quadrature scales linearly with $m$,
as demonstrated by Table \ref{ta:perform_m}, evaluation time 
scales linearly with the Fourier mode. Recall from the introduction
of this section that the evaluation was performed on a 
four-core processor using two threads.
\vspace{0.25in}
\begin{table}[H]
\begin{center}
\begin{tabular}{ccccc}
\multirow{2}{*}{$$}
% derf % derf1 \\
$m$& \multicolumn{2}{c}{Evaluation Time$$} & \multicolumn{2}{c}{$$}\\
\midrule
\addlinespace[.5em]
1   & 3.88\e{-5} secs \\ 
\addlinespace[.25em]
10   & 5.56\e{-5} secs \\ 
\addlinespace[.25em]
$10^{2}$   & 1.75\e{-4} secs \\ 
\addlinespace[.25em]
$10^{3}$   & 1.46\e{-3} secs \\ 
\addlinespace[.25em]
$10^{4}$   & 1.43\e{-2} secs \\ 
\addlinespace[.25em]
$10^{5}$   & 1.37\e{-1} secs \\ 
\addlinespace[.25em]
$10^{6}$   & 1.36\e{0} secs \\ 
\addlinespace[.25em]
$10^{7}$   & 1.29\e{1} secs \\ 
 
\end{tabular}
\caption[The evaluation time of the modal Green's function in double precision
for varying Fourier mode]{
{\bf The evaluation time of the modal Green's function in double precision for
varying $m$ 
($\beta_-=10^{-12}, \kappa = 10,000$)}.
}
\label{ta:perform_m}
\end{center}
\end{table}
\subsection{Parallelization of the Algorithm}
\label{se:parallel}
The cost of our algorithm is $O(m)$ and does not depend on $\kappa$ or
$\beta_-$ (see Section \ref{se:sumalg}).
Because our algorithm is quadrature based, 
it is embarrassingly
parallelizable.

We measured the algorithm's performance on a server with a 16-core Intel
Xeon 2.9 GHz processor, where each core can run two threads for a 
total of 32-threads. 
We vary the number of threads from 1 to 32, and report the results in 
Figure \ref{fig:evalTime}.
\begin{figure}[H]
\centering
\includegraphics[width=0.75\textwidth]{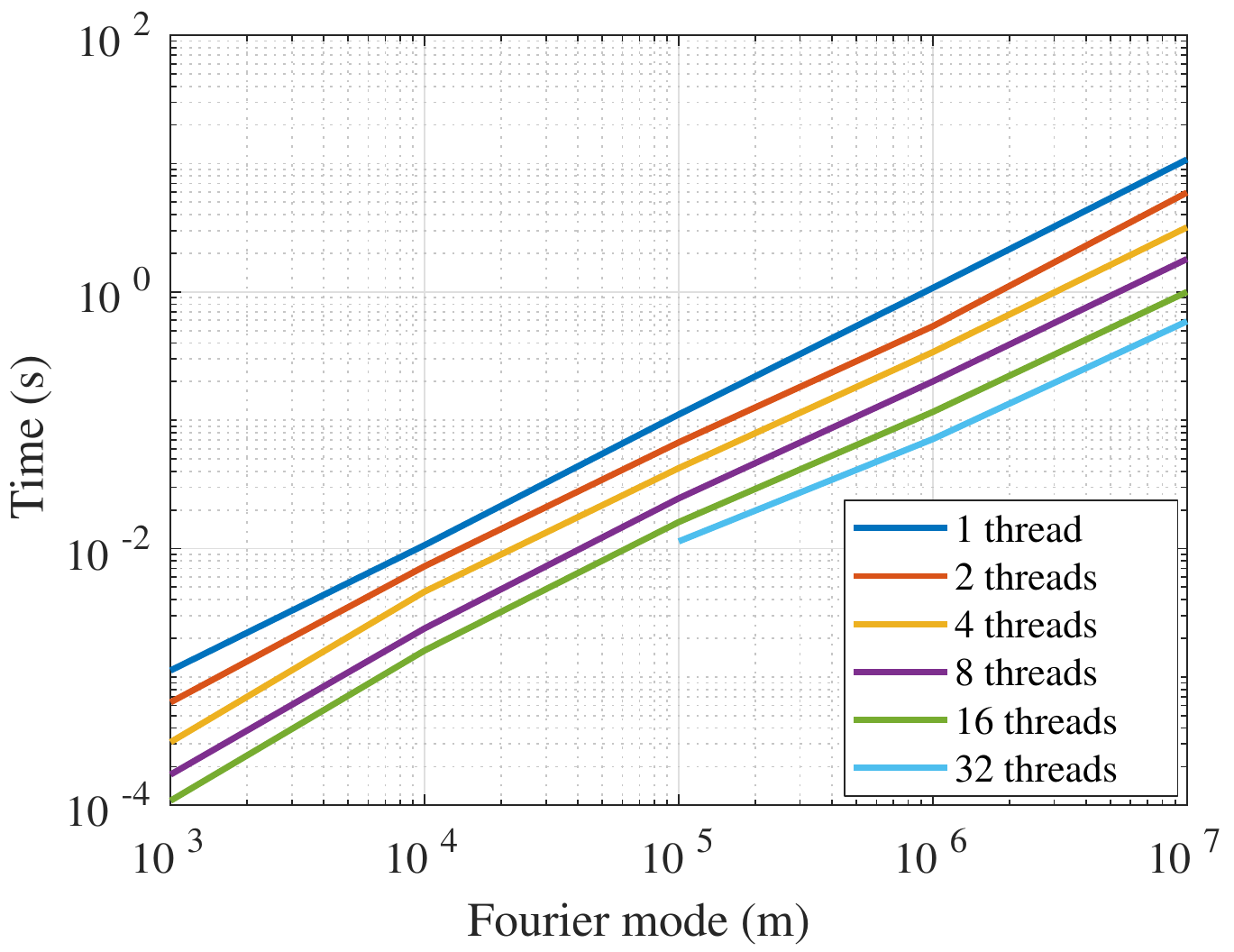}
\caption[Evaluation time of the modal Green's function with varying numbers of threads]{
{\bf Evaluation time of the modal Green's function plotted
against $m$ with varying numbers of threads 
($\beta_-=10^{-7}, \kappa = 10,000$)}. The calculation is performed in double precision.
The evaluation times corresponding to 32 threads are not plotted for small $m$.
}
\label{fig:evalTime}
\end{figure}
\section{Conclusions and Generalizations}
We have developed an algorithm which evaluates the modal Green's function
for the Helmholtz equation in $O(m)$ time, that is completely independent
of both the wavenumber and the source-to-target distance. Furthermore, our 
algorithm is embarrassingly parallelizable.
Our algorithm's method can be readily extended to several associated problems
in computational electromagnetics, described in Sections \ref{se:O1Small}-
\ref{se:O1amortize}.
\subsection{An $O(1)$ Evaluator for Small Wavenumber ($\kappa \ll m$)}
\label{se:O1Small}
Recall that our algorithm is independent of the wavenumber because 
we integrate along Gustafsson's contours, which are the steepest
descent contours with respect to the spherical wave component (see 
Section \ref{se:steepest}).
When the Fourier mode $m$ is larger than the scaled wavenumber $\kappa$, it is more
efficient to integrate along a different contour. If instead, we
choose the steepest descent 
contour on which $\exp{(im\phi)}$ does not oscillate,
we arrive at an alternative algorithm whose cost is
$O(\kappa)$ and independent of $m$ . When $\kappa$ is extremely small, this algorithm is 
essentially
$O(1)$. The case where $\beta_-$ 
is small (i.e., when the source and target are close) is handled in an identical
fashion to the method described in Section \ref{se:evalalpha}.
Thus, this alternative algorithm's cost 
is completely independent of
both $m$ and $\beta_-$, and grows as $O(\kappa)$.
\subsection{An $O(1)$ Evaluator of the Modal Green's Functions for the Laplace Equation}
\label{se:O1Laplace}
The same method described in Section \ref{se:O1Small} can be applied to the case
where $\kappa=0$ to yield an $O(1)$ evaluator of the modal Green's function
for the Laplace equation, whose cost is independent of $\beta_-$ (i.e.,
the cost is independent of the source-to-target distance).
\subsection{Extension of the Algorithm to Complex $\kappa$}
In this manuscript, we assumed $\kappa \in \mathbb{R}$ and $\kappa >0$,
where $\kappa$ is the scaled wavenumber.
When the scaled wavenumber $\kappa$ is complex (i.e., when the medium is
attenuating), 
Gustafsson's steepest descent contours are rotated in the complex plane.
The same algorithm described in this manuscript applies in this 
case, with the only modification being a change in the geometry
of the steepest descent contours and the locations of the intersection
points of the contours with the Bernstein ellipse.
\subsection{Extension to an $O(m)$ Evaluator for a Collection of
Modal Green's Functions, with Amortized Cost $O(1)$}
\label{se:O1amortize}
This manuscript presents an algorithm for the evaluation of a single
modal Green's function for the Helmholtz Equation in $O(m)$ time, 
independent of $\beta_-$ and $\kappa$,
where $\beta_-$ is the scaled minimum source-to-target distance and $\kappa$
is the scaled wavenumber.
It is possible to use this algorithm to compute all of the modal Green's functions
$-M,-M+1,\dots,M-1,M$ in $O(M)$ time using the following method.
In \cite{matviyenko}, Matviyenko presents a five-term recurrence 
relation for the modal Green's functions for the Helmholtz equation.
He observes that the recurrence relation is stable upwards for one range
of Fourier modes and stable downwards for another range of modes.
Furthermore, there exists a range of modes for which the recurrence
is bi-unstable. Thus, a classical Miller-type algorithm cannot be applied.
However, it was recently observed in \cite{osipov} that if a recurrence
relation is represented as a banded matrix, then the inverse power method
can be used to find a solution, even when the stability behavior is mixed
in the sense just described.
We thus apply the inverse power method, as described in \cite{osipov}, to
the resulting five-diagonal matrix corresponding to Matviyenko's recurrence
relation.
In this fashion, we obtain all the eigenvectors corresponding to the
zero eigenvalue; only one vector in this eigenspace corresponds to
the vector of modal Green's functions. We thus
use the $O(m)$ evaluator of this manuscript to select the vector 
corresponding to the modal Green's functions.
The cost of performing the inverse-power method is $O(M)$, and the cost
of the evaluation
of the $M$th modal Green's function is $O(M)$, meaning that all $M$ Fourier
coefficients are obtained in $O(M)$ time.

\end{document}